\documentclass[11pt]{article}
\usepackage{amsmath,amssymb,amsthm,amsfonts,afterpage}
\usepackage[pdftex]{graphicx}
\usepackage[usenames]{xcolor}
\usepackage{wrapfig}
\usepackage[sort]{cite}
\usepackage{subfig}

\usepackage[normalem]{ulem} 
\usepackage{cancel} 
\usepackage[textsize=tiny,color=blue!50!white]{todonotes} 
\usepackage[breaklinks,bookmarks=false]{hyperref}

\hypersetup{colorlinks, linkcolor=blue, citecolor=blue,
urlcolor=blue, plainpages=false, pdfwindowui=false,
pdfstartview={FitH}}

\numberwithin{equation}{section}

\newfont{\Fr}{eufm10 scaled\magstep1}
\newfont{\Bb}{msbm10 scaled\magstep1}

\def\bkC{\mbox{\Bb \symbol{67}}}

\def\u{{\bf u}}
\def\vv{{\bf v}} 
\def\V{{\bf V}}

\def\n{{\bf n}}

\def\grad{\nabla}
\def\div{\grad \cdot}
\def\divh{\grad_h \cdot}
\def\O{{\Omega}}

\def\d{\partial}
\def\dt{\tilde\partial_t}
\def\<{\langle}
\def\>{\rangle}

\def\S{{\bf S}}
\def\r{{\bf r}}
\def\R{\mathbb{R}}

\def\bzeta{\boldsymbol{\zeta}}

\def\bpsi{\boldsymbol{\psi}}

\def\bPi{\boldsymbol{\Pi}}

\def\H{{\bf H}}
\def\L{{\bf L}}
\def\M{{\bf M}}

\def\T{{\cal T}}
\def\P{{\cal P}}
\def\Q{{\cal Q}}
\def\cO{{\cal O}}
\def\I{{\cal I}}
\def\Dt{{\Delta t}}
\def\DT{{\Delta T}}
\def\DG{{\rm DG}}

\newcommand\ie{i.e.}
\newcommand\cf{cf.}
\newcommand\eg{e.g.}

\newtheorem{theorem}{Theorem}[section]
\newtheorem{remark}[theorem]{Remark}
\newtheorem{lemma}[theorem]{Lemma}

\title{A space-time multiscale mortar mixed finite element method for
  parabolic equations\thanks{This project has received partial
    funding from the US National Science Foundation grants DMS 1818775 and DMS 2111129 and the European Research Council (ERC) under the
      European Union's Horizon 2020 research and innovation program
      (grant agreement No 647134 GATIPOR).}}

\author{
  Manu Jayadharan\footnotemark[2]
\and
Michel Kern\footnotemark[3] \footnotemark[4]
\and
Martin Vohral\'ik\footnotemark[3] \footnotemark[4]
\and
Ivan Yotov\footnotemark[2]}
\renewcommand{\thefootnote}{\fnsymbol{footnote}}

\date{\today}

\begin{document}
\footnotetext[2]{Department of Mathematics, University of Pittsburgh, Pittsburgh, PA 15260, USA
  (\href{mailto:manu.jayadharan@pitt.edu}{\texttt{manu.jayadharan@pitt.edu}},
\href{mailto:yotov@math.pitt.edu}{\texttt{yotov@math.pitt.edu}}).}
\footnotetext[3]{Inria, 2 rue Simone Iff, 75589 Paris, France
  (\href{mailto:michel.kern@inria.fr}{\texttt{michel.kern@inria.fr}},
  \href{mailto:martin.vohralik@inria.fr}{\texttt{martin.vohralik@inria.fr}}).}
\footnotetext[4]{CERMICS, Ecole des Ponts, 77455 Marne-la-Vallée, France.}

\renewcommand{\thefootnote}{\arabic{footnote}}
\maketitle

\begin{abstract}
We develop a space-time mortar mixed finite element method for parabolic problems. The domain is decomposed into a union of subdomains discretized with non-matching spatial grids and asynchronous time steps. The method is based on a space-time variational formulation that couples mixed finite elements in space with discontinuous Galerkin in time. Continuity of flux (mass conservation) across space-time interfaces is imposed via a coarse-scale space-time mortar variable that approximates the primary variable. Uniqueness, existence, and stability, as well as a priori error estimates for the spatial and temporal errors are established. A space-time non-overlapping domain
decomposition method is developed that reduces the global problem to a
space-time coarse-scale mortar interface problem. Each interface
iteration involves solving in parallel space-time subdomain problems. The spectral properties of the interface operator and the convergence of the interface iteration are analyzed. Numerical experiments are provided that illustrate the theoretical results and the flexibility of the method for modeling problems with features that are localized in space and time.
\end{abstract}  


\section{Introduction}

The multiscale mortar mixed finite element method of~\cite{ACWY,APWY} allows for highly efficient and accurate discretization of elliptic problems. Let a spatial domain $\Omega$
be given, which is decomposed into subdomains $\O_i$. Then, on each $\O_i$, an individual mesh is set, and
a standard mixed finite element scheme is considered. A stand-alone mortar variable approximating the primary variable is further
introduced on an independent interface mesh, which is typically coarser but where one possibly
employs polynomials of higher degree. It is used to couple the subdomain problems and to
ensure (multiscale) weak continuity of the normal component of the mixed finite element flux
variable over the interfaces between subdomains. Moreover, the mortar variable enables a very efficient parallelization in space via a non-overlapping domain decomposition algorithm based on reduction to an interface problem
\cite{ACWY,APWY,GANIS20093989}.

We introduce here a {\em space-time discretization} of a model parabolic equation, extending the
above philosophy to time-dependent problems. Let a time interval $(0,T)$ be given. For each
subdomain $\O_i$, our approach considers an individual space mesh of $\O_i$ along with individual
time stepping on $(0,T)$. On each {\em space-time subdomain} $\O_i \times (0,T)$, any standard mixed
finite element scheme is combined with the discontinuous Galerkin (DG) time discretization \cite{Thomee}. Then a
stand-alone mortar variable approximating the primary variable is introduced on an independent {\em space-time interface mesh},
which is typically coarse and where higher polynomial degrees may be used. It
is used to couple the space-time subdomain problems and to ensure (multiscale) weak continuity of
the normal component of the mixed finite element flux variable (and consequently mass conservation) over the space-time interfaces.
This setting allows for high flexibility with individual discretizations of each space-time
subdomain $\Omega_i \times (0,T)$, and in particular for {\em local time stepping}. Moreover, {\em space-time parallelization}
  can be achieved via reduction to a space-time interface problem requiring
the solution of discrete problems on the individual space-time subdomains $\Omega_i \times (0,T)$, exchanging space-time boundary data through transmission conditions, in the spirit of space-time domain decomposition methods as in, \eg,
\cite{hoang2013space,hoang-space-time-fracture,GanderKwokMandal,FAUCHER,Yu-space-time,Gander50year,Benes,NakNakTor,HalJapSze,GanHal}.

We mention some of the related previous works on local time stepping for parabolic problems in mixed formulations. The early work \cite{EwiLazVas90} studies finite difference methods on grids with local refinement in space and time. A similar approach is employed in \cite{ADM-LTS} for transport equations. Two space-time domain decomposition methods are considered in \cite{hoang2013space} -- a space-time Steklov--Poincar\'e  operator and optimized Schwarz waveform relaxation (OSWR) \cite{Gander50year} with Robin transmission conditions. Asynchronous time stepping is allowed, but the spatial grids are assumed matching. The focus is on the analysis of the iterative convergence of the OSWR method. A posteriori error estimates for these methods are developed in \cite{Arbogast-aposteriori,Hassan-aposteriori} for nested time grids, with \cite{Arbogast-aposteriori} also allowing for non-matching spatial grids through the use of mortar finite elements. The methods from \cite{hoang2013space} a
 re extended to fracture modeling in \cite{hoang-space-time-fracture}. Overlapping Schwarz domain decomposition with local grid refinement in space and time for two-phase flow in porous media is developed in \cite{Kheriji-near-well}. Domain decomposition methods for mortar mixed finite element methods for parabolic problems with non-matching spatial grids and uniform time stepping are studied in \cite{GaiGloMas,ArshadParkShin}. For parabolic problems in non-mixed form, domain decomposition methods with local time stepping have been studied, e.g., in \cite{Benes,NakNakTor,FAUCHER,Yu-space-time,HalJapSze,GanHal,Krause,DawDuDup,Hager}.
Parallelism in time has also been explored, such as the Parareal
  algorithm~\cite{parareal, GanVan} and multigrid in
  time~\cite{Falgout,GanNeum}. Local time stepping techniques have been developed for multiphysics systems coupled through interface conditions, e.g., for the Stokes--Darcy system \cite{hoang2021global,RybMag}. Finally, we mention some earlier works on space-time methods for parabolic problems coupling mixed finite element discretizations in space with DG in time on a single domain. In \cite{bause2017error}, a method using continuous trial and discontinuous test functions in time is developed. A posteriori error estimation and space-time adaptivity for mixed finite element -- DG methods is studied in \cite{Cascon,KimParkSeo}.

  To the best of the authors' knowledge, the solvability, stability, and a priori error analysis for space-time domain decomposition methods with non-matching spatial grids and asynchronous time stepping have not been studied in the literature, which is the main goal of this paper. A key tool in the analysis is the construction of an interpolant in a space-time weakly continuous velocity space, which is used to prove a discrete divergence inf-sup condition on this space. Another key component is establishing a discrete space-time mortar inf-sup condition under a suitable assumption on the mortar space. In addition to performing complete analysis in the general case, we also consider conforming time discretizations. In this case, we provide stability and error bounds for the velocity divergence and improved error estimates. To the best of our knowledge, such result has not been established in the literature for space-time mixed finite element methods with a DG time discretization, eve
 n on a single domain. Finally, we develop a parallel non-overlapping domain decomposition algorithm for the solution of the resulting algebraic problem. In particular, we utilize a time-dependent Steklov--Poincar\'e operator approach to reduce the global problem to a space-time interface problem. We show that the interface operator is positive definite and analyze its spectral properties. We employ an interface GMRES algorithm, which involves solving in parallel space-time subdomain problems at each iteration. The current iteration value of the mortar variable provides a Dirichlet boundary data on the space-time interfaces for the subdomain problems. We emphasize that, due to the discontinuous time discretization, the space-time subdomain problems are solved using classical time marching over the local time grid. We utilize the spectral bound of the interface operator to obtain an estimate for the number of interface GMRES iterations through field-of-values analysis. 
  
This contribution is organized as follows. In Section~\ref{sec_setting}, we describe the model
problem and its domain decomposition weak formulation. Our space-time multiscale mortar discretization is introduced in
Section~\ref{space-time-mortar}, and we prove its existence, uniqueness, and stability with
respect to data in Section~\ref{sec_WP}. Section~\ref{sec_a_priori} derives a priori error estimates. Control of the velocity divergence and improved error estimates are established in Section~\ref{sec:div}. The reduction to a space-time interface problem and its analysis are presented in Section~\ref{sec_red_int}. We finally present numerical illustrations in Section~\ref{sec_num} and close with conclusions in Section~\ref{sec_concl}.

\section{Setting} \label{sec_setting}

In this section we introduce the setting for our study.

\subsection{Model problem}

We consider a parabolic partial differential equation in a mixed form,
modeling single phase flow in porous media. Let $\Omega \subset \R^d$,
$d = 2,3$, be a spatial polytopal domain with Lipschitz boundary
and let $T > 0$ be the final time. The governing equations are
\begin{subequations}\label{parabolic}\begin{equation}\label{PDE}
\u = - K \grad p, \quad \frac{\d p}{\d t} + \div \u = q \quad \mbox{in } \Omega \times (0,T],
\end{equation}
where $p$ is the fluid pressure, $\u$ is the Darcy velocity,
$q$ is a source term, and $K$ is a tensor representing the rock
permeability divided by the fluid viscosity. We assume for simplicity
homogeneous Dirichlet boundary condition
\begin{equation}\label{BC}
    p(x,t) = 0 \text{ on } \d \Omega \times (0,T]
\end{equation}
and assign the initial pressure
\begin{equation}\label{IC}
    p(x,0) = p_0(x) \text{ on } \Omega.
\end{equation}\end{subequations}
We assume that $q \in L^2(0,T;L^2(\O))$, $p_0 \in H^1_0(\O)$, $\div K \nabla p_0 \in L^2(\Omega)$,  and that
$K$ is a spatially-dependent, uniformly bounded, symmetric, and
positive definite tensor, \ie, for constants $0 < k_{\min} \le k_{\max}
< \infty$,
\begin{equation}\label{k-spd}
 \text{ for a.e. } x \in \O, \quad k_{\min}\zeta^T \zeta \le \zeta^T K(x) \zeta \le k_{\max}\zeta^T \zeta
  \quad \forall \zeta \in \R^d, \ d = 2,3.
\end{equation}
Moreover, we suppose a scaling such that the diameter of $\O$ and the final time $T$ are of order one.

\subsection{Space-time subdomains}

Let $\Omega$ be a union of non-overlapping
polytopal subdomains with Lipschitz boundary, $\overline\Omega = \cup \overline \Omega_i$. Let
$\Gamma_i = \d\Omega_i\setminus\d\Omega$ be the interior boundary of $\Omega_i$,
let $\Gamma_{ij} = \Gamma_i \cap \Gamma_j$ be the interface between two adjacent
subdomains $\O_i$ and $\O_j$, and let $\Gamma = \cup \Gamma_{ij}$ be
the union of all subdomain interfaces.
We also introduce the space-time counterparts
$\O^T = \O\times (0,T)$, $\O_i^T = \O_i\times (0,T)$, $\Gamma_i^T = \Gamma_i\times (0,T)$,
and $\Gamma_{ij}^T = \Gamma_{ij}\times (0,T)$.
We will introduce space-time domain decomposition discretizations based on $\O_i^T$.

\subsection{Basic notation}

We will utilize the following notation. For a domain $\cO \subset \R^d$, the
$L^2(\cO)$ inner product and norm for scalar and vector-valued functions
are denoted by $(\cdot,\cdot)_{\cO}$ and $\|\cdot\|_{\cO}$,
respectively. The norms and seminorms of the Sobolev spaces
$W^{k,p}(\cO),\, k \in \R, p \ge 1$, are denoted by $\| \cdot \|_{k,p,\cO}$ and
$| \cdot |_{k,p,\cO}$, respectively. The norms and seminorms of the
Hilbert spaces $H^k(\cO)$ are denoted by $\|\cdot\|_{k,\cO}$ and $| \cdot
|_{k,\cO}$, respectively. For
a section of a subdomain boundary $S \subset \R^{d-1}$ we
write $\<\cdot,\cdot\>_S$ and $\|\cdot\|_S$ for the $L^2(S)$ inner
product (or duality pairing) and norm, respectively. By $\M$ we denote the
vectorial counterpart of a generic scalar space $M$.

The above notation is extended
to space-time domains as follows. For $\cO^T = \cO\times(0,T)$ and $S^T = S\times(0,T)$,
let $(\cdot,\cdot)_{\cO^T} = \int_0^T (\cdot,\cdot)_{\cO}$ and
$\<\cdot,\cdot\>_{S^T} = \int_0^T \<\cdot,\cdot\>_S$. For space-time norms we use the standard
Bochner notation. For example, given a spatial norm $\|\cdot\|_V$, we denote, for $p > 0$,
$$
\|\cdot\|_{L^p(0,T;V)} = \left(\int_0^T \|\cdot\|_V^p\right)^{\frac1p}, \quad
\|\cdot\|_{L^\infty(0,T;V)} = \text{ess}\sup \|\cdot\|_V,
$$
with the usual extension for $\|\cdot\|_{W^{k,p}(0,T;V)}$ and $\|\cdot\|_{H^k(0,T;V)}$.
Let $\|\cdot\|_{S^T} = \|\cdot\|_{L^2(0,T;L^2(S))}$. We will use the space
$$
\H({\rm div}; \cO) = \left\{\vv \in \L^2(\cO): \div \vv \in L^2(\cO) \right\},
$$
equipped with the norm
$$
\|\vv\|_{{\rm div};\cO} = \left( \|\vv\|_{\cO}^2 + \|\div \vv\|_{\cO}^2 \right)^{\frac12}.
$$
Finally, throughout the paper, $C$ will denote a generic constant that is independent of the spatial and temporal discretization parameters.

\subsection{Weak formulation}

The weak formulation of problem~\eqref{parabolic} reads: find $(\u, p): [0,T]
\mapsto \H({\rm div}; \O) \times L^2(\O)$
such that $p(x,0) = p_0$ and for a.e. $t \in (0,T)$,
\begin{subequations}\label{weak-omega}\begin{eqnarray}
  &&(K^{-1}\u,\vv)_{\Omega} - (p,\div \vv)_{\Omega} = 0 \quad \forall \, \vv \in \H({\rm div}; \O),
\label{weak-omega.1} \\
&&(\d_t p,w)_{\Omega} + (\div \u, w)_{\Omega} = (q,w)_\Omega
\quad \forall \, w \in L^2(\O). \label{weak-omega.2}
\end{eqnarray}\end{subequations}
The following well-posedness result is rather standard and presented in, \eg,
  \cite[Theorem~2.1]{hoang2013space}.
\begin{theorem}[Well-posedness]\label{cont-existence}
    Problem~\eqref{weak-omega} has
  a unique solution  $\u \in L^2(0,T;\H({\rm div};\O)) \cap L^{\infty}(0,T;\L^2(\Omega))$,
$p \in H^1(0,T;H_0^1(\O))$.
\end{theorem}
We note that in particular the inclusion $p \in
H^1(0,T;H_0^1(\O))$ follows from~\eqref{weak-omega.1},
which implies that for a.e. $t \in (0,T)$, $\grad p = -K^{-1}\u$ in the sense of distributions.

\subsection{Domain decomposition weak formulation}

We now give a domain decomposition weak formulation of~\eqref{weak-omega}.
Introduce the subdomain velocity and pressure spaces
$$
\V_{i} = \H({\rm div}; \O_{i}), \,\, \V = \bigoplus \V_i,  \quad
W_{i} = L^2(\O_{i}), \,\, W = \bigoplus W_i = L^2(\O),
$$
endowed with the norms
$$
\|\vv\|_{\V_i} = \|\vv\|_{{\rm div}; \O_i}, \quad \|\vv\|_{\V}
= \left( \sum_i \|\vv\|_{\V_i}^2 \right)^{\frac12},
\quad \|w\|_W = \|w\|_{\Omega}.
$$
We also introduce the following spatial
bilinear forms, which will prove useful below:
\begin{subequations}\label{forms}\begin{align}
    & a_i(\u,\vv) = (K^{-1}\u,\vv)_{\Omega_i},
    \quad a(\u,\vv) = \sum_i a_i(\u,\vv), \label{forms.a}\\
  & b_i(\vv,w) = - (\div \vv, w)_{\Omega_i}, \quad b(\vv,w) = \sum_i b_i(\vv,w), \label{forms.b}\\
  & b_{\Gamma}(\vv,\mu) = \sum_i \<\vv\cdot\n_i, \mu\>_{\Gamma_i}. \label{forms.bG}
\end{align}\end{subequations}
In addition, for any spatial bilinear form $s(\cdot,\cdot)$, let
$\displaystyle s^T(\cdot,\cdot) = \int_0^T s(\cdot,\cdot)$.

Now, since $p \in H^1(0,T;H_0^1(\O))$, we can consider the trace of
the pressure $p$ on the interfaces, $\lambda = p|_\Gamma$. Thus,
integrating in time, it is easy to see that the solution $(\u,p)$
of~\eqref{weak-omega} satisfies
\begin{subequations}\label{weak}\begin{eqnarray}
&&a^T(\u,\vv) + b^T(\vv,p) + b^T_{\Gamma}(\vv,\lambda) = 0 \quad \forall \, \vv \in L^2(0,T;\V), \label{weak.1} \\
&&(\d_t p,w)_{\Omega^T} - b^T(\u,w) = (q,w)_{\Omega^T}
\quad \forall \, w \in L^2(0,T;W). \label{weak.2}
\end{eqnarray}\end{subequations}

\section{Space-time mortar mixed finite element method}\label{space-time-mortar}

We consider a space-time discretization of~\eqref{weak}, motivated
by~\cite{hoang2013space}.  It employs a mortar finite element variable to
approximate the pressure trace $\lambda$ from~\eqref{weak} and uses it
as a Lagrange multiplier to impose weakly the continuity of flux
across space-time interfaces.

\subsection{Space-time grids and spaces}\label{sec_grids_spaces}

Let $\T_{h,i}$ be a shape-regular partition of the subdomain
$\Omega_i$ into parallelepipeds or simplices in the sense
  of~\cite{ciarlet}. We stress that this allows for grids that
  do not match along the interfaces $\Gamma_{ij}$ between subdomains $\Omega_i$ and $\Omega_j$. Let $h_i = \max_{E \in \T_{h,i}} \text{diam}\,E$ and $h = \max_i h_i$.
  In some parts of the analysis, we will additionally require $\T_{h,i}$ to be quasi-uniform in that $h_i \le C \text{diam}\,E \ \forall E \in \T_{h,i}$, as well as that the mesh sizes in all subdomains be comparable in that $h \le C h_i \ \forall i$. Similarly, let $\T^\Dt_i: 0 =
t_i^0 < t_i^1 < \cdots < t_i^{N_i} = T$ be a partition of the time
interval $(0,T)$ corresponding to subdomain $\Omega_i$. This means
that we consider different time discretizations on different
subdomains. Let $\Dt_i = \max_{1\le k \le N_i} |t_i^k - t_i^{k-1}|$ and
$\Dt = \max_i \Dt_i$. Though we admit non-uniform time stepping, we will sometimes require that $\T^\Dt_i$ be quasi-uniform in that
$\Dt_i \le C |t_i^k - t_i^{k-1}| \ \forall (t_i^{k-1},t_i^k) \in \T^\Dt_i$, as well as that the time steps in all subdomains be comparable in that
$\Dt \le C \Dt_i \ \forall i$.
\begin{figure}
\centering
\includegraphics[width=3in]{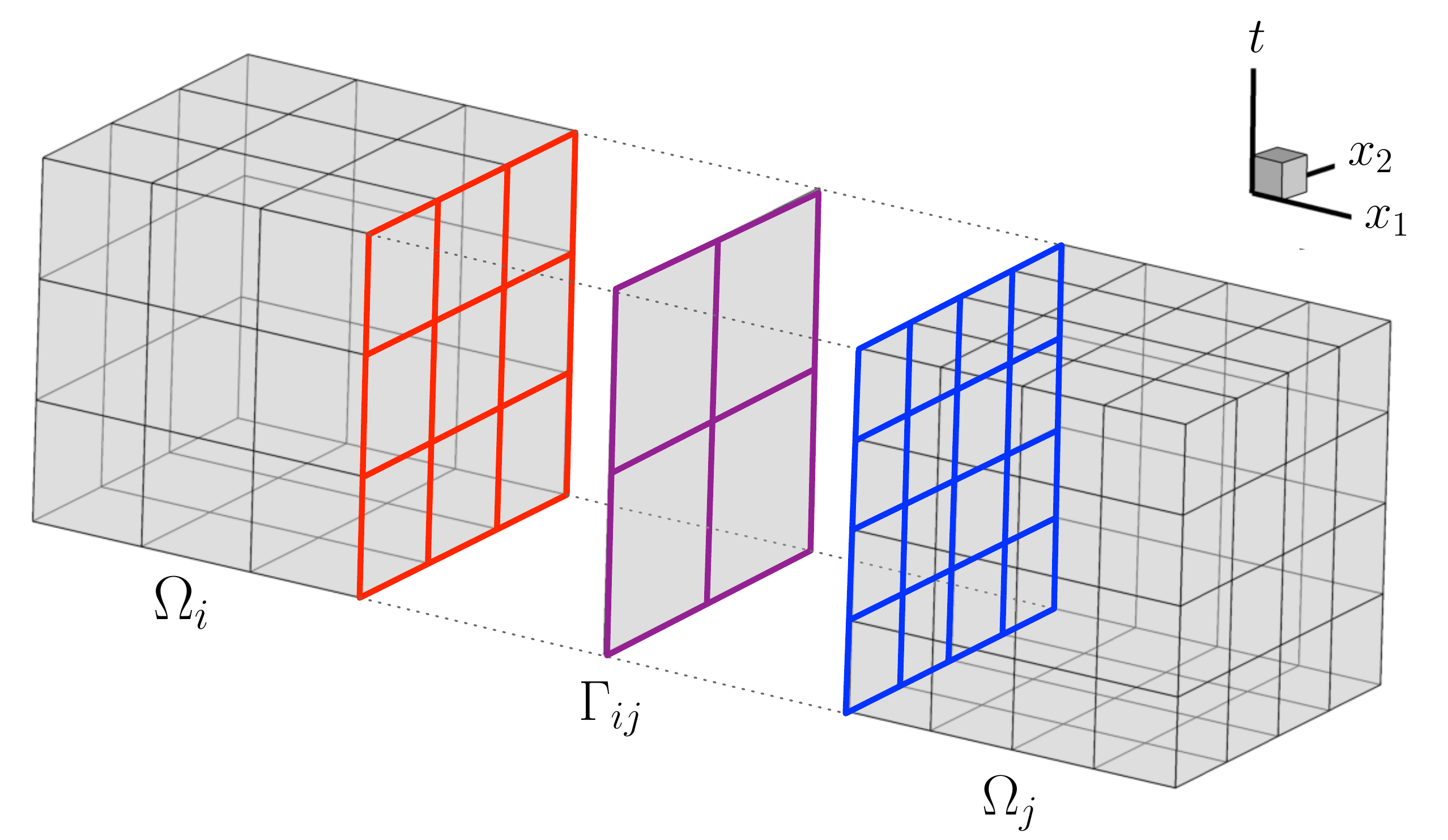}
\caption{Non-matching space-time subdomain and mortar grids in two spatial dimensions.}
\label{fig:spacetime-mortar}
\end{figure}
Composing $\T_{h,i}$ and $\T^\Dt_i$
by tensor product results in a space-time partition
\[
\T^\Dt_{h,i} = \T_{h,i} \times \T^\Dt_i
\]
of the space-time subdomain $\Omega_i^T$. An illustration is given in
Figure~\ref{fig:spacetime-mortar}, where yet a different, mortar space-time grid,
is also shown in the middle.

For discretization in space, we consider any of the inf--sup stable mixed finite element
spaces $\V_{h,i} \times W_{h,i} \subset \V_i \times W_i$ such as the
Raviart--Thomas or the Brezzi--Douglas--Marini spaces, see, \eg,
\cite{BF}. For discretization in time, we will in turn utilize the discontinuous
Galerkin (DG) method, \cf~\cite{Thomee}, which is based on a discontinuous
piecewise polynomial approximation of the solution on the mesh
$\T^\Dt_i$. Denote by $W_i^\Dt$ the subdomain time discretizations of the velocity and pressure. 
Composing the space and time discretizations
\[
\V_{h,i}^\Dt = \V_{h,i} \times W_i^\Dt, \quad W_{h,i}^\Dt = W_{h,i} \times W_i^\Dt
\]
results in the space-time mixed finite element spaces $\V_{h,i}^\Dt \times W_{h,i}^\Dt$ in
each space-time subdomain $\Omega_i^T$. We will also need the spatial variable only spaces
\[
    \V_h = \bigoplus \V_{h,i}, \quad W_h = \bigoplus W_{h,i}.
\]

Let $\T_{H,ij}$ be a shape-regular finite element partition of $\Gamma_{ij}$, where
$H = \max_{i,j} \max_{e \in \T_{H,ij}} \text{diam}\, e$, see Figure~\ref{fig:spacetime-mortar},
middle. The use of index $H$ indicates a possibly coarser interface grid compared to the
subdomain grids, resulting in a multiscale approximation. Let
$\T^\DT_{ij}: 0 = t_{ij}^0 < t_{ij}^1 < \cdots <
t_{ij}^{N_{ij}} = T$ be a partition of $(0,T)$ corresponding to
$\Gamma_{ij}$, which may be different from (and again possibly coarser than)
the time-partitions for the neighboring subdomains.
Let $\Delta T = \max_{i,j} \max_{1 \le k \le N_{ij}}|t_{ij}^k - t_{ij}^{k-1}|$.
Composing $\T_{H,ij}$ and $\T^\DT_{ij}$ by tensor product gives a space-time
partition
\[
    \T^\DT_{H,ij} = \T_{H,ij} \times \T^\DT_{ij}
\]
of the space-time interface $\Gamma_{ij}^T$. Finally, let
\[
    \Lambda_{H,ij}^\DT = \Lambda_{H,ij} \times \Lambda_{ij}^\DT
\]
be a space-time mortar finite element
space on $\T^\DT_{H,ij}$ consisting of continuous or
discontinuous piecewise polynomials in space and in time. We will also need the spatial variable only space
\[
    \Lambda_H = \bigoplus \Lambda_{H,ij}.
\]

Finally, the global space-time finite element spaces are defined as
\begin{equation}\label{spaces}
\V_h^\Dt = \bigoplus \V_{h,i}^\Dt, \quad W_h^\Dt =
\bigoplus W_{h,i}^\Dt, \quad \Lambda_H^\DT = \bigoplus
\Lambda_{H,ij}^\DT.
\end{equation}
In particular, the Lagrange multiplier will be sought for in the mortar space $\Lambda_H^\DT$. For the purpose of the analysis, we also define the space of velocities with space-time weakly continuous normal components
\begin{equation}\label{weak-cont}
  \V_{h,0}^\Dt = \left\{\vv \in \V_{h}^\Dt: b^T_{\Gamma}(\vv,\mu) = 0 \quad \forall \, \mu \in
  \Lambda_H^\DT\right\}.
\end{equation}
The discrete velocity and pressure spaces inherit
the norms $\|\cdot\|_{\V}$ and $\|\cdot\|_W$, respectively. The mortar space
is equipped with the spatial norm $\|\mu\|_{\Lambda_H} =
\|\mu\|_{L^2(\Gamma)}$.

\subsection{Space-time multiscale mortar mixed finite element method}\label{sec_MMFE}

For the DG time discretization, we introduce the notation for $\varphi(x,\cdot)$, $\phi(x,\cdot) \in W_i^\Dt$, $x \in \Omega_i$, see~\cite{Thomee},
\begin{equation}\label{dt-DG}
\int_0^T \dt \varphi \, \phi  =
\sum_{k=1}^{N_i} \int_{t_i^{k-1}}^{t_i^k} \d_t \varphi \, \phi +
\sum_{k=1}^{N_i} [\varphi]_{k-1} \, \phi_{k-1}^+,
\end{equation}
where
$[\varphi]_k = \varphi_k^+ - \varphi_k^-$, with $\varphi_k^+ = \lim_{t \to t_i^{k,+}}
\,\varphi$ and $\varphi_k^- = \lim_{t \to t_i^{k,-}} \,\varphi$.

\begin{remark}[Initial value]
In what follows, we will tacitly assume that a function $\varphi(x,\cdot) \in
W_i^\Dt$ has an associated initial value $\varphi_0^-$, which will be
defined if it is explicitly used.
\end{remark}

The space-time multiscale mortar mixed
finite element method for approximating~\eqref{weak}
is: find $\u_h^\Dt \in \V_h^\Dt$, $p_h^\Dt \in
W_h^\Dt$, and $\lambda_H^\DT \in \Lambda_{H}^\DT$
such that
\begin{subequations}\label{method}\begin{align}
  & a^T(\u_h^\Dt,\vv) + b^T(\vv,p_h^\Dt) + b^T_{\Gamma}(\vv,\lambda_H^\DT) = 0
  \quad \forall \, \vv \in \V_h^\Dt, \label{method.1} \\
  & (\dt p_h^\Dt,w)_{\Omega^T} - b^T(\u_h^\Dt,w) = (q,w)_{\Omega^T}
  \quad \forall \, w \in W_h^\Dt, \label{method.2} \\
  & b^T_{\Gamma}(\u_h^\Dt,\mu) = 0 \quad \forall \, \mu \in \Lambda_H^\DT, \label{method.3}
\end{align}\end{subequations}
where the obvious notation $(\dt p_h^\Dt,w)_{\Omega^T} = \sum_i (\dt p_h^\Dt,w)_{\Omega_i^T}$ has been used. We note that
$(\dt p_h^\Dt,w)_{\Omega^T}$ involves the term
$\big((p_{h}^\Dt)_0^+ - (p_{h}^\Dt)_0^-,w_0^+\big)_{\Omega}$, see the last term in
\eqref{dt-DG} for $k=1$. Here,
$(p_{h}^\Dt)_0^+ $ is computed by the method, while $(p_{h}^\Dt)_0^-$
is determined by the initial condition.
We discuss the construction of initial data in Section~\ref{sec:init}.

The above method provides
a highly general and flexible framework, allowing for different
spatial and temporal discretizations in different subdomains.  We note
that according to~\eqref{method.3}, continuity of the flux is imposed
weakly on the space-time interfaces $\Gamma_{ij}^T$, requiring that
the jump in flux is orthogonal to the space-time mortar space
$\Lambda_{H,ij}^\DT$. This formulation results in a correct
notion of mass conservation across interfaces for time-dependent
domain decomposition problems with non-matching grids in both space
and time. In the case of discontinuous mortars, \eqref{method.3}
implies that the total flux across any space-time interface cell $e \times (t_{ij}^{k-1},t_{ij}^{k})$, $e \in \T_{H,ij}$, is
continuous.

\section{Well-posedness analysis}\label{sec_WP}

In this section we analyze the existence, uniqueness, and stability of
the solution to~\eqref{method}.

\subsection{Space-time interpolants}\label{interpolants}

We will make use of several space-time interpolants. Let $\P_{h,i}$ be the $L^2$-orthogonal projection onto $W_{h,i}$ and let $\P_i^\Dt$ be the $L^2$-orthogonal projection onto $W_i^\Dt$. We then define the $L^2$-orthogonal projection in space and time on subdomain $\Omega_i$ by
\[
    \P_{h,i}^\Dt = \P_{h,i} \times \P_i^\Dt: L^2(0,T;L^2(\Omega_i)) \to
    W_{h,i}^\Dt
\]
and globally by
\[
    \P_h^\Dt: L^2(0,T;L^2(\Omega)) \to W_{h}^\Dt, \quad \P_h^\Dt|_{\Omega_i} = \P_{h,i}^\Dt.
\]
Setting $\P_h|_{\Omega_i} = \P_{h,i}$ and $\P^\Dt|_{\Omega_i} = \P_i^\Dt$, we will also write $\P_h^\Dt = \P_h \times \P^\Dt$. Since $\div \V_{h,i} = W_{h,i}$, we have, for all $\varphi \in L^2(0,T;L^2(\Omega_i))$,
\begin{equation}\label{div-orth}
  (\P_h^\Dt \varphi - \varphi,\div \vv)_{\O_i^T} = 0 \quad \forall \, \vv \in \V_{h,i}^\Dt.
\end{equation}
For $\epsilon > 0$, denote $\H^\epsilon({\rm div};\Omega_i):= \H^\epsilon(\Omega_i) \cap \H({\rm div};\Omega_i)$. Let $\bPi_{h,i}: \H^\epsilon({\rm div};\Omega_i) \to \V_{h,i}$ be the canonical mixed interpolant \cite{BF} and let
\[
    \bPi_{h,i}^\Dt = \bPi_{h,i} \times \P_i^\Dt: L^2(0,T;\H^\epsilon({\rm div};\Omega_i)) \to \V_{h,i}^\Dt.
\]
In particular, this space-time interpolant satisfies,
for all $\bpsi \in L^2(0,T;\H^\epsilon({\rm div};\Omega_i))$,
\begin{subequations}\begin{align}
  & (\div (\bPi_{h,i}^\Dt \bpsi - \bpsi),w)_{\Omega_i^T}
  = 0 \quad \forall \, w \in W_{h,i}^\Dt, \label{pi-div} \\
  & \< (\bPi_{h,i}^\Dt \bpsi
  - \bpsi)\cdot\n_i,\vv\cdot\n_i\>_{\d\Omega_i^T} = 0 \quad \forall \, \vv \in \V_{h,i}^\Dt, \label{pi-normal} \\
  & \|\bPi_{h,i}^\Dt \bpsi\|_{L^2(0,T;\V_{i})} \le C
  (\|\bpsi\|_{L^2(0,T; \H^\epsilon(\Omega_i))}
  + \|\div \bpsi\|_{L^2(0,T;L^2(\Omega_i))}). \label{pi-cont}
\end{align}\end{subequations}
Let $\Q_{h,i}: L^2(\d\Omega_i) \to \V_{h,i}\cdot\n_i$ be the $L^2$-orthogonal projection and let
\begin{equation}\label{eq_Qhidt}
    \Q_{h,i}^\Dt = \Q_{h,i} \times \P_i^\Dt: L^2(0,T;L^2(\d\Omega_i)) \to \V_{h,i}^\Dt\cdot\n_i.
\end{equation}
Finally, let $\P_{H,\Gamma_{ij}}: L^2(\Gamma_{ij}) \to \Lambda_{H,ij}$ and $\P_{ij}^\DT: L^2(0,T) \to \Lambda_{ij}^\DT$
be the $L^2$-orthogonal projections and let
\begin{equation}\label{eq_PHG}
    \P_{H,\Gamma_{ij}}^\DT = \P_{H,\Gamma_{ij}} \times \P_{ij}^\DT: L^2(0,T; L^2(\Gamma_{ij})) \to \Lambda_{H,ij}^\DT, \quad \P_{H,\Gamma}^\DT|_{\Gamma_{ij}} = \P_{H,\Gamma_{ij}}^\DT
\end{equation}
be the mortar space-time $L^2$-orthogonal projection.

\subsection{Assumptions on the mortar grids}

We make the
following assumptions on the mortar grids, which are needed to guarantee
that the method~\eqref{method} is well posed:
there exists a positive constant $C$ independent of the spatial mesh sizes $h$ and $H$ (as well as of the temporal mesh sizes $\Dt$ and $\Delta T$) such that
\begin{subequations}\label{mortar-ass}\begin{align}
  & \forall \, \mu \in \Lambda_H, \ \forall \, i,j, \quad
  \|\mu\|_{\Gamma_{ij}}
  \le C (\|\Q_{h,i} \, \mu\|_{\Gamma_{ij}} + \|\Q_{h,j} \, \mu\|_{\Gamma_{ij}}),
  \label{mortar-assumption-space} \\
  & \forall \, i,j, \quad \Lambda_{ij}^\DT \subset W_i^\Dt \cap W_j^\Dt.
    \label{mortar-assumption-time}
  \end{align}\end{subequations}

The spatial mortar assumption~\eqref{mortar-assumption-space} is the same as
the assumption made in~\cite{ACWY,APWY}. Note that it is in particular satisfied with $C = \frac12$ when $\T_{H,ij}$ is a coarsening of both $\T_{h,i}$ and $\T_{h,j}$ on the interface $\Gamma_{ij}$ and the space $\Lambda_{H,ij}$ consists of discontinuous piecewise polynomials contained in $\V_{h,i} \cdot\n_i$ and $\V_{h,j} \cdot\n_j$ on $\Gamma_{ij}$. In general, it requires that the mortar space $\Lambda_H$ is
sufficiently coarse, so that it is controlled by the normal traces of the neighboring subdomain velocity spaces.

The temporal mortar assumption~\eqref{mortar-assumption-time}
similarly provides control of the mortar time discretization by the subdomain time
discretizations. It requires that each subdomain time
discretization be a refinement of the mortar time discretization. We also note that
\eqref{mortar-assumption-space} and~\eqref{mortar-assumption-time} imply
\begin{equation}\label{mortar-assumption}
  \forall \, \mu \in \Lambda_H^\DT, \forall \, i,j, \quad
  \|\mu\|_{L^2(0,T;L^2(\Gamma_{ij}))}
  \le C (\|\Q_{h,i}^\Dt \, \mu\|_{L^2(0,T;L^2(\Gamma_{ij}))} +
  \|\Q_{h,j}^\Dt \, \mu\|_{L^2(0,T;L^2(\Gamma_{ij}))})
\end{equation}
for a constant $C$ independent of $h$, $H$, $\Dt$, and $\Delta T$.

\subsection{Discrete inf--sup conditions}

Recall the form $b^T(\cdot,\cdot)$ from~\eqref{forms.b}. Under the above assumptions on the mortar grids, the weakly continuous velocity space $\V_{h,0}^\Dt$ of~\eqref{weak-cont} satisfies the following inf--sup condition.

\begin{lemma}[Discrete divergence inf--sup condition on $\V_{h,0}^\Dt$]\label{lem_inf--sup-weak-cont}
Let~\eqref{mortar-ass} hold. Then there exists a constant $\beta > 0$, independent of $h$, $H$, $\Dt$, and $\Delta T$, such that
\begin{equation}\label{inf--sup-weak-cont}
\forall \, w \in W_h^\Dt, \quad \sup_{0 \ne \vv \in \V_{h,0}^\Dt}
\frac{b^T(\vv,w)}
{\|\vv\|_{L^2(0,T;\V)}} \ge \beta \|w\|_{L^2(0,T;L^2(\O))}.
\end{equation}
\end{lemma}
\begin{proof}
  Let $\V_{h,0} = \left\{\vv \in \V_{h}: b_\Gamma(\vv,\mu) = 0 \quad
\forall \, \mu \in \Lambda_H\right\}$.
It is shown in~\cite{ACWY,APWY} that if~\eqref{mortar-assumption-space} holds, then
there is an interpolant
$\bPi_{h,0}: \H^{\frac12+\epsilon}({\rm div};\Omega) \to \V_{h,0}$ such that,
for all $\bpsi \in \H^{\frac12+\epsilon}({\rm div};\Omega)$,
\begin{subequations}\begin{align}
  & \sum_i (\div (\bPi_{h,0} \bpsi - \bpsi),w)_{\Omega_i}
  = 0 \quad \forall \, w \in W_{h}, \label{pi-h0-div} \\
  & \|\bPi_{h,0} \bpsi\|_{\V} \le C (\|\bpsi\|_{\H^{\frac12+\epsilon}(\Omega)}
  + \|\div \bpsi\|_{L^2(\Omega)}), \label{pi-h0-cont}
\end{align}\end{subequations}
for a constant $C$ independent of $h$ and $H$. Define
\[
    \bPi_{h,0}^\Dt = \bPi_{h,0} \times \P^\Dt.
\]
We claim that $\bPi_{h,0}^\Dt: L^2(0,T;\H^{\frac12+\epsilon}({\rm div};\Omega)) \to \V_{h,0}^\Dt$. To see this, note first that, for all functions
$\bpsi \in L^2(0,T;\H^{\frac12+\epsilon}({\rm div};\Omega))$, clearly $\bPi_{h,0}^\Dt\bpsi \in \V_h^\Dt$. Thus~\eqref{mortar-assumption-time} implies
\begin{align*}
b^T_{\Gamma}(\bPi_{h,0}^\Dt\bpsi,\mu) & =
\sum_i \int_0^T\<\bPi_{h,0}^\Dt\bpsi \cdot\n_i, \mu\>_{\Gamma_i} =
\sum_i \int_0^T\<\bPi_{h,0}\bpsi \cdot\n_i, \mu\>_{\Gamma_i} \\
& = \int_0^T b_\Gamma(\bPi_{h,0}\bpsi,\mu) = 0
\quad \forall \, \mu \in \Lambda_H^\DT,
\end{align*}
i.e., indeed $\bPi_{h,0}^\Dt\bpsi \in \V_{h,0}^\Dt$ by virtue of~\eqref{weak-cont}.
Moreover, \eqref{pi-h0-div} and~\eqref{pi-h0-cont} imply
\begin{subequations}\begin{align}
& \sum_i (\div (\bPi_{h,0}^\Dt \bpsi - \bpsi),w)_{\Omega_i^T}
  = 0 \quad \forall \, w \in W_{h}^\Dt, \label{pi-0-div} \\
& \|\bPi_{h,0}^\Dt \bpsi\|_{L^2(0,T;\V)} \le C
  (\|\bpsi\|_{L^2(0,T; \H^{\frac12+\epsilon}(\Omega))}
  + \|\div \bpsi\|_{L^2(0,T;L^2(\Omega))}). \label{pi-0-cont}
\end{align}\end{subequations}
The inf--sup condition~\eqref{inf--sup-weak-cont} then follows from the classical
continuous inf--sup condition for $b^T(\cdot,\cdot)$, the existence of the interpolant
$\bPi_{h,0}^\Dt$, and Fortin's lemma~\cite{BF}.
\end{proof}

To control the mortar variable, we need the following mortar inf--sup condition.

\begin{lemma}[Discrete mortar inf--sup condition on $\V_{h}^\Dt$]\label{lem:inf--sup} Let~\eqref{mortar-assumption} hold. Then there exists a constant $\beta_\Gamma > 0$, independent of $h$, $H$, $\Dt$, and $\Delta T$, such that
\begin{equation}\label{inf--sup}
\forall \, \mu \in \Lambda_H^\DT, \quad \sup_{0 \ne \vv \in \V_h^\Dt}
\frac{b^T_\Gamma(\vv,\mu)}
{\|\vv\|_{L^2(0,T;\V)}} \ge \beta_\Gamma \|\mu\|_{L^2(0,T;L^2(\Gamma))}.
\end{equation}
\end{lemma}
\begin{proof}
  Let $\mu \in \Lambda_H^\DT$ be given. In the
  following we assume that $\mu$ is extended by zero on $\d\Omega$.
  We consider a set of auxiliary subdomain problems.
Let $\varphi_i(x,t)$ be the solution for a.e.
$t \in (0,T)$ of the problem
  \begin{subequations}\label{aux-problem}\begin{align}
    & \div \grad \varphi_i (\cdot,t) =  \overline{(\Q_{h,i}^\Dt \,  \mu) (\cdot,t)} \quad \mbox{ in } \Omega_i,
    \label{aux-problem-1}\\
    & \grad \varphi_i (\cdot,t) \cdot \n_i = (\Q_{h,i}^\Dt \, \mu) (\cdot,t) \quad \mbox{ on }
    \d\Omega_i, \label{aux-problem-2}
  \end{align}\end{subequations}
where $\overline{\Q_{h,i}^\Dt \, \mu}$ denotes the mean value
of $\Q_{h,i}^\Dt \, \mu$ on $\d\Omega_i$.  Let $\bpsi_i = \grad \varphi_i$.  Elliptic regularity~\cite{lions2011non,grisvard2011elliptic} implies that for a.e. $t \in
(0,T)$,
\begin{equation}\label{ell-reg}
    \|\bpsi_i\|_{\frac12,\Omega_i} + \|\div \bpsi_i\|_{\Omega_i}
    \le C \|\Q_{h,i}^\Dt \, \mu\|_{\d\Omega_i}.
\end{equation}
Let $\vv_i = \bPi_{h,i}^\Dt \bpsi_i \in \V_{h,i}^\Dt$. Note that~\eqref{pi-normal} together with~\eqref{eq_Qhidt} and~\eqref{aux-problem-2} imply that $\vv_i \cdot \n_i = \Q_{h,i}^\Dt \, \mu$ on $\d\Omega_i$. Thus,
using definition~\eqref{forms.bG} of $b^T_\Gamma$, the fact that $\mu$ is extended by zero on $\d\Omega_i \setminus \Gamma_i$, and definition~\eqref{eq_Qhidt} of the projection $\Q_{h,i}^\Dt$, we have
\begin{equation}\label{inf--sup-1}\begin{split}
  b^T_\Gamma(\vv,\mu) & =
   \sum_i \<\bPi_{h,i}^\Dt \bpsi_i\cdot\n_i, \mu\>_{\Gamma_i^T}
   = \sum_i \<\bPi_{h,i}^\Dt \bpsi_i\cdot\n_i, \mu\>_{\d\Omega_i^T}
   = \sum_i \<\bPi_{h,i}^\Dt \bpsi_i\cdot\n_i, \Q_{h,i}^\Dt \, \mu\>_{\d\Omega_i^T} \\
   & = \sum_i\|\Q_{h,i}^\Dt \, \mu\|_{L^2(0,T;L^2(\d\Omega_i))}^2
   \ge C  \sum_i \|\mu\|_{L^2(0,T;L^2(\Gamma_i))}^2,
\end{split}\end{equation}
  where we used~\eqref{mortar-assumption} in the inequality. On the other hand,
  \eqref{pi-cont} with $\epsilon = \frac12$ and~\eqref{ell-reg}, along with
  the stability of $L^2$-orthogonal projection $Q_{h,i}^\Dt$, imply
\begin{equation}\label{inf--sup-2}
    \|\vv_i\|_{L^2(0,T;\V_{i})}
    \le C \|\mu\|_{L^2(0,T;L^2(\Gamma_i))}.
\end{equation}
The assertion of the lemma follows from combining~\eqref{inf--sup-1} and~\eqref{inf--sup-2}.
\end{proof}

\subsection{Initial data}\label{sec:init}
We next discuss the construction of discrete initial data for all variables. The data need to be compatible in the sense that they satisfy the equations without time derivatives in the method, \eqref{method.1} and \eqref{method.3}. Recall that we are given initial pressure datum $p(0) = p_0 \in H^1_0(\O)$ with $\div K \nabla p_0 \in L^2(\Omega)$. Let us define $\u_0 = -K \nabla p_0$ and
$\lambda_0 = p_0|_{\Gamma}$. Then the solution to \eqref{weak} satisfies
$\u(0) = \u_0$ and $\lambda(0) = \lambda_0$. Moreover, we have
\begin{align*}
  & a(\u_0,\vv) + b(\vv,p_0) + b_\Gamma(\vv,\lambda_0) = 0
  \quad \forall \, \vv \in \V_h, \\
  & b_\Gamma(\u_0,\mu) = 0 \quad \quad \forall \, \mu \in \Lambda_H.
\end{align*}
Next, define the discrete initial data $(\u_{h,0},p_{h,0},\lambda_{H,0}) \in
\V_h \times W_h \times \Lambda_H$ as the elliptic projection of
$(\u_{0},p_{0},\lambda_{0})$, i.e., the unique solution to the problem
\begin{subequations}\label{init}\begin{align}
    & a(\u_{h,0},\vv) + b(\vv,p_{h,0}) + b_{\Gamma}(\vv,\lambda_{H,0})
= a(\u_{0},\vv) + b(\vv,p_{0}) + b_{\Gamma}(\vv,\lambda_{0})
    = 0
  \quad \forall \, \vv \in \V_h, \label{init.1} \\
  & b(\u_{h,0},w) = b(\u_{0},w) = -(\div K \nabla p_0,w)
  \quad \forall \, w \in W_h, \label{init.2} \\
  & b_{\Gamma}(\u_{h,0},\mu) = b_{\Gamma}(\u_{0},\mu) = 0 \quad \forall \, \mu \in \Lambda_H. \label{init.3}
\end{align}\end{subequations}
The well-posedness of \eqref{init} is shown in \cite{ACWY,APWY} under the spatial mortar assumption \eqref{mortar-assumption-space}. In particular, it follows from the analysis in \cite{ACWY,APWY} that
\begin{align}
  & \|\u_{h,0}\|_{\V} + \|p_{h,0}\|_W + \|\lambda_{H,0}\|_{\Lambda_H}
  \le C \|\div K \nabla p_0\|_{\Omega}, \label{init-stab} \\
  & \|\u_0 - \u_{h,0}\|_{\V} + \|p_0 - p_{h,0}\|_W
  + \|\lambda_0 - \lambda_{H,0}\|_{\Lambda_H} \nonumber \\
  & \qquad \le C(\|\u_0 - \bPi_{h,0}\u_0\|_{\V}
  + \|p_0 - \P_h p_{0}\|_W + \|\lambda_0 - \P_{H,\Gamma}\lambda_{0}\|_{\Lambda_H}).
  \label{init-error}
\end{align}
We now set
\begin{equation}\label{init-data}
  (p_{h}^\Dt)_0^- = p_{h,0}, \quad (\u_{h}^\Dt)_0^- = \u_{h,0}, \quad
  (\lambda_{H}^\DT)_0^- = \lambda_{h,0}.
\end{equation}
As we noted earlier, $(p_{h}^\Dt)_0^-$ provides initial condition for the method \eqref{method}, cf. \eqref{method.2}. The data $(\u_{h}^\Dt)_0^-$ and $(\lambda_{H}^\DT)_0^-$ are not needed in the method, but it will be utilized in the analysis of $\div \u_{h}^\Dt$ in Section~\ref{sec:div}.

\subsection{Existence, uniqueness, and stability with respect to data}

In the analysis we will utilize the following auxiliary result.

\begin{lemma}[Summation in time]\label{lem:dt-positive}
For all $\Omega_i$ and for any $\phi(x,\cdot) \in W_i^\Dt$, $x \in \Omega_i$, there holds
\begin{equation}\label{dt-positive}
\int_0^T \dt \varphi \, \varphi = \frac12 \left((\varphi_{N_i}^-)^2
  - (\varphi_0^-)^2 \right)
+ \frac12\sum_{k=1}^{N_i}([\varphi]_{k-1})^2.
\end{equation}
\end{lemma}
\begin{proof}
Using the definition~\eqref{dt-DG} of $\dt w$, we have
\begin{align*}
\int_0^T \dt \varphi \, \varphi & = \sum_{k=1}^{N_i} \int_{t_i^{k-1}}^{t_i^k}
\frac12\frac{\d}{\d t} \varphi^2
+ \sum_{k=1}^{N_i} [\varphi]_{k-1} \varphi_{k-1}^+ \\
& = \frac12 \sum_{k=1}^{N_i}\left((\varphi_k^-)^2 - (\varphi_{k-1}^+)^2
+ (\varphi_{k-1}^+)^2 - (\varphi_{k-1}^-)^2
+ (\varphi_{k-1}^+ - \varphi_{k-1}^-)^2 \right) \\
& = \frac12 \left((\varphi_{N_i}^-)^2
  - (\varphi_0^-)^2 \right)
+ \frac12\sum_{k=1}^{N_i}(\varphi_{k-1}^+ - \varphi_{k-1}^-)^2.
\end{align*}
\end{proof}

To simplify the presentation, we introduce the notation
\begin{equation}\label{eq_DG}
\|\varphi\|_{\DG}^2 = \sum_i \big(\|\varphi_{N_i}^-\|_{\Omega_i}^2
  + \sum_{k=1}^{N_i}\|[\varphi]_{k-1}\|_{\Omega_i}^2 \big).
\end{equation}

\begin{theorem}[Existence and uniqueness of the discrete solution, stability with respect to data]\label{well-posed}
Assume that conditions~\eqref{mortar-ass} hold. Then
the space-time mortar method~\eqref{method} has a
unique solution. Moreover, for some constant $C > 0$ independent of $h$, $H$, $\Dt$, and $\Delta T$,
 \begin{align}
  \|p_h^\Dt\|_\DG
  + \|\u_h^\Dt\|_{\Omega^T} + \|p_h^\Dt\|_{\Omega^T}
  + \|\lambda_H^\DT\|_{\Gamma^T}
  \le C(\|q\|_{\Omega^T} + \|\div K \nabla p_{0}\|_{\O}). \label{stability}
 \end{align}
\end{theorem}
\begin{proof}
We begin with establishing the stability bound~\eqref{stability}. Taking $\vv = \u_h^\Dt$, $w = p_h^\Dt$, and $\mu = \lambda_H^\DT$ in~\eqref{method} and combining the equations, we obtain, using~\eqref{dt-positive} and Young's inequality,
\begin{align*}
    \frac12 \sum_i \Big(\|(p_h^\Dt)_{N_i}^- \|_{\Omega_i}^2
  + \sum_{k=1}^{N_i}\|[p_h^\Dt]_{k-1}\|_{\Omega_i}^2 \Big)
  + \|K^{-\frac12}\u_h^\Dt\|_{\Omega^T}^2
  \le \frac{\epsilon}{2} \|p_h^\Dt\|_{\Omega^T}^2
  + \frac{1}{2\epsilon} \|q\|_{\Omega^T}^2 + \frac12\|p_{h,0}\|_\Omega^2.
\end{align*}
The inf--sup condition for the weakly continuous velocity~\eqref{inf--sup-weak-cont} and~\eqref{method.1} imply
  \begin{equation*}
\|p_h^\Dt\|_{\Omega^T} \le C \|K^{-\frac12}\u_h^\Dt\|_{\Omega^T}.
  \end{equation*}
Furthermore, the mortar inf--sup condition~\eqref{inf--sup} and~\eqref{method.1}
imply
  \begin{equation*}
    \|\lambda_H^\DT\|_{\Gamma^T} \le C(\|K^{-\frac12}\u_h^\Dt\|_{\Omega^T}
    + \|p_h^\Dt\|_{\Omega^T}).
  \end{equation*}
Combining the above three inequalities, taking $\epsilon$ sufficiently small, and using~\eqref{k-spd} and \eqref{init-stab}, we obtain~\eqref{stability}.
 The existence and uniqueness of a
 solution follows from~\eqref{stability} by taking $q = 0$ and $p_0 = 0$.
\end{proof}

\begin{remark}[Control of divergence]
Control on $\|\div \u_h^\Dt\|_{L^2(0,T;L^2(\O_i))}$ can be established under the assumption of matching time steps between subdomains and choosing the mortar finite element space in time to match the subdomains. We present this result later in Section~\ref{sec:div}, along with improved error estimates.
\end{remark}

\section{A priori error analysis}\label{sec_a_priori}

In this section we derive a priori error estimates for the solution of
the space-time mortar MFE method~\eqref{method}.

\subsection{Approximation properties of the space-time interpolants}

Assume that the spaces $\V_h^\Dt$ and $W_h^\Dt$ from~\eqref{spaces}
contain on each space-time element polynomials in $P_k$ and $P_l$, respectively, in space and $P_q$ in time, where $P_r$ denotes the space of polynomials of degree up to $r$.
Let $\Lambda_H^\DT$ contain on each space-time mortar element polynomials in $P_m$
in space and $P_s$ in time. We have
the following approximation properties for the space-time interpolants $\P_h^\Dt$ and $\P_{H,\Gamma}^\DT$ of Section~\ref{interpolants} and $\bPi_{h,0}^\Dt$ of the proof of Lemma~\ref{lem_inf--sup-weak-cont}:
\begin{subequations}\label{approx}\begin{align}
  & \|\bpsi - \bPi_{h,0}^\Dt\bpsi\|_{\Omega^T} \le C
  \sum_i \|\bpsi\|_{H^{r_q}(0,T;\H^{r_k}(\Omega_i))}(h^{r_k} + \Dt^{r_q})
  + C \|\bpsi\|_{H^{r_q}(0,T;\H^{\tilde r_k+\frac12}(\Omega))}(h^{\tilde r_k}H^{\frac12} + \Dt^{r_q}), \nonumber \\
  & \qquad\qquad\qquad\qquad\qquad\qquad
  0 < r_k \le k+1, \quad 0 < \tilde r_k \le k+1, \quad 0 \le r_q \le q+1, \label{pi-0-approx} \\
  & \|\varphi - \P_h^\Dt \varphi\|_{\Omega_i^T}
  \le C \|\varphi\|_{H^{r_q}(0,T;H^{r_l}(\Omega_i))}(h^{r_l} + \Dt^{r_q}),
  \quad 0 \le r_l \le l+1, \quad 0 \le r_q \le q+1, \label{p-h-approx}\\
  & \|\varphi - \P_{H,\Gamma}^\DT\varphi\|_{\Gamma_{ij}^T} \le C \|\varphi\|_{H^{r_s}(0,T;H^{r_m}(\Gamma_{ij}))}
  (H^{r_m} + \DT^{r_s}), \quad 0 \le r_m \le m+1, \quad 0 \le r_s \le s+1. \label{p-gamma-approx}
\end{align}\end{subequations}
Bound~\eqref{pi-0-approx} follows from the approximation properties of $\bPi_{h,0}$ obtained
in~\cite{ACWY,APWY}. Bounds~\eqref{p-h-approx} and~\eqref{p-gamma-approx} are standard
approximation properties of the $L^2$ projection~\cite{ciarlet}.

In the analysis we will also use the following approximation property, which follows from the stability of the $L^2$ projection in $L^\infty$~\cite{crouzeix1987stability}:
\begin{align}\label{p-h-approx-infty}
\|\varphi - \P_h^\Dt \varphi\|_{L^\infty(0,T;L^2(\Omega_i))}
  \le C \|\varphi\|_{W^{r_q,\infty}(0,T;H^{r_l}(\Omega_i))}(h^{r_l} + \Dt^{r_q}),
  \quad 0 \le r_l \le l+1, \quad 0 \le r_q \le q+1.
\end{align}
We also recall the well-known discrete trace (inverse) inequality for a quasi-uniform mesh $\T_{h,i}$: for all $\vv \in \V_{h,i}$, $\|\vv\cdot\n_i\|_{\Gamma_i} \le C h_i^{-\frac12}\|\vv\|_{\Omega_i}$. This implies
\begin{equation}\label{trace}
  \forall \, \vv \in \V_{h,i}^\Dt, \quad
  \|\vv\cdot\n_i\|_{\Gamma_i^T} \le C h_i^{-\frac12}\|\vv\|_{\Omega_i^T}.
\end{equation}

\subsection{A priori error estimate}

We proceed with the error estimate for the space-time mortar MFE
method~\eqref{method}.

\begin{theorem}[A priori error estimate]\label{conv-thm}
  Assume that conditions~\eqref{mortar-ass} hold
  and that the solution to~\eqref{weak} is sufficiently smooth. Let the space and time meshes $\T_{h,i}$ and $\T^\Dt_i$ be quasi-uniform, as well as $h \le C h_i$ and $\Dt \le C \Dt_i \ \forall i$. Then
  there exists a constant $C > 0$ independent of the mesh sizes $h$, $H$, $\Dt$, and $\Delta T$, such that the solution to the space-time mortar MFE method~\eqref{method} satisfies
  \begin{equation}\label{error-estimate} \begin{split}
    & \|p - p_h^\Dt\|_\DG + \|\u - \u_h^\Dt\|_{\Omega^T} + \|p - p_h^\Dt\|_{\Omega^T} +
    \|\lambda - \lambda_H^\DT\|_{\Gamma^T} \\
    & \quad \le C \Big(
    \sum_i \|\u\|_{H^{r_q}(0,T;\H^{r_k}(\Omega_i))}(h^{r_k} + \Dt^{r_q})
    + \|\u\|_{H^{r_q}(0,T;\H^{\tilde r_k+\frac12}(\Omega))}(h^{\tilde r_k}H^{\frac12} + \Dt^{r_q})\\
    & \qquad + \sum_i \|p\|_{W^{r_q,\infty}(0,T;H^{r_l}(\Omega_i))}
    \Delta t^{-\frac12} (h^{r_l} + \Dt^{r_q})
    + \sum_{i,j} \|\lambda\|_{H^{r_s}(0,T;H^{r_m}(\Gamma_{ij}))} h^{-\frac12}(H^{r_m} + \DT^{r_s})  \Big),   \\
    & \ 0 < r_k \le k+1, \ 0 < \tilde r_k \le k+1, \ 0 \le r_q \le q+1,
    \ 0 \le r_l \le l+1, \ 0 \le r_m \le m+1, \ 0 \le r_s \le s+1.
 \end{split}\end{equation}
  \end{theorem}

\begin{proof}
For
the purpose of the analysis, we consider the following equivalent
formulation of~\eqref{method} in the space of weakly continuous velocities $\V_{h,0}^\Dt$ given by~\eqref{weak-cont}: find $\u_{h,0}^\Dt \in \V_{h,0}^\Dt$ and $p_h^\Dt \in W_h^\Dt$ such that $(p_{h,0}^\Dt)^- = p_{h,0}$ and
\begin{subequations}\begin{align}
  & a^T(\u_h^\Dt,\vv) + b^T(\vv,p_h^\Dt) = 0
  \quad \forall \, \vv \in \V_{h,0}^\Dt, \label{weak-cont.1} \\
  & (\dt p_h^\Dt,w)_{\Omega^T} - b^T(\u_h^\Dt,w) = (q,w)_{\Omega^T}
  \quad \forall \, w \in W_h^\Dt. \label{weak-cont.2}
 \end{align}\end{subequations}
The fact that $\P_{H,\Gamma}^\DT$ defined in~\eqref{eq_PHG} maps to $\Lambda_H^\DT$ and definition~\eqref{weak-cont} imply that $b^T_{\Gamma}(\vv, \P_{H,\Gamma}^\DT \lambda) = 0$ for all $\vv \in \V_{h,0}^\Dt$, where $\lambda = p|_\Gamma$ is the pressure trace from~\eqref{weak}. Then,
subtracting~\eqref{weak-cont.1}--\eqref{weak-cont.2} from
\eqref{weak.1}--\eqref{weak.2},
we obtain the error equations
\begin{subequations}\begin{align}
  & a^T(\u - \u_h^\Dt,\vv) + b^T(\vv,\P_h^\Dt p - p_h^\Dt)
  + b^T_{\Gamma}(\vv,\lambda - \P_{H,\Gamma}^{\DT} \lambda) = 0
  \quad \forall \, \vv \in \V_{h,0}^\Dt, \label{error.1} \\
  & \big(\d_t p - \dt p_h^\Dt,w \big)_{\Omega^T} - b^T(\bPi_{h,0}^\Dt\u - \u_h^\Dt,w) = 0
  \quad \forall \, w \in W_h^\Dt, \label{error.2}
 \end{align}\end{subequations}
where we have also used~\eqref{div-orth} and~\eqref{pi-0-div} to incorporate the
interpolants $\P_h^\Dt$ and $\bPi_{h,0}^\Dt$.
We take $\vv = \bPi_{h,0}^\Dt\u - \u_h^\Dt$ and
$w = \P_h^\Dt p - p_h^\Dt$ and sum the two equations, resulting in
\begin{equation}\begin{split}
  & a^T(\bPi_{h,0}^\Dt\u - \u_h^\Dt,\bPi_{h,0}^\Dt\u - \u_h^\Dt)
  + \big(\d_t p - \dt p_h^\Dt,\P_h^\Dt p - p_h^\Dt\big)_{\Omega^T} \\
  & \qquad = a^T(\bPi_{h,0}^\Dt\u - \u,\bPi_{h,0}^\Dt\u - \u_h^\Dt)
  - b^T_{\Gamma}(\bPi_{h,0}^\Dt\u - \u_h^\Dt,\lambda - \P_{H,\Gamma}^{\DT} \lambda). \label{energy-1}
\end{split}\end{equation}
For the second term on the left of~\eqref{energy-1}, restricted to a subdomain, we write
\begin{equation}\begin{split}\label{dt-term}
& \int_0^T \big(\d_t p - \dt p_h^\Dt,\P_h^\Dt p - p_h^\Dt\big)_{\Omega_i}
    = \int_0^T \big(\dt(p - p_h^\Dt),\P_h^\Dt p - p_h^\Dt\big)_{\Omega_i} \\
& \qquad =  \int_0^T \big(\dt(p - p_h^\Dt),p - p_h^\Dt\big)_{\Omega_i}
+ \int_0^T \big(\dt(p - p_h^\Dt),\P_h^\Dt p - p\big)_{\Omega_i} =: I_1^i + I_2^i.
\end{split}\end{equation}
Using \eqref{dt-positive} and notation~\eqref{eq_DG}, for the first term, we have
\begin{equation}\label{I1}
  \sum_i I_1^i = \frac12\|p - p_h^\Dt\|_{\DG}^2 - \frac12\|p_0 - p_{h,0}\|_\Omega^2.
\end{equation}
Using~\eqref{dt-DG}, the second term is
\begin{equation}\label{I2}
 I_2^i = \sum_{k=1}^{N_i} \int_{t_i^{k-1}}^{t_i^k}
 \big(\d_t(p - p_h^\Dt),\P_h^\Dt p - p\big)_{\Omega_i}
 + \sum_{k=1}^{N_i} \big([p - p_h^\Dt]_{k-1},(\P_h^\Dt p - p)_{k-1}^+ \big)_{\Omega_i}.
\end{equation}
For the first term on the right above, using the orthogonality property of $\P_h^\Dt$, we develop
\begin{equation}\begin{split}\label{I21}
    & \sum_{k=1}^{N_i} \int_{t_i^{k-1}}^{t_i^k}
    \big(\d_t(p - p_h^\Dt),\P_h^\Dt p - p\big)_{\Omega_i}
    = \sum_{k=1}^{N_i} \int_{t_i^{k-1}}^{t_i^k}
    \big(\d_t(p - \P_h^\Dt p),\P_h^\Dt p - p\big)_{\Omega_i} \\
    &\qquad = - \frac12 \sum_{k=1}^{N_i} \int_{t_i^{k-1}}^{t_i^k}
    \d_t\|\P_h^\Dt p - p\|_{\Omega_i}^2
    = -\frac12\sum_{k=1}^{N_i}\|\P_h^\Dt p - p\|_{\Omega_i}^2\Big|_{t^{k-1}}^{t^k}.
\end{split}
\end{equation}
Combining~\eqref{energy-1}--\eqref{I21}, and using~\eqref{k-spd} and the Cauchy--Schwarz and Young inequalities, we obtain,
\begin{align*}
  & \|\bPi_{h,0}^\Dt\u - \u_h^\Dt\|_{\Omega^T}^2
  + \|p - p_h^\Dt\|_{\DG}^2 \nonumber \\
  & \quad \leq C\Big(\|\bPi_{h,0}^\Dt\u - \u\|_{\Omega^T}
  \|\bPi_{h,0}^\Dt\u - \u_h^\Dt\|_{\Omega^T}
  + \sum_i \|(\bPi_{h,0}^\Dt\u - \u_h^\Dt)\cdot\n_i\|_{\Gamma_i^T}
  \|\lambda - \P_{H,\Gamma}^{\DT} \lambda\|_{\Gamma_i^T} \\
  & \qquad + \sum_i \|[p - p_h^\Dt]_{k-1}\|_{\Omega_i}
  \|(\P_h^\Dt p - p)_{k-1}^+\|_{\Omega_i}
+ \sum_i\sum_{k=1}^{N_i}\|\P_h^\Dt p - p\|_{\Omega_i}^2\Big|_{t^{k-1}}^{t^k}
+ \|p_0 - p_{h,0}\|_\Omega^2 \Big)\\
  & \quad \leq \epsilon \left(\|\bPi_{h,0}^\Dt\u - \u_h^\Dt\|_{\Omega^T}^2
  + \|p - p_h^\Dt\|_{\DG}^2 \right) \nonumber \\
  & \qquad + C_{\epsilon} \Big(\|\bPi_{h,0}^\Dt\u - \u\|_{\Omega^T}^2
  + h^{-1}\|\lambda - \P_{H,\Gamma}^{\DT} \lambda\|_{\Gamma^T}^2
  + \sum_i \sum_{k=1}^{N_i}\|(\P_h^\Dt p - p)_{k-1}^+\|_{\Omega_i}^2 \Big)\\
  & \qquad + C\Big(\sum_i\sum_{k=1}^{N_i}\|\P_h^\Dt p - p\|_{\Omega_i}^2\Big|_{t^{k-1}}^{t^k}
+ \|p_0 - p_{h,0}\|_\Omega^2 \Big),
\end{align*}
where we used the trace (inverse) inequality~\eqref{trace} for a quasi-uniform mesh $\T_{h,i}$ and $h \le C h_i$ in the last estimate. Taking $\epsilon$ sufficiently small gives
\begin{equation}\begin{split}
 & \|\bPi_{h,0}^\Dt\u - \u_h^\Dt\|_{\Omega^T}
    + \|p - p_h^\Dt\|_{\DG} \le C \Big(\|\bPi_{h,0}^\Dt\u - \u\|_{\Omega^T}
+ h^{-\frac12}\|\lambda - \P_{H,\Gamma}^{\DT} \lambda\|_{\Gamma^T} \\
  & \qquad  \label{bound-u-p}
  + \Delta t ^{-\frac12}\|\P_h^\Dt p - p\|_{L^\infty(0,T;L^2(\Omega))}
+ \|p_0 - p_{h,0}\|_{\Omega} \Big),
\end{split}\end{equation}
where we used that $N_i \le \frac{CT}{\Dt_i}$ for a quasi-uniform time mesh $\T^\Dt_i$ and $\Dt \le C \Dt_i$ to obtain the factor $\Dt^{-\frac12}$.

Next, the inf--sup condition for the weakly continuous velocity
\eqref{inf--sup-weak-cont} and~\eqref{error.1} imply, using~\eqref{trace} and $h \le C h_i$,
\begin{equation}\label{bound-p}
  \|\P_h^\Dt p - p_h^\Dt\|_{\Omega^T} \le C \left( \|\u - \u_h^\Dt\|_{\Omega^T}
  + h^{-\frac12}\|\lambda - \P_{H,\Gamma}^{\DT} \lambda\|_{\Gamma^T} \right).
\end{equation}
Finally, to obtain a bound on $\lambda_H^\DT$, we subtract~\eqref{method.1} from
\eqref{weak.1}, to obtain the error equation
\begin{equation}\label{error-lambda}
  a^T(\u - \u_h^\Dt,\vv) + b^T(\vv,p - p_h^\Dt)
  + b^T_{\Gamma}(\vv,\P_{H,\Gamma}^{\DT} \lambda - \lambda_H^\DT)
  = b^T_{\Gamma}(\vv,\P_{H,\Gamma}^{\DT} \lambda - \lambda)
  \quad \forall \, \vv \in \V_{h}^\Dt.
\end{equation}
The mortar inf--sup condition~\eqref{inf--sup} and~\eqref{error-lambda}
imply, using~\eqref{trace} and $h \le C h_i$,
\begin{equation}\label{bound-lambda}
  \|\P_{H,\Gamma}^{\DT} \lambda - \lambda_H^\DT\|_{\Gamma^T} \le C \left(
  \|\u - \u_h^\Dt\|_{\Omega^T} + \|p - p_h^\Dt\|_{\Omega^T}
  + h^{-\frac12}\|\lambda - \P_{H,\Gamma}^{\DT} \lambda\|_{\Gamma^T} \right).
\end{equation}
The assertion of the theorem follows from combining~\eqref{bound-u-p},
\eqref{bound-p}, and~\eqref{bound-lambda} and using the triangle inequality, \eqref{init-error}, and
the approximation bounds~\eqref{approx}--\eqref{p-h-approx-infty}.
\end{proof}

\begin{remark}[The factors $\Delta t^{-\frac12}$ and $h^{-\frac12}$ and appropriate choice of the polynomial degrees $m$ and $s$] \label{rem:div}
The term $h^{-\frac12}(H^{r_m} + \DT^{r_s})$ in the error bound
appears due the use of the discrete trace (inverse) inequality~\eqref{trace} to
control the consistency error $b^T_{\Gamma}(\bPi_{h,0}^\Dt\u -
\u_h^\Dt,\lambda - \P_{H,\Gamma}^{\DT} \lambda)$. This term can be
made comparable to the other error terms in~\eqref{error-estimate} by
choosing $m$ and $s$ sufficiently large, assuming
that the solution is sufficiently smooth. Alternatively, this term can
be improved if a bound on $\|\div(\u - \u_h^\Dt)\|_{\Omega_i^T}$ is
available, with the use of the normal trace inequality for $\H({\rm div};\Omega_i)$ functions. In this case the factor $\Delta t^{-\frac12}$ in the term $\Delta t^{-\frac12} (h^{r_l} + \Dt^{r_q})$ can also be avoided by using a suitable time-interpolant for the pressure. We present this argument in the next section in the special case of matching time steps between subdomains; then, additionally, the assumptions on quasi-uniform space and time meshes $\T_{h,i}$ and $\T^\Dt_i$ can be avoided.
\end{remark}

\section{Control of the velocity divergence and improved error estimates}
\label{sec:div}

In this section we establish stability and error estimates for the velocity divergence, along with an improved error bound for the rest of the variables, as noted in
Remark~\ref{rem:div} above. For this section only, we make the following assumption on the temporal discretization:
\begin{equation}\label{dt-match}
\forall \, i,j, \quad W_i^\Dt = \Lambda_{ij}^\DT = W_j^\Dt.
\end{equation}
In particular, we assume that all subdomains and mortar interfaces have the same time discretization, which we denote by $W^\Dt$. Let $t^k$, $k = 0,\ldots,N$, be the discrete times and let $q$ be the polynomial degree in $W^\Dt$. In this section, consequently, $\Delta T = \Dt$  and $s=q$.

We will utilize the Radau reconstruction operator $\I$ \cite{Makr_Noch_a_post_par_06, Ern_Sme_Voh_heat_HO_Y_17}, which satisfies, for any $\varphi(x,\cdot) \in W^\Dt$,
$\I\varphi(x,\cdot) \in H^1(0,T)$, $\I\varphi(x,\cdot)|_{(t^{k-1},t^k)} \in P_{q+1}$, such that
\begin{equation*}
  \int_{t^{k-1}}^{t^k} \d_t \I\varphi \,\phi = \int_{t^{k-1}}^{t^k} \d_t \varphi \, \phi
+ [\varphi]_{k-1} \, \phi_{k-1}^+ \quad \forall \, \phi(x,\cdot) \in W^\Dt.
\end{equation*}
Recalling \eqref{dt-DG}, this implies that
\begin{equation}\label{radau-prop}
\int_0^T \d_t \I\varphi \,\phi = \int_0^T \dt \varphi \, \phi \quad \forall \, \phi(x,\cdot) \in W^\Dt.
\end{equation}
Then the second equation \eqref{method.2} of the space-time mortar mixed
method can be rewritten as
\begin{equation}\label{method.2-radau}
  (\d_t \I p_h^\Dt,w)_{\Omega^T} - b^T(\u_h^\Dt,w) = (q,w)_{\Omega^T}
  \quad \forall \, w \in W_h^\Dt.
\end{equation}

For notational convenience, for $\vv \in \V$, let henceforth $\divh \vv \in L^2(\Omega)$ be such that $\forall \, i$, $(\divh \vv)|_{\Omega_i} = \div (\vv|_{\Omega_i})$.

\subsection{Stability bound}

We proceed with the stability bound for $\|\divh \u_h^\Dt\|_{L^2(0,T;L^2(\O))}$; under assumption~\eqref{dt-match}, this complements the bound~\eqref{stability} of Theorem~\ref{well-posed}.

\begin{theorem}[Control of divergence]
Assume that condition \eqref{dt-match} holds. Then, for the solution of the space-time mortar method~\eqref{method}, there exists a constant $C > 0$ independent of $h$, $H$, $\Dt$, and $\Delta T$ such that
\begin{equation*}
  \|\d_t \I p_h^\Dt\|_{\Omega^T} + \|\divh \u_h^\Dt\|_{\Omega^T}
  + \|\u_h^\Dt\|_{\DG}
  \le C(\|q\|_{\Omega^T} + \|\div K \nabla p_{0}\|_{\O}).
\end{equation*}
\end{theorem}

\begin{proof}
Since $\d_t \I p_h^\Dt \in W_h^\Dt$ and $\divh \u_h^\Dt \in W_h^\Dt$, \eqref{method.2-radau} implies that $\d_t \I p_h^\Dt + \divh \u_h^\Dt = \P_h^\Dt q$. Therefore,
\begin{equation}\label{div.1}
  \|\d_t \I p_h^\Dt\|_{\Omega^T}^2 + \|\divh \u_h^\Dt\|_{\Omega^T}^2
  + 2\big(\d_t \I p_h^\Dt,\divh \u_h^\Dt\big)_{\Omega^T}
  = \|\P_h^\Dt q\|_{\Omega^T}^2.
\end{equation}
To control the third term on the left, we note that \eqref{method.1} implies that for each $k = 1,\ldots,N$ and every $t \in [t^{k-1},t^k]$, it holds that
$$
a(\u_h^\Dt,\vv) + b(\vv, p_h^\Dt)
+ b_\Gamma(\vv,\lambda_H^\DT) = 0 \quad \forall \, \vv \in \V_h.
$$
Therefore, using that the initial data satisfy the above equation, see
\eqref{init-data} and \eqref{init.1}, we have
\begin{equation}\label{method.1-radau}
  a^T(\d_t \I \u_h^\Dt,\vv) + b^T(\vv,\d_t \I p_h^\Dt) + b^T_{\Gamma}(\vv,\d_t \I \lambda_H^\DT) = 0 \quad \forall \, \vv \in \V_h^\Dt.
\end{equation}
Taking $\vv = \u_h^\Dt$ in \eqref{method.1-radau} and $\mu = \d_t \I \lambda_H^\DT$ in \eqref{method.3} and combining the equations, we obtain
\begin{equation}\label{div.2}
-b^T(\u_h^\Dt,\d_t \I p_h^\Dt) = a^T(\d_t \I \u_h^\Dt,\u_h^\Dt)
= \frac12 \|K^{-\frac12}\u_h^\Dt\|^2_{\DG} - \frac12\|K^{-\frac12}(\u_h^\Dt)_0^-\|_\Omega^2,
\end{equation}
using \eqref{radau-prop}, \eqref{dt-positive}, and \eqref{eq_DG} for the second equality. The assertion of the lemma follows by combining \eqref{div.1} and \eqref{div.2} and using \eqref{k-spd}, \eqref{init-data}, and \eqref{init-stab}.
\end{proof}

\subsection{Improved a priori error error estimate}

In this section we utilize the control on $\|\divh \u_h^\Dt\|_{\Omega^T}$ to obtain error estimates that avoid the factors $h^{-\frac12}$ and $\Delta t^{-\frac12}$ that appear in the error estimate~\eqref{error-estimate}, together with the quasi-uniformity assumption on the space and time meshes $\T_{h,i}$ and $\T^\Dt_i$. To this end, we will use an alternative time interpolant. Let $\tilde\P^\Dt: H^1(0,T) \to W^\Dt$ be such that, for any $\varphi \in H^1(0,T)$,
\begin{equation}\label{P-tilde-defn}
\forall \, k = 1,\ldots,N, \quad
  \int_{t^{k-1}}^{t^k} (\tilde\P^\Dt \varphi - \varphi) w = 0 \quad \forall \, w \in P_{q-1}, \quad (\tilde\P^\Dt \varphi)_k^- = \varphi(t^{k}).
\end{equation}
Let us further set $(\tilde\P^\Dt \varphi)_0^- = \varphi(0)$ and define the space-time interpolant
\begin{equation}\label{eq_tilde_sp_tm}
    \tilde\P_h^\Dt = \P_h \times \tilde\P^\Dt.
\end{equation}
The following properties of $\tilde\P^\Dt$ and $\tilde\P_h^\Dt$ will be useful in the analysis.

\begin{lemma}[Time derivative orthogonality]\label{lem:P-tilde-prop}
For all $\varphi \in H^1(0,T)$ and $w \in W^\Dt$,
\begin{equation}\label{P-tilde-prop}
\int_0^T \d_t \varphi \, w = \int_0^T \dt \tilde\P^\Dt \varphi \, w.
\end{equation}
Furthermore, for all $\varphi \in H^1(0,T)$ and $w \in W_h^\Dt$,
\begin{equation}\label{P-tilde-prop-h}
\int_0^T (\d_t \varphi, w)_{\Omega} = \int_0^T (\dt \tilde\P_h^\Dt \varphi,w)_{\Omega}.
\end{equation}
\end{lemma}
\begin{proof}
Using \eqref{P-tilde-defn}, we write
\begin{equation*}\begin{split}
    & \int_0^T \d_t \varphi \, w = \sum_{k=1}^N \int_{t^{k-1}}^{t^k}\d_t \varphi \, w
    = \sum_{k=1}^N\Big(-\int_{t^{k-1}}^{t^k}\varphi \, \d_t w + \varphi(t^k) w_k^-
    - \varphi(t^{k-1}) w_{k-1}^+ \Big)\\
    & \quad = \sum_{k=1}^N\Big(-\int_{t^{k-1}}^{t^k}\tilde\P^\Dt \varphi \, \d_t w
    + (\tilde\P^\Dt \varphi)_k^- w_k^-
    - (\tilde\P^\Dt \varphi)_{k-1}^- w_{k-1}^+ \Big)\\
    & \quad = \sum_{k=1}^N \Big( \int_{t^{k-1}}^{t^k}\d_t \tilde\P^\Dt \varphi \, w
    + [\tilde\P^\Dt \varphi]_{k-1}w_{k-1}^+ \Big)
    = \int_0^T \dt \tilde\P^\Dt \varphi \, w,
  \end{split}
\end{equation*}
where we used the definition \eqref{dt-DG} in the last equality. This establishes \eqref{P-tilde-prop}. The identity \eqref{P-tilde-prop-h} follows from \eqref{P-tilde-prop}, using that $\displaystyle \int_0^T(\d_t \varphi, w)_{\Omega} = \int_0^T (\d_t \P_h \varphi,w)_{\Omega}$.
\end{proof}

Consider the space-time interpolants $\tilde\bPi_{h,0}^\Dt = \bPi_{h,0} \times\tilde\P^\Dt$ in $\V_{h,0}^\Dt$ and
$\tilde\P_h^\Dt$ from~\eqref{eq_tilde_sp_tm} in $W_h^\Dt$. Similarly to \eqref{pi-0-approx} and \eqref{p-h-approx}, they satisfy
\begin{align}
& \|\bpsi - \tilde\bPi_{h,0}^\Dt\bpsi\|_{\Omega^T} \le C
  \sum_i \|\bpsi\|_{H^{r_q}(0,T;\H^{r_k}(\Omega_i))}(h^{r_k} + \Dt^{r_q})
  + C \|\bpsi\|_{H^{r_q}(0,T;\H^{\tilde r_k+\frac12}(\Omega))}(h^{\tilde r_k}H^{\frac12} + \Dt^{r_q}), \nonumber \\
  & \qquad\qquad\qquad\qquad\qquad\qquad
  0 < r_k \le k+1, \quad 0 < \tilde r_k \le k+1, \quad 1 \le r_q \le q+1, \label{tilde-pi-0-approx} \\
  & \|\div(\bpsi - \tilde\bPi_{h,0}^\Dt\bpsi)\|_{\Omega_i^T} \le C
  \|\div\bpsi\|_{H^{r_q}(0,T;H^{r_l}(\Omega_i))}(h^{r_l} + \Dt^{r_q}), \nonumber \\
  & \qquad\qquad\qquad\qquad\qquad\qquad
  0 \le r_l \le l+1, \quad 1 \le r_q \le q+1,
  \label{tilde-div-approx}\\
& \|\varphi - \tilde\P_h^\Dt \varphi\|_{\Omega_i^T}
  \le C \|\varphi\|_{H^{r_q}(0,T;H^{r_l}(\Omega_i))}(h^{r_l} + \Dt^{r_q}),
  \quad 0 \le r_l \le l+1, \quad 1 \le r_q \le q+1, \label{tilde-p-h-approx}
\end{align}
with similar bounds in $\|\cdot\|_{L^\infty(0,T;L^2(\Omega_i))}$, cf. \eqref{p-h-approx-infty}.
The use of $\tilde\P_h^\Dt$ will allow us to avoid the term $\Dt^{-\frac12}$ in the error estimate.

We will also utilize the Scott--Zhang interpolant \cite{Scott-Zhang} $\tilde \P_{H,\Gamma}:H^1(\Gamma) \to \Lambda_{H,\Gamma}\cap C(\Gamma)$, which can be defined to preserve the trace on $\partial \Gamma$. Thus the function $\varphi - \tilde \P_{H,\Gamma} \, \varphi$ can be extended continuously by zero to $\d\Omega \setminus \Gamma$ in the $H^{\frac12}$-norm. Let us denote this extension by
$E(\varphi - \tilde \P_{H,\Gamma} \, \varphi)$. Let
$\tilde \P_{H,\Gamma}^\DT = \tilde \P_{H,\Gamma} \times \tilde\P^\Dt$ be the space-time mortar interpolant. It has the approximation property \cite{Scott-Zhang}
\begin{equation}\begin{split}\label{SZ-approx}
& \|\varphi - \tilde\P_{H,\Gamma}^\DT\|_{t,\Gamma_{ij}^T}  \le C \|\varphi\|_{H^{r_s}(0,T;H^{r_m}(\Gamma_{ij}))}
(H^{r_m - t} + \DT^{r_s}), \\
& \qquad 1 \le r_m \le m+1, \quad 1 \le r_s \le s+1,
\quad 0 \le t \le 1.
\end{split}
\end{equation}
The use of $\tilde \P_{H,\Gamma}^\DT$, along with the well-known normal trace inequality, for all $\vv \in \H({\rm div}; \Omega_i)$, \linebreak $\|\vv\cdot\n_i\|_{-\frac12,\d\Omega_i} \le C \|\vv\|_{{\rm div};\Omega_i}$, which implies
\begin{equation}\label{normal-trace}
  \forall \, \vv \in L^2(0,T;\V_i), \quad
  \|\vv\cdot\n_i\|_{L^2(0,T;H^{-\frac12}(\d\Omega_i))} \le C \|\vv\|_{L^2(0,T;\V_i)},
\end{equation}
will allow us to avoid the $h^{-\frac12}$ term in the error estimate.

We are ready to prove the following error estimate that improves the bound~\eqref{error-estimate} of Theorem~\ref{conv-thm} under assumption~\eqref{dt-match}.

\begin{theorem}[Improved a priori error estimate]\label{conv-thm-tilde}
Assume that conditions~\eqref{mortar-assumption-space} and~\eqref{dt-match} hold
  and that the solution to~\eqref{weak} is sufficiently smooth. Then
  there exists a constant $C > 0$ independent of the mesh sizes $h$, $H$, $\Dt$, and $\Delta T$, such that the solution to the space-time mortar MFE method~\eqref{method} satisfies
  \begin{equation}\label{error-estimate-tilde} \begin{split}
      & \|\u(t^N) - (\u_h^\Dt)_N^-\|_{\Omega} + \|\u - \u_h^\Dt\|_{\Omega^T}
      + \|\divh(\u - \u_h^\Dt)\|_{\Omega^T} \\
      & \qquad\quad + \|p(t^N) - (p_h^\Dt)_N^-\|_\Omega + \|p - p_h^\Dt\|_{\Omega^T}
    + \|\lambda - \lambda_H^\DT\|_{\Gamma^T} \\
    & \quad \le C \Big(
    \sum_i \|\u\|_{W^{r_q,\infty}(0,T;\H^{r_k}(\Omega_i))}(h^{r_k} + \Dt^{r_q})
    + \|\u\|_{W^{r_q,\infty}(0,T;\H^{\tilde r_k+\frac12}(\Omega))}(h^{\tilde r_k}H^{\frac12} + \Dt^{r_q})\\
    & \qquad + \sum_i \|p\|_{W^{r_q,\infty}(0,T;H^{r_l}(\Omega_i))}
    (h^{r_l} + \Dt^{r_q})
    + \sum_{i,j} \|\lambda\|_{H^{r_s}(0,T;H^{r_m}(\Gamma_{ij}))} (H^{r_m-\frac12} + \DT^{r_s})\\
    & \qquad + \sum_i \|\u\|_{H^1(0,T;\H^{r_k}(\Omega_i))}h^{r_k}
    + \|\u\|_{H^1(0,T;\H^{\tilde r_k+\frac12}(\Omega))}h^{\tilde r_k}H^{\frac12}
    + \sum_{i,j} \|\lambda\|_{H^1(0,T;H^{r_m}(\Gamma_{ij}))} H^{r_m-\frac12}
    \Big),   \\
    & \ 0 < r_k \le k+1, \ 0 < \tilde r_k \le k+1, \ 1 \le r_q \le q+1,
    \ 0 \le r_l \le l+1, \ \frac12 \le r_m \le m+1, \ 1 \le r_s \le s+1.
 \end{split}\end{equation}
\end{theorem}

\begin{proof}
Subtracting~\eqref{weak-cont.1}--\eqref{weak-cont.2} from
\eqref{weak.1}--\eqref{weak.2}, we obtain the error equations,
$\forall \, \vv \in \V_{h,0}^\Dt$, $w \in W_h^\Dt$,
\begin{subequations}\begin{align}
    & a^T(\u - \u_h^\Dt,\vv) + b^T(\vv,\tilde\P_h^\Dt p - p_h^\Dt)
    + b^T(\vv,p - \tilde\P_h^\Dt p)
    + b^T_{\Gamma}(\vv,\lambda - \tilde\P_{H,\Gamma}^{\DT} \lambda) = 0, \label{error-tilde.1} \\
  & \big(\d_t p - \dt p_h^\Dt,w \big)_{\Omega^T}
  - b^T(\tilde\bPi_{h,0}^\Dt\u - \u_h^\Dt,w)
- b^T(\u - \tilde\bPi_{h,0}^\Dt\u,w) = 0. \label{error-tilde.2}
 \end{align}\end{subequations}
We note the extra third terms in \eqref{error-tilde.1} and \eqref{error-tilde.2} compared to \eqref{error.1} and \eqref{error.2},
which appear since $\tilde\P^\Dt$ has orthogonality only for $P_{q-1}$.
We take $\vv = \tilde\bPi_{h,0}^\Dt\u - \u_h^\Dt$ and $w = \tilde\P_h^\Dt p - p_h^\Dt$ and sum the two equations, obtaining
\begin{equation}\begin{split}
  & a^T(\tilde\bPi_{h,0}^\Dt\u - \u_h^\Dt,\tilde\bPi_{h,0}^\Dt\u - \u_h^\Dt)
  + \big(\d_t p - \dt p_h^\Dt,\tilde\P_h^\Dt p - p_h^\Dt\big)_{\Omega^T} \\
  & \qquad = a^T(\tilde\bPi_{h,0}^\Dt\u - \u,\tilde\bPi_{h,0}^\Dt\u - \u_h^\Dt)
  - b^T(\tilde\bPi_{h,0}^\Dt\u - \u_h^\Dt,p - \tilde\P_h^\Dt p)\\
  & \qquad\quad
  - b^T_{\Gamma}(\tilde\bPi_{h,0}^\Dt\u - \u_h^\Dt,
  \lambda - \tilde\P_{H,\Gamma}^{\DT} \lambda)
+ b^T(\u - \tilde\bPi_{h,0}^\Dt\u,\tilde\P_h^\Dt p - p_h^\Dt). \label{tilde-energy-1}
\end{split}\end{equation}
For the second term on the left, using \eqref{P-tilde-prop-h}, \eqref{dt-positive}, and \eqref{eq_DG}, we write
\begin{equation}\begin{split}\label{tilde-dt-term}
  \big(\d_t p - \dt p_h^\Dt, \tilde\P_h^\Dt p - p_h^\Dt\big)_{\Omega^T}
& = \big(\dt(\tilde\P_h^\Dt p - p_h^\Dt),\tilde\P_h^\Dt p - p_h^\Dt \big)_{\Omega^{T}}\\
& = \frac12\|\tilde\P_h^\Dt p - p_h^\Dt\|_{\DG}^2
  - \frac12\|\P_h p_0 - p_{h,0}\|_{\Omega}^2.
  \end{split}
\end{equation}
Combining \eqref{tilde-energy-1} and \eqref{tilde-dt-term}, and using the Cauchy--Schwarz and Young inequalities, we obtain,
\begin{equation}\begin{split}\label{tilde-energy-2}
  & \|\tilde\bPi_{h,0}^\Dt\u - \u_h^\Dt\|_{\Omega^T}^2
  + \|\tilde\P_h^\Dt p - p_h^\Dt\|_{\DG}^2 \\
  & \quad \leq C\Big(\|\tilde\bPi_{h,0}^\Dt\u - \u\|_{\Omega^T}
  \|\tilde\bPi_{h,0}^\Dt\u - \u_h^\Dt\|_{\Omega^T}
  + \|\divh(\tilde\bPi_{h,0}^\Dt\u - \u_h^\Dt)\|_{\Omega^T}
  \|p - \tilde\P_h^\Dt p\|_{\Omega^T} + \|\P_h p_0 - p_{h,0}\|_\Omega^2\\
  & \quad\quad + \|\divh(\u - \tilde\bPi_{h,0}^\Dt\u)\|_{\Omega^T}
  \|\tilde\P_h^\Dt p - p_h^\Dt\|_{\Omega^T}\\
  & \quad\quad + \sum_i \|(\tilde\bPi_{h,0}^\Dt\u - \u_h^\Dt)\cdot\n_i\|_{L^2(0,T;H^{-\frac12}(\d\Omega_i))}
  \|E(\lambda - \tilde\P_{H,\Gamma}^{\DT} \lambda)\|_{L^2(0,T;H^\frac12(\d\Omega_i))}
  \Big) \\
  & \quad \le \epsilon_1 \left(\|\tilde\bPi_{h,0}^\Dt\u - \u_h^\Dt\|^2_{\Omega^T}
  + \|\divh(\tilde\bPi_{h,0}^\Dt\u - \u_h^\Dt)\|^2_{\Omega^T}
  + \|\tilde\P_h^\Dt p - p_h^\Dt\|_{\Omega^T}^2 \right)
  + C \|\P_h p_0 - p_{h,0}\|_\Omega^2 \\
  & \quad\quad + C_{\epsilon_1}\Big(\|\tilde\bPi_{h,0}^\Dt\u - \u\|_{\Omega^T}^2
  + \|\divh(\u - \tilde\bPi_{h,0}^\Dt\u)\|_{\Omega^T}^2
  + \|p - \tilde\P_h^\Dt p\|_{\Omega^T}^2
  + \|\lambda - \tilde\P_{H,\Gamma}^{\DT} \lambda\|_{L^2(0,T;H^\frac12(\Gamma))}^2 \Big),
  \end{split}
\end{equation}
where we used \eqref{normal-trace} in the last inequality. To complete the estimate, we bound the pressure, mortar, and divergence errors. The inf--sup condition for the weakly continuous velocity
\eqref{inf--sup-weak-cont} and~\eqref{error-tilde.1} imply, using~\eqref{normal-trace},
\begin{equation}\begin{split}\label{bound-p-tilde}
  & \|\tilde\P_h^\Dt p - p_h^\Dt\|_{\Omega^T} \le C \Big(
  \|\tilde\bPi_{h,0}^\Dt\u - \u_h^\Dt\|_{\Omega^T}
  + \|\u - \tilde\bPi_{h,0}^\Dt\u\|_{\Omega^T} \\
  & \qquad + \|p - \tilde\P_h^\Dt p\|_{\Omega^T}
  + \|\lambda - \tilde\P_{H,\Gamma}^{\DT} \lambda\|_{L^2(0,T;H^\frac12(\Gamma))} \Big).
  \end{split}
\end{equation}
To bound the mortar error using the mortar inf--sup condition \eqref{inf--sup}, we
subtract~\eqref{method.1} from \eqref{weak.1}, to obtain the error equation
\begin{equation}\label{error-lambda-tilde}
a^T(\u - \u_h^\Dt,\vv) + b^T(\vv,p - p_h^\Dt)
    + b^T_{\Gamma}(\vv,\tilde\P_{H,\Gamma}^{\DT} \lambda - \lambda_H^\DT)
+ b^T_{\Gamma}(\vv,\lambda - \tilde\P_{H,\Gamma}^{\DT} \lambda)
    = 0 \quad \forall \, \vv \in \V_{h}^\Dt.
\end{equation}
The mortar inf--sup condition \eqref{inf--sup} and \eqref{error-lambda-tilde}
imply, using~\eqref{normal-trace},
\begin{equation}\begin{split}\label{bound-lambda-tilde}
    & \|\tilde\P_{H,\Gamma}^{\DT} \lambda - \lambda_H^\DT\|_{\Gamma^T} \le C \Big(
 \|\tilde\bPi_{h,0}^\Dt\u - \u_h^\Dt\|_{\Omega^T}
 + \|\u - \tilde\bPi_{h,0}^\Dt\u\|_{\Omega^T}\\
& \qquad + \|\tilde\P_h^\Dt p - p_h^\Dt\|_{\Omega^T}
 + \|p - \tilde\P_h^\Dt p\|_{\Omega^T}
  + \|\lambda - \tilde\P_{H,\Gamma}^{\DT} \lambda\|_{L^2(0,T;H^\frac12(\Gamma))} \Big).
  \end{split}
\end{equation}
Next, to bound on the divergence error, using \eqref{radau-prop} and \eqref{P-tilde-prop-h}, we rewrite \eqref{error-tilde.2} as
$$
\big(\d_t(\I\tilde \P_h^\Dt p - \I p_h^\Dt),w \big)_{\Omega^T}
- b^T(\tilde\bPi_{h,0}^\Dt\u - \u_h^\Dt,w) = b^T(\u - \tilde\bPi_{h,0}^\Dt\u,w),
$$
concluding that $\d_t(\I \tilde \P_h^\Dt p - \I p_h^\Dt)
+ \divh(\tilde\bPi_{h,0}^\Dt\u - \u_h^\Dt) = -\P_h^\Dt \big(\divh(\u - \tilde\bPi_{h,0}^\Dt\u)\big)$. Therefore,
\begin{equation}\begin{split}\label{div-err-1}
    & \|\d_t(\I \tilde \P_h^\Dt p - \I p_h^\Dt)\|_{\Omega^T}^2
    + \|\divh(\tilde\bPi_{h,0}^\Dt\u - \u_h^\Dt)\|_{\Omega^T}^2 \\
    & \qquad
    + 2\big(\d_t(\I \tilde \P_h^\Dt p - \I p_h^\Dt),
    \divh(\tilde\bPi_{h,0}^\Dt\u - \u_h^\Dt) \big)_{\Omega^T}
    = \|\P_h^\Dt\big(\divh(\u - \tilde\bPi_{h,0}^\Dt\u)  \big)\|_{\Omega^T}^2.
  \end{split}
\end{equation}
To control the third term on the left, we note that, since the subdomain and mortar time discretizations are the same, the error equation \eqref{error-tilde.1} holds for every $t \in [t^{k-1},t^k]$, $k = 1,\ldots,N$, with a test function $\vv \in \V_{h,0}$. Therefore, similarly to \eqref{method.1-radau}, it implies that
\begin{equation}\label{div-err-2}
a^T(\d_t(\u - \I \u_h^\Dt),\vv) + b^T(\vv,\d_t(p - \I p_h^\Dt))
  + b^T_{\Gamma}(\vv,\d_t(\lambda - \I \lambda_H^\DT)) = 0 \quad \forall \, \vv \in \V_{h}^\Dt.
\end{equation}
The first term is manipulated as
\begin{equation}\begin{split}\label{u-term}
  a^T(\d_t(\u - \I \u_h^\Dt),\vv) & =
  a^T(\d_t(\u - \bPi_{h,0} \u),\vv)
  + a^T(\d_t(\bPi_{h,0}\u  - \I \u_h^\Dt),\vv)\\
  & = a^T(\d_t(\u - \bPi_{h,0} \u),\vv)
  + a^T(\dt(\tilde\bPi_{h,0}^\Dt\u  - \u_h^\Dt),\vv),
  \end{split}
\end{equation}
using \eqref{P-tilde-prop} and \eqref{radau-prop} in the last equality. For the second term in \eqref{div-err-2} we write
\begin{equation}\label{p-term}
  b^T(\vv,\d_t(p - \I p_h^\Dt))
      = b^T(\vv,\d_t(\I \tilde \P_h^\Dt p - \I p_h^\Dt)),
\end{equation}
using \eqref{P-tilde-prop-h}, \eqref{radau-prop}, and the fact that $\divh \vv \in W_h^\Dt$. The third term in \eqref{div-err-2} is manipulated as
\begin{equation}\begin{split}\label{lambda-term}
    b^T_{\Gamma}(\vv,\d_t(\lambda - \I \lambda_H^\DT))
    & = b^T_{\Gamma}(\vv,\d_t(\lambda - \tilde\P_{H,\Gamma} \lambda))
    + b^T_{\Gamma}(\vv,\d_t(\tilde\P_{H,\Gamma} \lambda - \I \lambda_H^\DT))\\
    & = b^T_{\Gamma}(\vv,\d_t(\lambda - \tilde\P_{H,\Gamma} \lambda))
    + b^T_{\Gamma}(\vv,\d_t(\I \tilde\P_{H,\Gamma}^\Dt \lambda - \I \lambda_H^\DT)),
\end{split}
\end{equation}
using \eqref{P-tilde-prop} and \eqref{radau-prop} in the last equality. Now, combining \eqref{div-err-2}--\eqref{lambda-term}, taking $\vv = \tilde\bPi_{h,0}^\Dt\u  - \u_h^\Dt \in \V_{h,0}^\Dt$, and using that
$\d_t(\I \tilde\P_{H,\Gamma}^\Dt \lambda - \I \lambda_H^\DT) \in \Lambda_H^\DT$ to deduce that the last term in \eqref{lambda-term} is zero, we obtain
\begin{equation}\begin{split}\label{div-err-3}
& \big(\d_t(\I \tilde \P_h^\Dt p - \I p_h^\Dt),
\divh(\tilde\bPi_{h,0}^\Dt\u - \u_h^\Dt) \big)_{\Omega^T} =
a^T(\dt(\tilde\bPi_{h,0}^\Dt\u  - \u_h^\Dt),\tilde\bPi_{h,0}^\Dt\u  - \u_h^\Dt)\\
&\qquad + a^T(\d_t(\u - \bPi_{h,0} \u),\tilde\bPi_{h,0}^\Dt\u  - \u_h^\Dt)
+ b^T_{\Gamma}(\tilde\bPi_{h,0}^\Dt\u  - \u_h^\Dt,
  \d_t(\lambda - \tilde\P_{H,\Gamma} \lambda)).
\end{split}
\end{equation}
Combining \eqref{div-err-1} and \eqref{div-err-3} and using \eqref{dt-positive}
and the Cauchy--Schwarz and Young inequalities, we arrive at
\begin{equation}\begin{split}\label{div-err-4}
& \|\d_t(\I \tilde \P_h^\Dt p - \I p_h^\Dt)\|_{\Omega^T}^2
    + \|\divh(\tilde\bPi_{h,0}^\Dt\u - \u_h^\Dt)\|_{\Omega^T}^2
    + \|\tilde\bPi_{h,0}^\Dt\u  - \u_h^\Dt\|_{\DG}^2\\
& \quad  \le
\epsilon_2\left(\|\tilde\bPi_{h,0}^\Dt\u  - \u_h^\Dt\|_{\Omega^T}^2
+ \|\divh(\tilde\bPi_{h,0}^\Dt\u  - \u_h^\Dt)\|_{\Omega^T}^2 \right)
+ \|\divh(\u - \tilde\bPi_{h,0}^\Dt\u)\|_{\Omega^T}^2
\\
& \qquad + C_{\epsilon_2}\left(\|\d_t(\u - \bPi_{h,0} \u)_{\Omega^T}^2
+ \|\d_t(\lambda - \tilde\P_{H,\Gamma} \lambda)\|_{L^2(0,T;H^\frac12(\Gamma))}^2
\right)
+ C\|\bPi_{h,0}\u_0 - \u_{h,0}\|_{\Omega}^2.
  \end{split}
\end{equation}
Combining \eqref{tilde-energy-2}, \eqref{bound-p-tilde}, \eqref{bound-lambda-tilde}, and \eqref{div-err-4}, taking $\epsilon_1$ and $\epsilon_2$ small enough, we obtain
\begin{equation*}\begin{split}
    & \|\tilde\bPi_{h,0}^\Dt\u - \u_h^\Dt\|_{\Omega^T}^2
    + \|\divh(\tilde\bPi_{h,0}^\Dt\u - \u_h^\Dt)\|_{\Omega^T}^2
    + \|\tilde\bPi_{h,0}^\Dt\u  - \u_h^\Dt\|_{\DG}^2 \\
& \qquad\qquad    + \|\tilde\P_h^\Dt p - p_h^\Dt\|_{\Omega^T}^2
    + \|\tilde\P_h^\Dt p - p_h^\Dt\|_{\DG}^2
    + \|\tilde\P_{H,\Gamma}^{\DT} \lambda - \lambda_H^\DT\|_{\Gamma^T}^2\\
    & \quad \le C\Big(
    \|\u - \tilde\bPi_{h,0}^\Dt\u\|_{L^2(0,T;\V)}^2
+ \|p - \tilde\P_h^\Dt p\|_{\Omega^T}^2
+ \|\lambda - \tilde\P_{H,\Gamma}^{\DT} \lambda\|_{L^2(0,T;H^\frac12(\Gamma))}^2
+ \|\d_t(\u - \bPi_{h,0} \u)\|_{\Omega^T}^2
\\
& \qquad\qquad
+ \|\d_t(\lambda - \tilde\P_{H,\Gamma} \lambda)\|_{L^2(0,T;H^\frac12(\Gamma))}^2
    + \|\bPi_{h,0}\u_0 - \u_{h,0}\|_{\Omega}^2
    + \|\P_h p_0 - p_{h,0}\|_\Omega^2 \Big).
  \end{split}
\end{equation*}
Finally, using the triangle inequality, the approximation properties \eqref{tilde-pi-0-approx}--\eqref{SZ-approx}, and the initial error bound \eqref{init-error}, we arrive at \eqref{error-estimate-tilde}. We remark that in the final bound, we have kept only the term at time $t^N$ from the norm $\|\cdot\|_\DG$, since the approximation error in the full $\|\cdot\|_\DG$ norm involves a $\Dt^{-\frac12}$ factor.
\end{proof}

\begin{remark}[Significance of the improved error estimate]
  The error estimate \eqref{error-estimate-tilde} avoids the factors $h^{-\frac12}$ and $\Dt^{-\frac12}$, which appeared in the earlier bound \eqref{error-estimate}, and thus provides optimal order of convergence for all variables. Moreover, it provides a bound on the velocity divergence error. To the best of the authors' knowledge, such result has not been established in the literature for space-time mixed finite element methods with a DG time discretization, even on a single domain.
\end{remark}

\section{Reduction to an interface problem}\label{sec_red_int}

In this section we combine the time-dependent Steklov--Poincar\'e operator
approach from~\cite{hoang2013space} with the mortar domain
decomposition algorithm from~\cite{ACWY,APWY} to reduce the global
problem~\eqref{method} to a space-time interface
problem.

\subsection{Decomposition of the solution}

Consider a decomposition of the solution to
\eqref{method} in the form
\begin{equation}\label{split}
\u_h^\Dt = \u_h^{\Dt,*}(\lambda_H^\DT) + \overline \u_h^\Dt, \quad
p_h^\Dt = p_h^{\Dt,*}(\lambda_H^\DT) + \overline p_h^\Dt.
\end{equation}
Here, $\overline \u_h^\Dt \in \V_h^\Dt$,
$\overline p_h^\Dt \in W_h^\Dt$ are such that for each $\Omega_i^T$,
$(\overline \u_h^\Dt|_{\O_i^T} \in \V_{h,i}^\Dt, \, \overline p_h^\Dt|_{\O_i^T} \in W_{h,i}^\Dt)$ is the solution
to the space-time subdomain problem
in $\Omega_i^T$ with zero Dirichlet data on the space-time interfaces and the
prescribed source term, initial data,
and boundary data on the external boundary:
\begin{subequations}\label{bar}\begin{align}
  & a^T_i(\overline\u_h^\Dt,\vv) + b^T_i(\vv,\overline p_h^\Dt) = 0
  \quad \forall \, \vv \in \V_{h,i}^\Dt \label{bar.1} \\
  & (\dt \overline p_h^\Dt,w)_{\Omega_i^T} - b^T_i(\overline\u_h^\Dt,w) = (q,w)_{\Omega_i^T}
  \quad \forall \, w \in W_{h,i}^\Dt. \label{bar.2}
\end{align}\end{subequations}
Furthermore, for a given $\mu \in \Lambda_{H}^\DT$,
$\u_h^{\Dt,*}(\mu) \in \V_h^\Dt$,
$p_h^{\Dt,*}(\mu) \in W_h^\Dt$ are such that for each $\Omega_i^T$,
$(\u_h^{\Dt,*}(\mu)|_{\O_i^T}\in \V_{h,i}^\Dt, \, p_h^{\Dt,*}(\mu)|_{\O_i^T} \in W_{h,i}^\Dt)$ is
the solution to the space-time subdomain problem
in $\Omega_i^T$ with Dirichlet data $\mu$ on the space-time interfaces and zero
source term, initial data, and boundary data on the external boundary:
\begin{subequations}\label{star}\begin{align}
  & a^T_i(\u_h^{\Dt,*}(\mu),\vv) + b^T_i(\vv,p_h^{\Dt,*}(\mu)) = - \<\vv\cdot\n_i,\mu\>_{\Gamma_i^T}
  \quad \forall \, \vv \in \V_{h,i}^\Dt, \label{star.1} \\
  & (\dt p_h^{\Dt,*}(\mu),w)_{\Omega_i^T}
  - b^T_i(\u_h^{\Dt,*}(\mu),w) = 0
  \quad \forall \, w \in W_{h,i}^\Dt. \label{star.2}
\end{align}\end{subequations}
Note that both~\eqref{bar} and~\eqref{star} are posed in the individual space-time subdomains $\Omega_i^T$ and can thus be solved in parallel (on the entire space-time subdomains $\Omega_i^T$, without any synchronization on time steps).
It is easy to check that~\eqref{method} is equivalent to
solving the space-time interface problem: find $\lambda_H^\DT \in \Lambda_{H}^\DT$ such that
\begin{equation}\label{interface}
  - b^T_{\Gamma}(\u_h^{\Dt,*}(\lambda_H^\DT),\mu) = b^T_{\Gamma}(\overline\u_h^\Dt,\mu)
  \quad \forall \, \mu \in \Lambda_{H}^\DT,
  \end{equation}
and obtaining $\u_h^\Dt$ and $p_h^\Dt$ from~\eqref{split}--\eqref{star}.

\subsection{Space-time Steklov--Poincar\'e operator}\label{sec_SP}

The above problem can be written in an operator form: find $\lambda_H^\DT \in \Lambda_{H}^\DT$ such that
\begin{equation}\label{int-operator}
  S \, \lambda_H^\DT = g,
\end{equation}
where $S:\Lambda_{H}^\DT \to \Lambda_{H}^\DT$ is the space-time
Steklov--Poincar\'e operator defined as
\begin{equation}\label{S-defn}
  \<S \lambda,\mu\>_{\Gamma^T} = \sum_i \<S_i\lambda,\mu\>_{\Gamma_i^T},
  \quad \<S_i \lambda, \mu\>_{\Gamma_i^T} =
  - \<\u_h^{\Dt,*}(\lambda)\cdot\n_i,\mu\>_{\Gamma_i^T} \quad
  \forall \, \lambda,\mu \in \Lambda_{H}^\DT,
\end{equation}
and $g \in \Lambda_{H}^\DT$ is defined as
$\<g,\mu\>_{\Gamma^T} = b_{\Gamma}(\overline\u_h^\Dt,\mu) \,
\forall \, \mu \in \Lambda_{H}^\DT$.

\begin{lemma}[Space-time Steklov--Poincar\'e operator]\label{pos-def}
  Assume that conditions~\eqref{mortar-ass} hold. Then the operator $S$
  defined in~\eqref{S-defn} is positive definite.
\end{lemma}
\begin{proof}
For $\lambda, \mu \in \Lambda_H^\DT$, consider~\eqref{star.1} with data $\mu$ and test function
$\vv = \u_h^{\Dt,*}(\lambda)$. This implies, using~\eqref{S-defn},
\begin{equation}\label{S-charact}\begin{split}
  \<S \lambda,\mu\>_{\Gamma^T} & = a^T(\u_h^{\Dt,*}(\mu),\u_h^{\Dt,*}(\lambda)) + b^T(\u_h^{\Dt,*}(\lambda),p_h^{\Dt,*}(\mu)) \\
  & = a^T(\u_h^{\Dt,*}(\mu),\u_h^{\Dt,*}(\lambda))
  + (\dt p_h^{\Dt,*}(\lambda),p_h^{\Dt,*}(\mu))_{\Omega^T},
\end{split}\end{equation}
where we have used~\eqref{star.2} with data $\lambda$ and test function $p_h^{\Dt,*}(\mu)$ in the second equality. Lemma~\ref{lem:dt-positive} together with
$p_h^{\Dt,*}(\mu)(x,0) = 0$ (recall that zero initial data is supposed in~\eqref{star}) imply that
\begin{equation}\label{S-pos-def}
  \<S \mu,\mu\>_{\Gamma^T} \ge a^T(\u_h^{\Dt,*}(\mu),\u_h^{\Dt,*}(\mu)) \ge 0 \ \
  \forall \, \mu \in \Lambda_H^\DT,
\end{equation}
hence $S$ is positive semi-definite.
Assume that $\<S \mu,\mu\>_{\Gamma^T} = 0$. Then $\u_h^{\Dt,*}(\mu) = 0$. The inf--sup
condition for the weakly continuous velocity~\eqref{inf--sup-weak-cont} and
\eqref{star.1} imply $p_h^{\Dt,*}(\mu) = 0$. Then the mortar inf--sup
condition~\eqref{inf--sup} and~\eqref{star.1} imply $\mu = 0$, thus
$S$ is positive definite.
  \end{proof}
Due to Lemma~\ref{pos-def}, GMRES can be employed to solve the
interface problem~\eqref{int-operator}. On each GMRES iteration, the
dominant computational cost is the evaluation of the action of $S$,
which requires solving space-time problems with prescribed Dirichlet
interface data in each individual space-time subdomain $\Omega_i \times (0,T)$. The following result can
be used to provide a bound on the number of GMRES iterations.

\begin{theorem}[Spectral bound]\label{thm:specbound}
  Assume that conditions~\eqref{mortar-ass} hold. Let $\T_{h,i}$ be quasi-uniform and $h \le C h_i \ \forall i$. Then there exist positive constants
  $C_0$ and $C_1$ independent of the mesh sizes $h$, $H$, $\Dt$, and $\Delta T$, such that
\begin{equation}\label{spectral-bound}
  \forall \, \mu \in \Lambda_H^\DT, \quad
  C_0 \|\mu\|_{\Gamma^T}^2 \le \<S \mu,\mu\>_{\Gamma^T} \le C_1 h^{-1}
  \|\mu\|_{\Gamma^T}^2.
\end{equation}
\end{theorem}
\begin{proof}
  Using~\eqref{S-defn}, the Cauchy--Schwarz inequality, \eqref{trace}, and
  $h \le C h_i$, we obtain
\begin{align*}
  \<S_i \mu,\mu\>_{\Gamma_i^T} \le
  \|\u_h^{\Dt,*}(\mu)\cdot\n_i\|_{\Gamma_i^T} \|\mu\|_{\Gamma_i^T}
  \le C h^{-\frac12}
  \|\u_h^{\Dt,*}(\mu)\|_{\Omega_i^T} \|\mu\|_{\Gamma_i^T}
  \le C h^{-\frac12} \<S_i \mu,\mu\>_{\Gamma_i^T}^{\frac12}\|\mu\|_{\Gamma_i^T},
\end{align*}
where we used~\eqref{S-pos-def}, which is also valid on each $\Omega_i^T$, in the last inequality. This implies the upper bound in~\eqref{spectral-bound}.

To prove the lower bound in~\eqref{spectral-bound}, we consider the
set of auxiliary subdomain problems~\eqref{aux-problem} with data $\mu$. Let
$\vv_i = \bPi_{h,i}^\Dt \bpsi_i$ and recall that
$\vv_i\cdot\n_i = \Q_{h,i}^\Dt \, \mu$. Using~\eqref{mortar-assumption}
and~\eqref{star.1}, we have
\begin{align*}
  \|\mu\|_{\Gamma^T}^2 & \le
    C \sum_i \<\Q_{h,i}^\Dt \, \mu, \Q_{h,i}^\Dt \, \mu\>_{\Gamma_i^T} =
    C \sum_i \<\Q_{h,i}^\Dt \, \mu, \mu\>_{\Gamma_i^T}
    = C \sum_i \<\vv_i\cdot\n_i, \mu\>_{\Gamma_i^T},\\
    & = -C\sum_i\left(a^T_i(\u_h^{\Dt,*}(\mu),\vv_i)
    + b^T_i(\vv_i,p_h^{\Dt,*}(\mu)) \right)\\
    & \le C \sum_i \left(\| \u_h^{\Dt,*}(\mu)\|_{\Omega_i^T}^2 +
    \| p_h^{\Dt,*}(\mu)\|_{\Omega_i^T}^2 \right)^{\frac12} \|\vv_i\|_{L^2(0,T;\V_{i})} \\
    & \le C \Bigg\{\sum_i \| \u_h^{\Dt,*}(\mu)\|_{\Omega_i^T}^2\Bigg\}^{\frac12} \Bigg\{ \sum_i \|\mu\|_{\Gamma_i^T}^2 \Bigg\}^{\frac12}\\
    & \le C \<S \mu,\mu\>_{\Gamma^T}^{\frac12} \|\mu\|_{\Gamma^T} .
\end{align*}
In the next to last inequality above, we used the Cauchy--Schwarz inequality together with the inf--sup condition~\eqref{inf--sup-weak-cont} and~\eqref{star.1} to bound $\| p_h^{\Dt,*}(\mu)\|_{\Omega^T}$ and the elliptic regularity
\eqref{inf--sup-2} to bound $\|\vv_i\|_{L^2(0,T;\V_{i})}$. In the last
inequality we used~\eqref{S-pos-def}. This concludes the proof.
\end{proof}

\subsection{GMRES convergence through the field-of-values estimates}\label{sec_fov}

Theorem~\ref{thm:specbound} leads to convergence estimates for solving
the interface problem~\eqref{int-operator} with GMRES. In
\cite[Theorem~3.3]{EisElmSch:83}, a bound is shown for the $k$-th
residual $\r_k$ of the generalized conjugate residual method for
solving a system with a positive definite matrix $\S \in \R^{n \times n}$, which also applies to GMRES.  It can be stated in terms of
angle $\beta \in [0, \pi/2)$, see~\cite{beck:2005}:
  \begin{equation}
  \label{eq:elman}
  \|\r_k \| \leq \sin^k(\beta) \| \r_0 \|, \quad \text{where }
  \cos(\beta) = \dfrac{\lambda_{\text{min}}((\S+\S^T)/2)}{\|\S\|},
\end{equation}
where $\|\cdot\|$ denotes the Euclidean vector norm and the induced matrix norm.
The quantities in~\eqref{eq:elman} can be interpreted in terms of the
field-of-values of $\S$, defined as
\[
W(\S) = \{ \bzeta^T \S \, \bzeta: \bzeta \in \bkC^n,
\|\bzeta \| = 1 \}.
\]
It is known (see~\cite[Chapter~15]{HornJohns:94}) that $W(\S)$ is a
compact and convex set in the complex plane that contains (but is
usually much larger than) the eigenvalues of $\S$. Because $\S$ is
positive definite, ${\bf 0} \not \in W(\S)$, and because $\S$ is real,
the smallest eigenvalue of the symmetric part of $\S$ is actually the
distance from ${\bf 0}$ to $W(\S)$, so that the angle $\beta$ can be
improved to, see~\cite{beck:2005} or~\cite[Theorem~2.2.2] {GreenbBook:97},
\[
\cos(\beta) = \dfrac{\text{dist}({\bf 0}, W(\S))}{\| \S \|}.
\]
The above bound, together with inequalities~\eqref{spectral-bound}
obtained in Theorem~\ref{thm:specbound}, imply that
the reduction in the $k$-th GMRES residual for solving the interface
problem~\eqref{int-operator} is bounded by
\begin{equation}
  \label{eq:fovbound}
 \|\r_k \| \leq  \left(\sqrt{1- (C_0/C_1)^2 h^2}\right)^k\, \|\r_0 \|.
\end{equation}
A similar inequality, allowing for an explicit preconditioning matrix,
has been obtained in~\cite{StarkeFOV:97}.


\section{Numerical results}\label{sec_num}

In this section, we present several numerical results obtained with 
the space-time mortar method developed in Section~\ref{sec_MMFE}, illustrating the convergence rates and other theoretical results obtained in the
previous sections. The method is implemented using the deal.II finite element package~\cite{BangerthHartmannKanschat2007}.

In all the examples, we consider two-dimensional spatial domains and 
take the mixed finite element spaces ${\bf {V}}_{h,i}\times W_{h,i}$
on the spatial subdomain $\Omega_{i}$ to be the lowest order Raviart--Thomas pair $RT_{0}\times DGQ_{0}$ (\ie, $k=l=0$) on quadrilateral meshes \cite{BF},
where $DGQ_r$ denotes the space of discontinuous piecewise polynomials of degree up to $r$ in each variable. Combining this with the lowest-order DG (backward Euler, $q=0$) for time discretization
on the mesh $\mathcal{T}_{i}^\Dt$ gives us a space-time mixed
finite element space ${\bf {V}}_{h,i}^\Dt\times W_{h,i}^\Dt$
in $\Omega_{i}^{T}$ as detailed in Section~\ref{sec_grids_spaces}.
We test two different choices for the mortar finite
element space $\Lambda_{H,ij}^{\DT}$ on the space-time interface mesh $\mathcal{T}_{H,ij}^{\DT}$, with $\DT$ suitably chosen as a function of $\Dt$, depending on the mortar space polynomial degree. These are discontinuous bilinear $DGQ_{1}$ ($m=s=1$) and biquadratic $DGQ_{2}$ ($m=s=2$) mortars.

For solving the interface problem identified in Section~\ref{sec_SP}, we have
implemented the GMRES algorithm without preconditioner. Developing
a preconditioner for the iterative solver, which could significantly
reduce the number of iterations, and its theoretical analysis is a subject
of future research.

\subsection{Example 1: convergence test}\label{num_model}

In this example, we solve the parabolic problem~\eqref{parabolic} with a known solution to verify the accuracy of the space-time mortar method. We also discuss the correspondence of the number of interface GMRES iterations to the theoretical estimate. We further compare the accuracy and computational cost with $DGQ_{1}$ and $DGQ_{2}$ mortar spaces. We use the known pressure function $p(x,y,t)=\sin(8t)\sin(11x)\cos(11y-\frac{\pi}{4})$
along with permeability $K=I_{2\times2}$ to determine the right-hand side $q$ in~\eqref{parabolic} and impose Dirichlet boundary condition and initial condition on the space-time domain $\Omega^{T}=(0,1)^{2}\times(0,0.5)$.

\renewcommand{\tabcolsep}{4.7pt}

\begin{table}
	\caption{Example 1, mesh size and number of degrees of freedom \label{tab:Example-1:-Mesh}}
\centering{}%
	\begin{tabular}{c|c|c|c|c|c|c|c|c|c|c|c|c|c|c|c|c|c|c}
		\hline
		{\small{}Ref.} & \multicolumn{3}{c|}{{\small{}$\Omega^{T}_{1}$}} & \multicolumn{3}{c|}{{\small{}$\Omega^{T}_{2}$}} & \multicolumn{3}{c|}{{\small{}$\Omega^{T}_{3}$}} & \multicolumn{3}{c|}{{\small{}$\Omega^{T}_{4}$}} & \multicolumn{3}{c|}{{\small{}$\Gamma^{T}$$(m=1)$}} & \multicolumn{3}{c}{{\small{}$\Gamma^{T}$$(m=2)$}}\tabularnewline
		\cline{2-19}
		{\small{}No.} & {\small{}$n_1$} & {\small{}$N_1$} & {\tiny{}\#DoF} & {\small{}$n_2$} & {\small{}$N_2$} & {\tiny{}\#DoF} & {\small{}$n_3$} & {\small{}$N_3$} & {\tiny{}\#DoF} & {\small{}$n_4$} & {\small{}$N_4$} & {\tiny{}\#DoF} & {\small{}$n_\Gamma$} & {\small{}$N_\Gamma$} & {\tiny{}\#DoF} & {\small{}$n_\Gamma$} & {\small{}$N_\Gamma$} & {\tiny{}\#DoF}\tabularnewline
		\hline
		{\small{}0} & {\small{}3} & {\small{}3} & {\small{}33} & {\small{}2} & {\small{}2} & {\small{}16} & {\small{}4} & {\small{}4} & {\small{}56} & {\small{}3} & {\small{}3} & {\small{}33} & {\small{}1} & {\small{}1} & {\small{}16} & {\small{}1} & {\small{}1} & {\small{}36}\tabularnewline
		\hline
		{\small{}1} & {\small{}6} & {\small{}6} & {\small{}120} & {\small{}4} & {\small{}4} & {\small{}56} & {\small{}8} & {\small{}8} & {\small{}208} & {\small{}6} & {\small{}6} & {\small{}120} & {\small{}2} & {\small{}2} & {\small{}64} &  &  & \tabularnewline
		\hline
		{\small{}2} & {\small{}12} & {\small{}12} & {\small{}456} & {\small{}8} & {\small{}8} & {\small{}208} & {\small{}16} & {\small{}16} & {\small{}800} & {\small{}12} & {\small{}12} & {\small{}456} & {\small{}4} & {\small{}4} & {\small{}256} & {\small{}2} & {\small{}2} & {\small{}144}\tabularnewline
		\hline
		{\small{}3} & {\small{}24} & {\small{}24} & {\small{}1776} & {\small{}16} & {\small{}16} & {\small{}800} & {\small{}32} & {\small{}32} & {\small{}3136} & {\small{}24} & {\small{}24} & {\small{}1776} & {\small{}8} & {\small{}8} & {\small{}1024} &  &  & \tabularnewline
		\hline
		{\small{}4} & {\small{}48} & {\small{}48} & {\small{}7008} & {\small{}32} & {\small{}32} & {\small{}3136} & {\small{}64} & {\small{}64} & {\small{}12416} & {\small{}48} & {\small{}48} & {\small{}7008} & {\small{}16} & {\small{}16} & {\small{}4096} & {\small{}4} & {\small{}4} & {\small{}576}\tabularnewline
		\hline
	\end{tabular}
\end{table}

\begin{table}
\caption{Example 1, convergence with bilinear mortars\label{tab:Linear-mortar-convergence}}
	
\centering{}%
\begin{tabular}{c|c|c|c|c|c|c|c|c|c|c}
\hline
{\small{}Ref.} & \multicolumn{2}{c|}{{\small{}\# GMRES}}
&
\multicolumn{2}{c|}
{{\small{}$\|\u-\u_h^{\Dt}\|_{L^{2}(0,T;{\bf L}^{2}(\Omega))}$}}
& \multicolumn{2}{c|}{{\small{}$\|p-p_h^{\Dt}\|_{\DG}$}}
& \multicolumn{2}{c|}{{\small{$\|p-p_h^{\Dt}\|_{L^{2}(0,T;W)}$}}}
& \multicolumn{2}{c}{{\small{}$\|\lambda-\lambda_H^\DT\|_{L^2(0,T;\Lambda_H)}$}}\tabularnewline
		\hline
		{\small{}0} & {\small{}11 } & {\small{}Rate} & {\small{}6.50e-01 } & {\small{}Rate} & {\small{}1.21e+00 } & {\small{}Rate} & {\small{}7.91e-01 } & {\small{}Rate} & {\small{}7.98e-01 } & {\small{}Rate}\tabularnewline
		\hline
		{\small{}1} & {\small{}23 } & {\small{}-1.06 } & {\small{}3.63e-01 } & {\small{}0.84 } & {\small{}7.21e-01 } & {\small{}0.75 } & {\small{}4.76e-01 } & {\small{}0.73 } & {\small{}5.11e-01 } & {\small{}0.64}\tabularnewline
		\hline
		{\small{}2} & {\small{}39 } & {\small{}-0.76 } & {\small{}1.74e-01 } & {\small{}1.06 } & {\small{}3.19e-01 } & {\small{}1.18 } & {\small{}2.46e-01 } & {\small{}0.95 } & {\small{}2.34e-01 } & {\small{}1.13}\tabularnewline
		\hline
		{\small{}3} & {\small{}59 } & {\small{}-0.60 } & {\small{}8.63e-02 } & {\small{}1.02 } & {\small{}1.46e-01 } & {\small{}1.13 } & {\small{}1.25e-01 } & {\small{}0.98 } & {\small{}1.20e-01 } & {\small{}0.96}\tabularnewline
		\hline
		{\small{}4} & {\small{}86 } & {\small{}-0.54 } & {\small{}4.29e-02 } & {\small{}1.01 } & {\small{}6.93e-02 } & {\small{}1.08 } & {\small{}6.25e-02 } & {\small{}1.00 } & {\small{}6.11e-02 } & {\small{}0.97}\tabularnewline
		\hline
	\end{tabular}
\end{table}

\begin{table}
	\caption{Example 1, convergence with biquadratic mortars \label{tab:Quadratic-mortar-convergence}}
	
	\centering{}%
	\begin{tabular}{c|c|c|c|c|c|c|c|c|c|c}
		\hline
		{\small{}Ref.} & \multicolumn{2}{c|}{{\small{}\# GMRES}} & \multicolumn{2}{c|}{{\small{}$\|\u-\u_h^{\Dt}\|_{L^{2}(0,T;{\bf L}^{2}(\Omega))}$}} & \multicolumn{2}{c|}{{\small{}$\|p-p_h^{\Dt}\|_{\DG}$}} & \multicolumn{2}{c|}{{\small{}$\|p-p_h^{\Dt}\|_{L^{2}(0,T;W)}$}} & \multicolumn{2}{c}{{\small{}$\|\lambda-\lambda_H^\DT\|_{L^2(0,T;\Lambda_H)}$}}\tabularnewline
		\hline
		{\small{}0} & {\small{}18 } & {\small{} Rate} & {\small{}6.81e-01 } & {\small{}Rate } & {\small{}1.35e+00 } & {\small{}Rate} & {\small{}8.39e-01 } & {\small{}Rate} & {\small{}2.13e+00 } & {\small{}Rate}\tabularnewline
		\hline
		{\small{}2} & {\small{}34 } & {\small{}-0.46} & {\small{}1.70e-01 } & {\small{}1.00 } & {\small{}3.51e-01 } & {\small{}0.97} & {\small{}2.51e-01 } & {\small{}0.87 } & {\small{}2.82e-01 } & {\small{}1.46}\tabularnewline
		\hline
		{\small{}4} & {\small{}57 } & {\small{}-0.37} & {\small{}4.48e-02 } & {\small{}0.96} & {\small{}8.59e-02 } & {\small{}1.02} & {\small{}6.59e-02 } & {\small{}0.96} & {\small{}9.20e-02 } & {\small{}0.81}\tabularnewline
		\hline
	\end{tabular}
\end{table}

\begin{figure}
	\begin{centering}
		\subfloat[On the whole space-time domain $\Omega^{T}$]{\includegraphics[width=0.8\columnwidth]{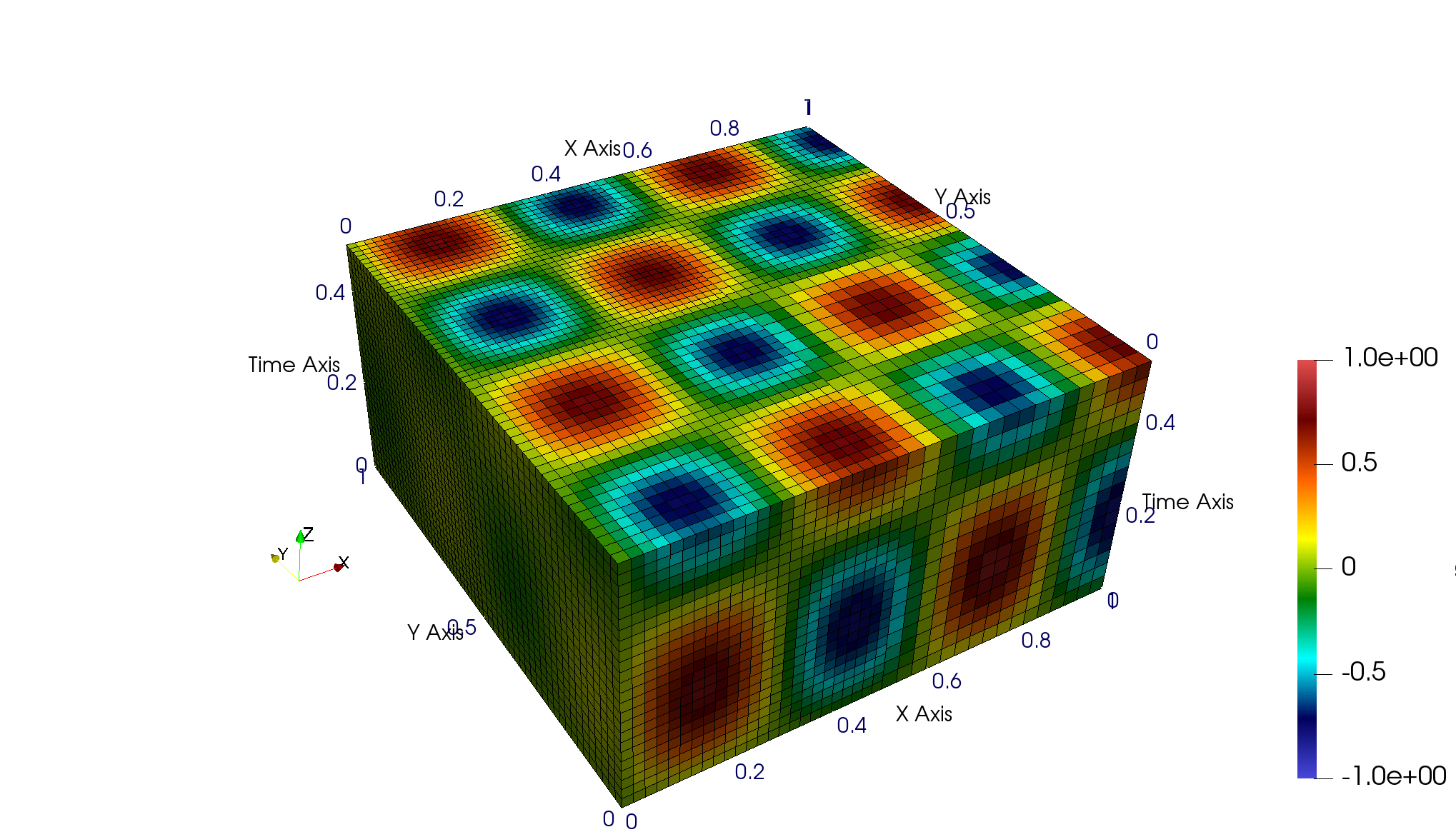}
			
		}
		\par\end{centering}
	\begin{centering}
		\subfloat[On $\Omega_{1}^{T}\cup\Omega_{4}^{T}$]{\includegraphics[width=0.47\columnwidth]{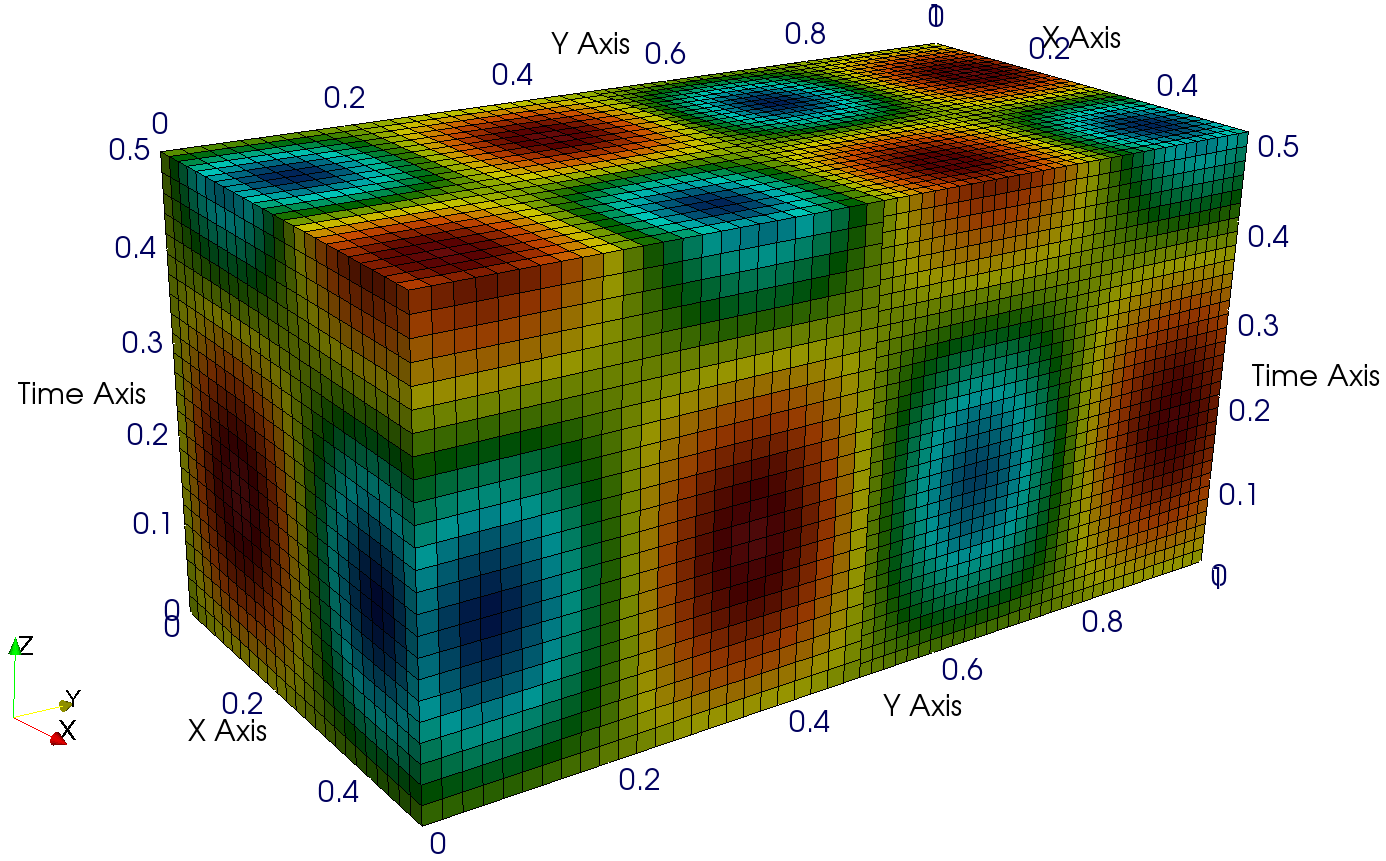}
			
		}\subfloat[On $\Omega_{2}^{T}\cup\Omega_{3}^{T}$]{\includegraphics[width=0.53\columnwidth]{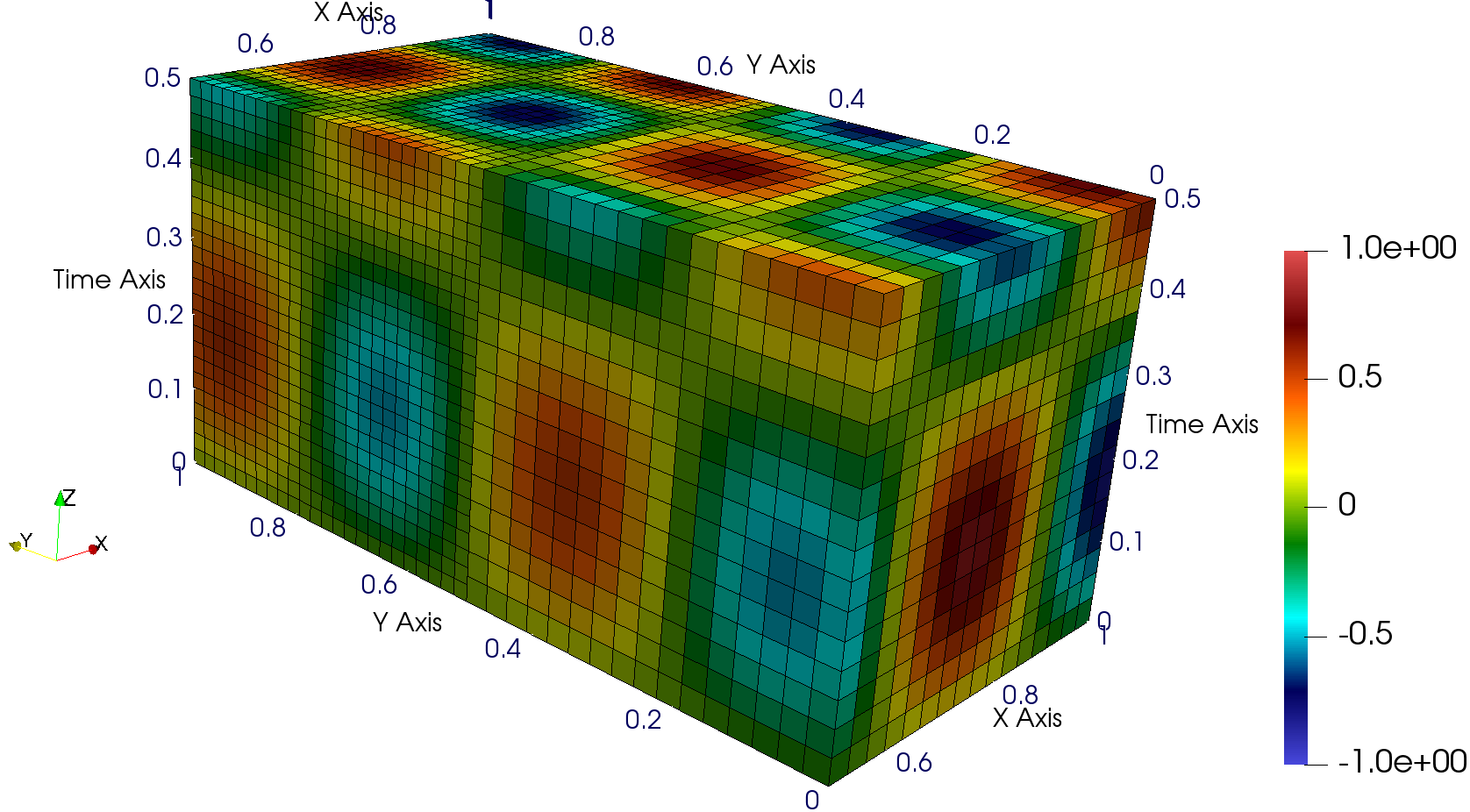}}
		\par\end{centering}
	\caption{Example 1, pressure computed using bilinear mortars shown on the
		space-time grid at refinement 3. \label{fig:Example-1:-Pressure}}
\end{figure}

\begin{figure}
	\begin{centering}
		\subfloat[On $\Omega_{1}^{T}\cup\Omega_{4}^{T}$]{\includegraphics[width=0.47\columnwidth]{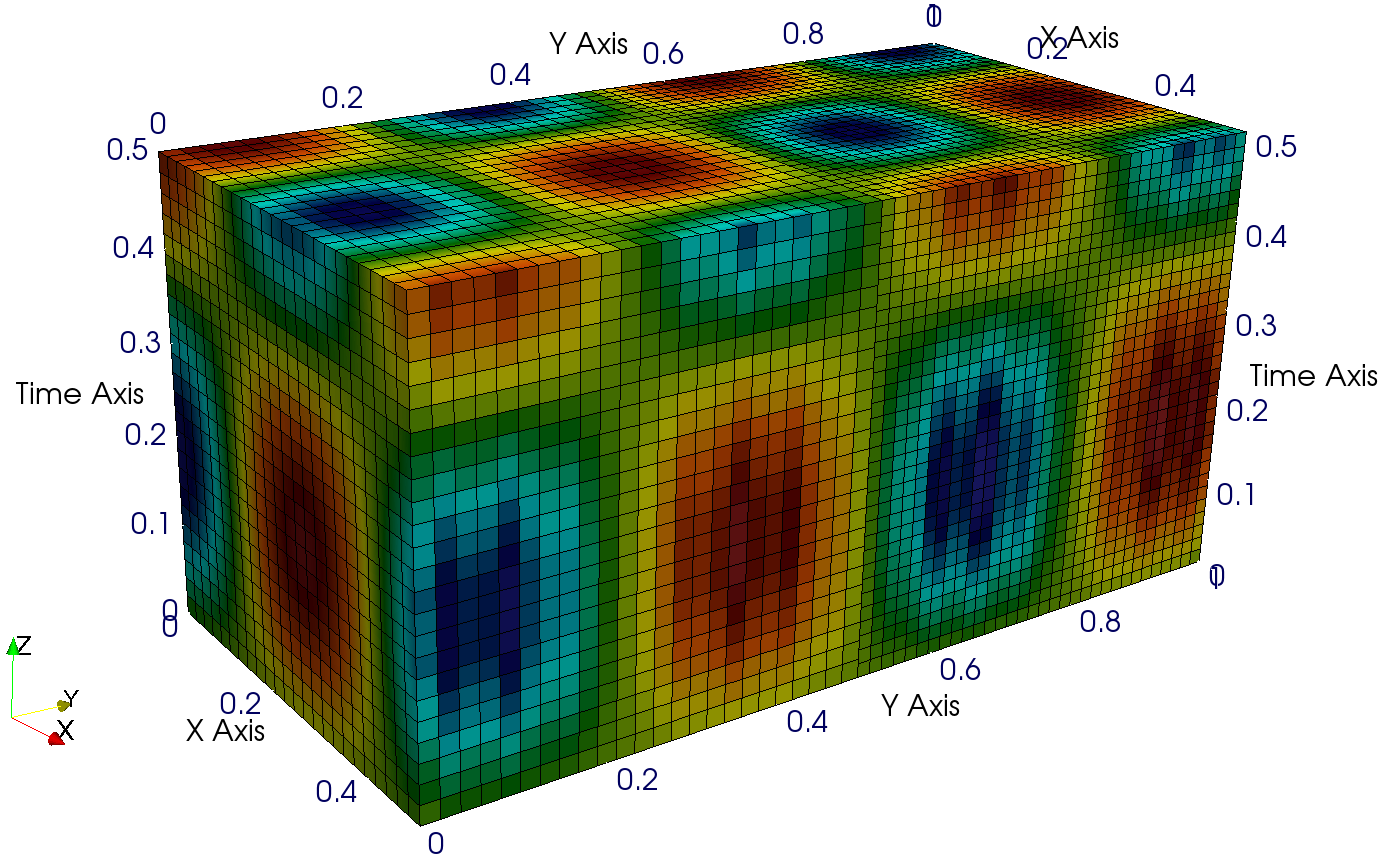}
			
		}\subfloat[On $\Omega_{2}^{T}\cup\Omega_{3}^{T}$]{\includegraphics[width=0.53\columnwidth]{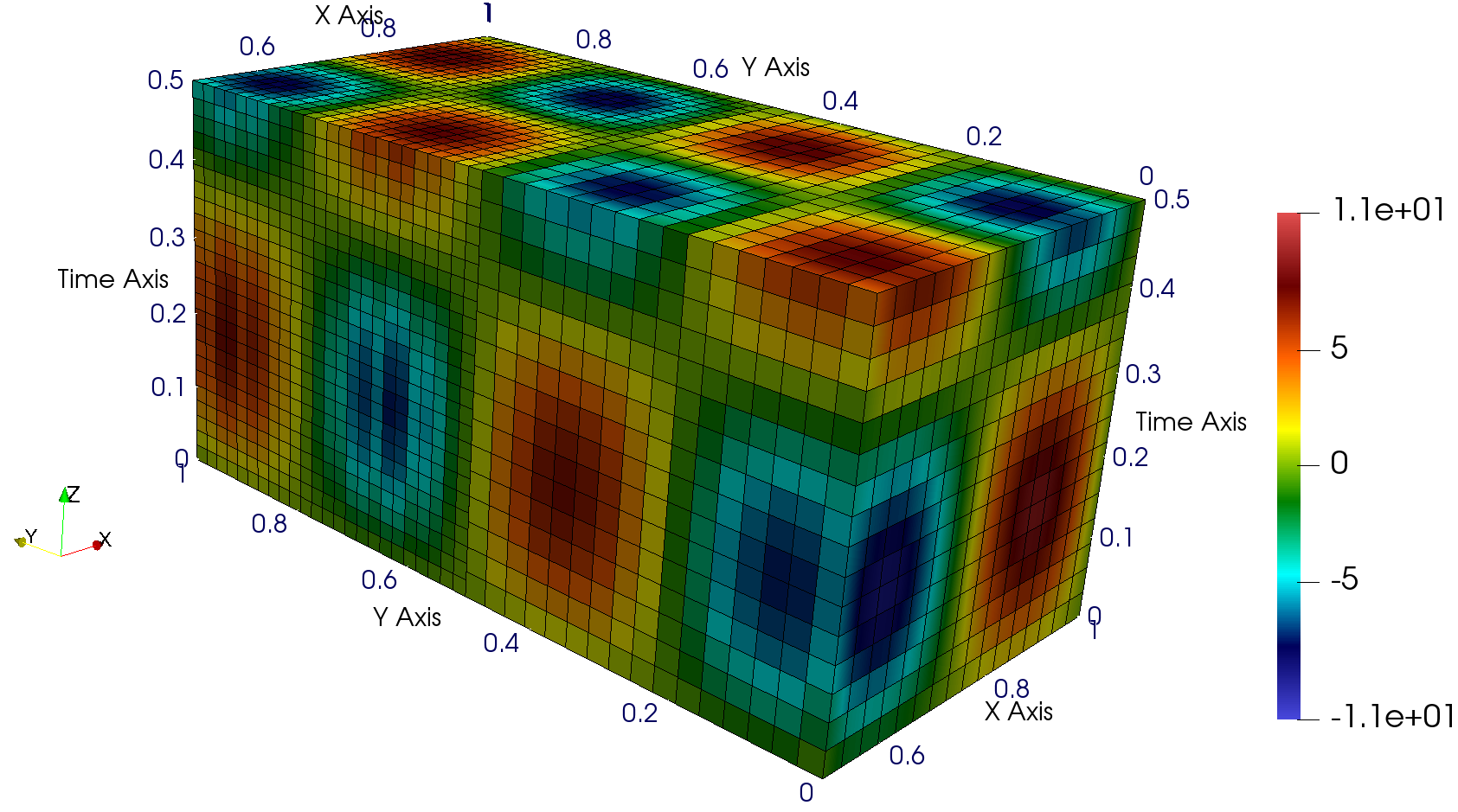}}
		\par\end{centering}
	\caption{Example 1, $x$-component of velocity computed using bilinear
		mortars shown on the space-time grid at refinement 3.}
\end{figure}

\begin{figure}
	\begin{centering}
		\subfloat[On $\Omega_{1}^{T}\cup\Omega_{4}^{T}$]{\includegraphics[width=0.47\columnwidth]{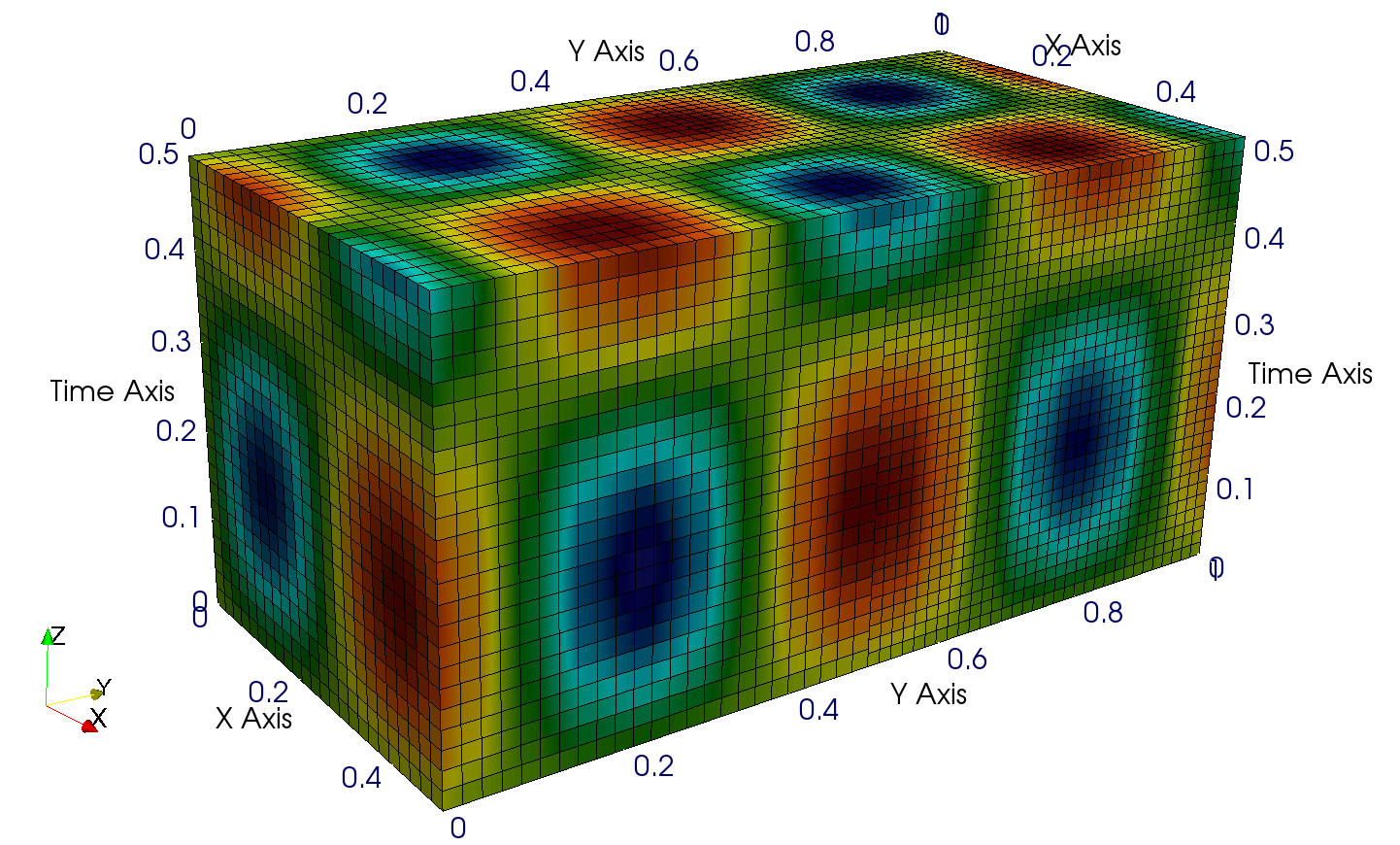}
			
		}\subfloat[On $\Omega_{2}^{T}\cup\Omega_{3}^{T}$]{\includegraphics[width=0.53\columnwidth]{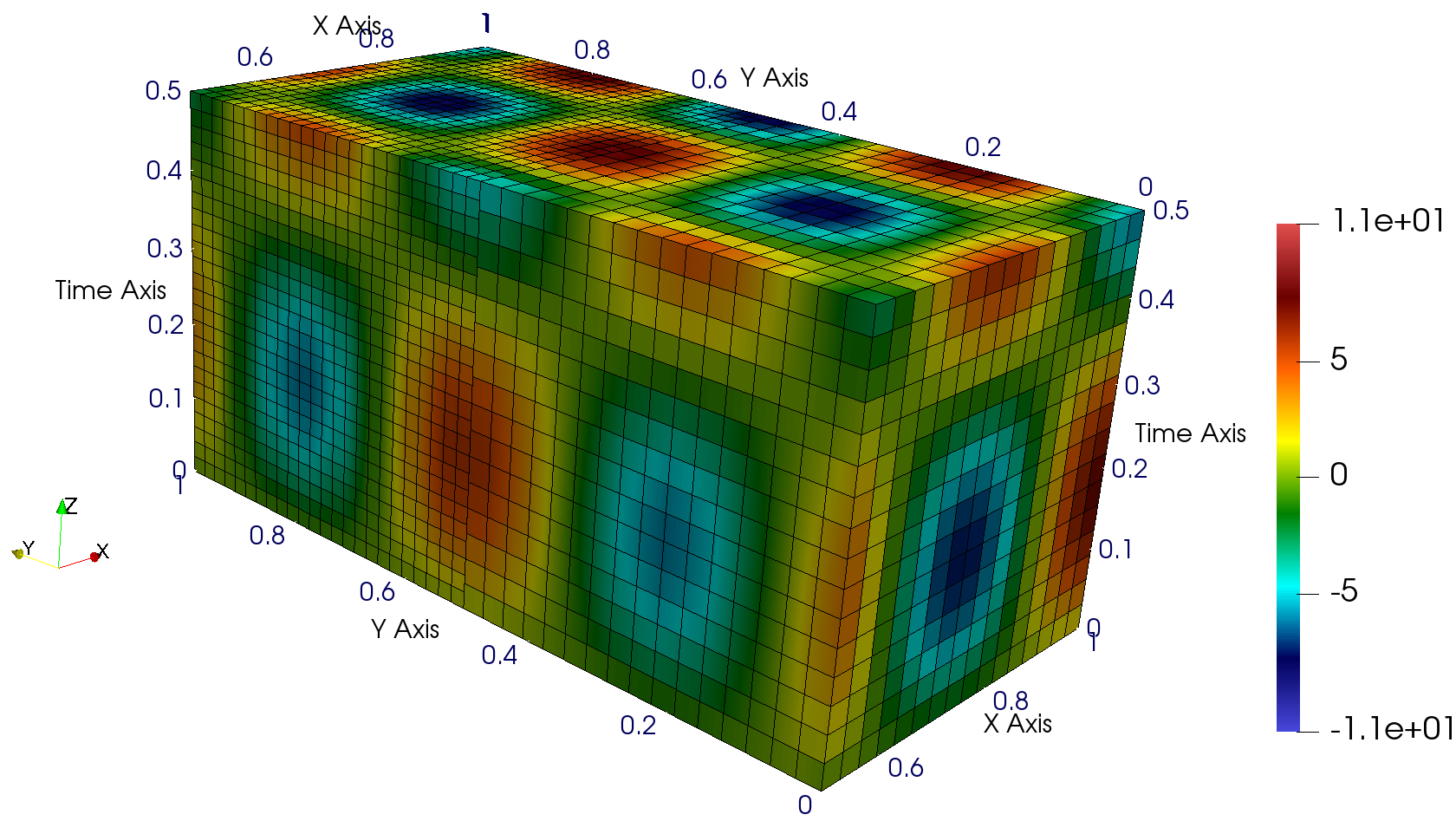}}
		\par\end{centering}
	\caption{Example 1, $y$-component of velocity computed using bilinear
		mortars shown on the space-time grid at refinement 3. \label{fig:Example-1:-ycomponent}}
\end{figure}

We partition the space domain $\Omega$ into four identical
squares $\Omega_i$, $i=1,2,3,4$, with $\Omega_1 = (0,0.5)\times(0,0.5)$, $\Omega_2 = (0.5,1)\times(0,0.5)$, $\Omega_3 = (0,0.5)\times(0.5,1)$, and
$\Omega_4 = (0.5,1)\times(0.5,1)$. The space-time domain $\Omega^{T}$ is
correspondingly partitioned into four space-time subdomains $\Omega_{i}^{T}$, $i=1,2,3,4$.
We start with an initial grid for each $\Omega_{i}^{T}$ and $\Gamma_{ij}^{T}$
and refine it successively 4 times to test the convergence rate of
the solutions with respect to the actual known solution. The subdomains $\Omega_{i}^{T}$ maintain a checkerboard non-matching mesh structure throughout the refinement cycles. In particular, let $n_i$ be number of elements in either the $x$ or $y$-directions and let $N_i$ be the number of elements in the $t$-direction in subdomain $\Omega_{i}^{T}$. The initial grids are chosen as $n_1 = N_1 = 3$, $n_2 = N_2 = 2$, $n_3 = N_3 = 4$, and $n_4 = N_4 = 3$, see Table~\ref{tab:Example-1:-Mesh}. Note that $h = \Dt$.
In the case of bilinear mortars, we maintain $H=2h$ and $\Delta T=2\Dt$, halving the mesh sizes on each refinement cycle. For biquadratic mortars, we start with $H=2h$ and $\Delta T=2\Dt$, but refine the mortar mesh only every other time to maintain $H=\sqrt{h}$ and $\Delta T=\sqrt{\Dt}$.
The coarser mesh on $\Gamma^{T}$ in the biquadratic mortar case $DGQ_{2}$ is compensated by the higher degree of the space. In particular, the last term on the right hand side in the bound \eqref{error-estimate} from Theorem~\ref{conv-thm} gives $\cO(h^{-\frac12}H^{m+1})$. With $m=1$ and $H = \cO(h)$, this results in $\cO(h^{\frac32})$, while with $m=2$ and $H = \cO(h^{\frac12})$, it gives $\cO(h)$. Similarly, the last term in the bound
\eqref{error-estimate-tilde} from Theorem~\ref{conv-thm-tilde} gives
$\cO(H^{m+\frac12})$. With $m=1$ and $H = \cO(h)$,
this results in $\cO(h^{\frac32})$, while with $m=2$ and $H = \cO(h^{\frac12})$, it gives $\cO(h^\frac54)$. In all cases, the order is no smaller than the optimal order $\cO(h)$ with respect to the $RT_{0}\times DGQ_{0}$ finite element
spaces. More details on the mesh refinement are given in Table~\ref{tab:Example-1:-Mesh}.
There we also report the number of spatial degrees of freedom of the spaces $RT_{0}\times DGQ_{0}$ on $\Omega_{i}$, as well as the number of space-time degrees of freedom of the mortar space $DGQ_{1}$ or $DGQ_{2}$
on $\Gamma^{T}$.

In Tables~\ref{tab:Linear-mortar-convergence} and~\ref{tab:Quadratic-mortar-convergence} we report the 
relative errors with respect to the norm of the true solution, as well as the convergence rates as powers of the subdomain discretization parameters
$h$ and $\Dt$. We observe optimal first order convergence of the method
using both bilinear and biquadratic mortars. We note that the 
$\Dt^{-\frac12}$ loss in convergence rate in the bound from Theorem~\ref{conv-thm} is not observed in the numerical results. In Tables~\ref{tab:Linear-mortar-convergence} and~\ref{tab:Quadratic-mortar-convergence} we also report the
growth rate for the number of GMRES iterations in the case of bilinear and biquadratic mortars, respectively. We recall that Theorem \ref{thm:specbound}
bounds the spectral ratio of the interface operator $S$ by
$\mathcal{O}(h^{-1})$. Thus, up to deviation from a normal matrix, the
growth rate is expected to be $\mathcal{O}(h^{-0.5})$ \cite{kelley1995iterative,IpsenGMRES}. This is close to what we observe in Tables~\ref{tab:Linear-mortar-convergence} and~\ref{tab:Quadratic-mortar-convergence}.

We further compare the performance of bilinear and biquadratic mortars.
As expected, the accuracy of the two cases at the same refinement level is comparable, which is evident from Tables~\ref{tab:Linear-mortar-convergence} and~\ref{tab:Quadratic-mortar-convergence}. On the other hand, since the biquadratic mortar space $DGQ_{2}$ 
has far fewer degrees of freedom compared to bilinear mortar
space $DGQ_{1}$, the former results in a smaller number of GMRES iterations. Thus, choosing higher mortar degrees $m,s$ with a coarser mortar mesh results
in a computationally more efficient method compared to using smaller $m,s$ and a finer mortar mesh.

Finally, the computed solution is presented in three-dimensional space-time plots
in  Figures~\ref{fig:Example-1:-Pressure}--\ref{fig:Example-1:-ycomponent}. Note that the $z$-axis corresponds to time direction $t$. The plots clearly show the continuity of pressure and velocity, which is imposed in a weak sense, across the space-time interfaces. 

\subsection{Example 2: problem with a boundary layer}\label{sec:7.2} \label{multiscale_biot_example}

In this example, we demonstrate the advantages of applying the multiscale mortar space-time domain
decomposition method to a problem where the solution variables, pressure and velocity, vary on
different scales across the space-time domain. For this, we use the known solution,
$p(x,y,t)=1000xyt \, e^{-10\left(x^{2}+y^{2}+\frac{1}{4}t^{2}\right)}$ along with permeability
$K=I_{2\times2}$ to determine the right-hand side $q$ in~\eqref{parabolic} and impose Dirichlet boundary condition and initial condition on the space-time domain $\Omega^{T}=(0,1)^{2}\times(0,0.5)$. By
construction, $p(x,y,t)$ varies rapidly in both space and time along the lower-left corner of $\Omega^{T}$, forming a sharp boundary layer. The pressure decays exponentially away from this corner and is close to zero over a large part of the domain. This calls for an efficient multiscale method that
would take advantage of the multiscale nature of the problem and provide better resolution around the lower-left corner compared to the rest of $\Omega^{T}$.

We partition $\Omega^{T}$ into $4\times4$ identical space-time subdomains
$\Omega^{T}_{i}$. From the knowledge about the variation of the true pressure, we design a multiscale space-time grid on $\Omega^{T}$, where the refinement of the grid on each $\Omega^{T}_{i}$ is proportional to the amount of pressure variation. The finest mesh on $\Omega^{T}_{i}$ has
$h = 1/128$ and $\Dt = 1/64$, and the coarsest mesh on $\Omega^{T}_{i}$ has
$h = 1/8$ and $\Dt = 1/8$, see Figures
\ref{fig:ST:Example-2:pressure_left_right}--\ref{fig:ST:Example-2:pressure_top_bottom} for the
mesh refinement. The coarser meshes on the majority of the space-time subdomains reduce the computational complexity of the subdomain solves associated with them. We use a linear mortar ($m=s=1$) on the subdomain interfaces. The mortar mesh sizes in space are chosen as follows. For vertical interfaces (fixed $x$) between subdomains on the bottom row, the one along the boundary layer, we set $H = 1/32$. For the next row of subdomains we set $H = 1/16$, and for the other two rows, $H = 1/8$. Similarly, for the horizontal interfaces (fixed $y$) between subdomains on the left column we set $H = 1/32$, for the second column, $H = 1/16$, and for the other two columns, $H = 1/8$.
We choose $\Delta T=1/8$ on all interfaces.  These choices guarantee that the mortar assumption \eqref{mortar-assumption} is satisfied and that the dimension of the interface problem is reduced, while at the same time provide suitable resolution to enforce weakly flux continuity across the space-time 
interfaces. For comparison, we solve the problem using a uniformly fine and matching subdomain mesh with
$h=H=1/128$ and $\Dt=\DT=1/64$. A comparison of the number of GMRES iterations and the relative
errors from the multiscale and the fine-scale methods is given in
Table~\ref{table:ST:Example-2:multiscale}. It shows that the
multiscale and the fine-scale solutions attain comparable accuracy.
We observe slightly smaller relative errors in the fine-scale case because of the matching grids and higher resolution
throughout the space-time domain. The slightly higher errors
for the multiscale method are compensated by the cheaper subdomain
solves and smaller interface problem that converges in fewer iterations compared to the fine-scale method.

\renewcommand{\tabcolsep}{4.4pt}
\begin{center}
\begin{table}[h]
  \caption{Example 2, errors and GMRES iterations for the multiscale and fine-scale methods
    \label{table:ST:Example-2:multiscale}}
\begin{centering}
\begin{tabular}{c|c|c|c|c|c}
\hline
\multicolumn{1}{c|}{Method} &
\multicolumn{1}{|c|}{{\small{}\# GMRES}} & \multicolumn{1}{c|}{{\small{}$\|\u-\u_h^{\Dt}\|_{L^{2}(0,T;{\bf L}^{2}(\Omega))}$}} & \multicolumn{1}{c|}{{\small{}$\|p-p_h^{\Dt}\|_{\DG}$}} & \multicolumn{1}{c|}{{\small{$\|p-p_h^{\Dt}\|_{L^{2}(0,T;W)}$}}} & \multicolumn{1}{c}{{\small{}$\|\lambda-\lambda_H^\DT\|_{L^2(0,T;\Lambda_H)}$}}\tabularnewline
\hline
multiscale & 102 & 5.657e-02 & 8.425e-02 & 6.319e-02 & 5.796e-02 \tabularnewline
\hline
fine-scale & 140 & 1.524e-02 & 2.234e-02  & 2.154e-02 & 3.016e-02\tabularnewline
\hline
\end{tabular}
\par\end{centering}
\end{table}
\par\end{center}

\begin{figure}[h]
	\begin{raggedright}
		\includegraphics[width=0.43\columnwidth]{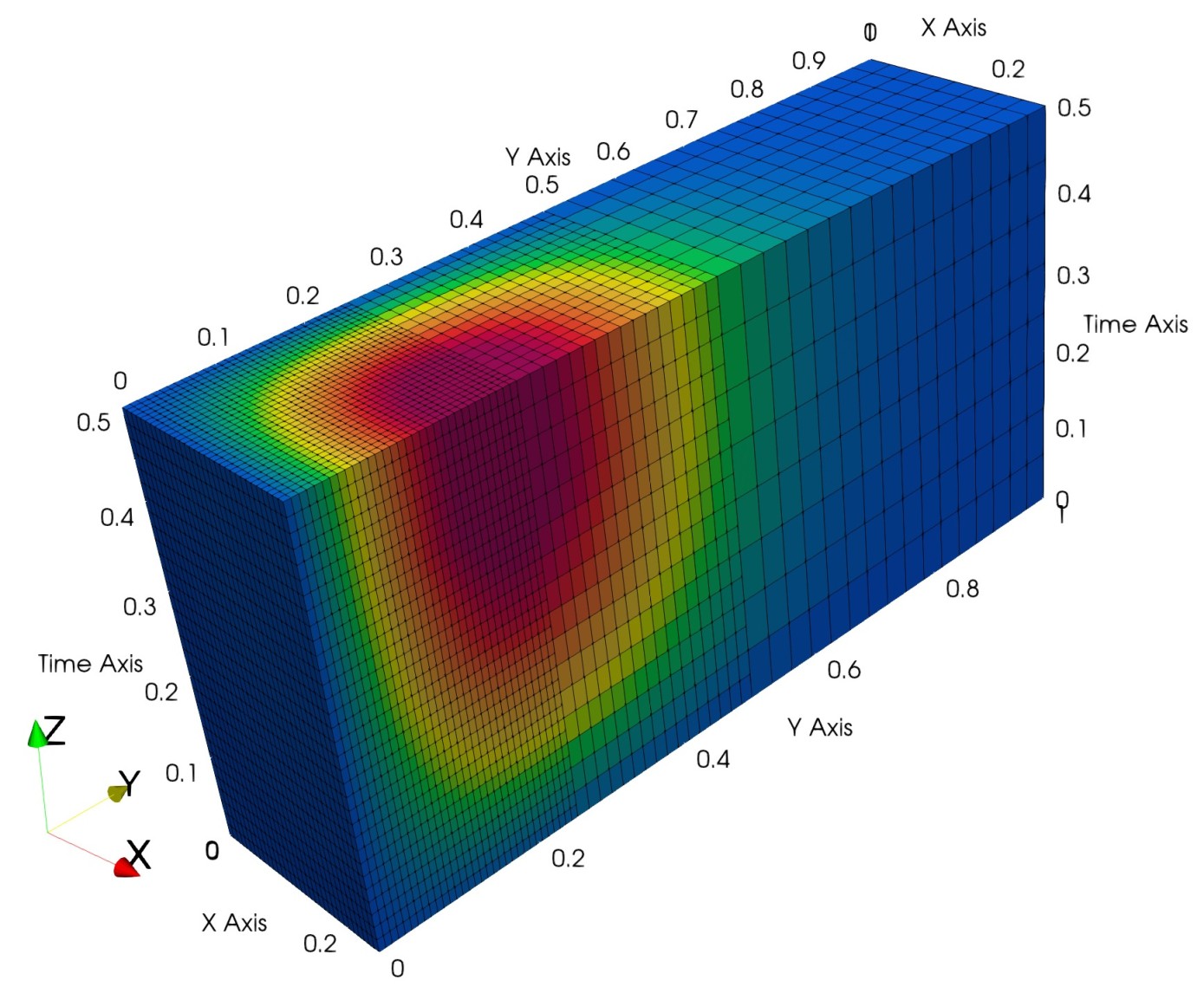}\includegraphics[width=0.57\columnwidth]{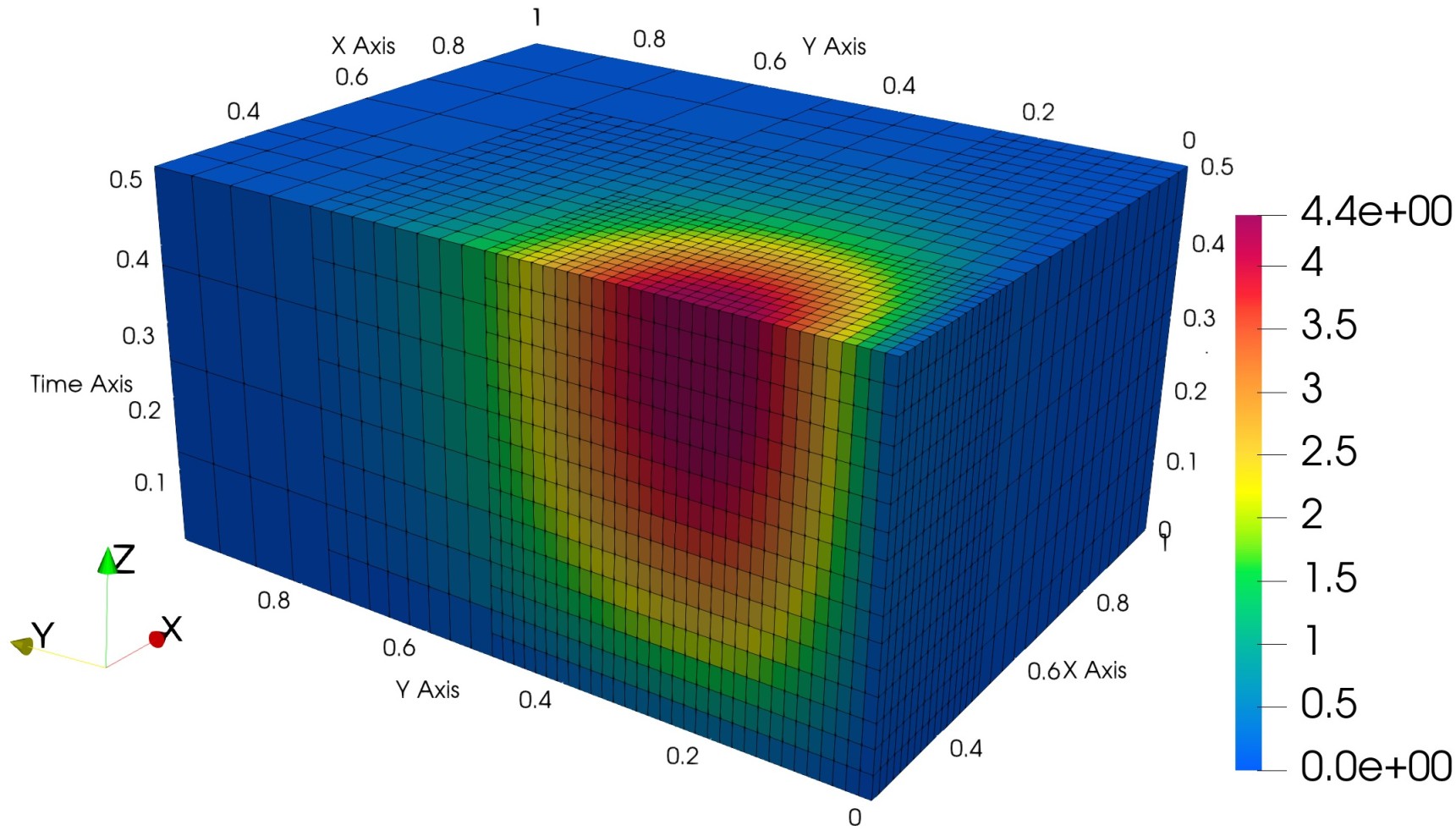}
\par\end{raggedright}
\caption{Example 2, pressure from the multiscale method, cut along the plane $x=0.25$.}
\label{fig:ST:Example-2:pressure_left_right}
\end{figure}

\begin{figure}[h]
\begin{minipage}[c]{0.5\columnwidth}%
		\includegraphics[width=1\columnwidth]{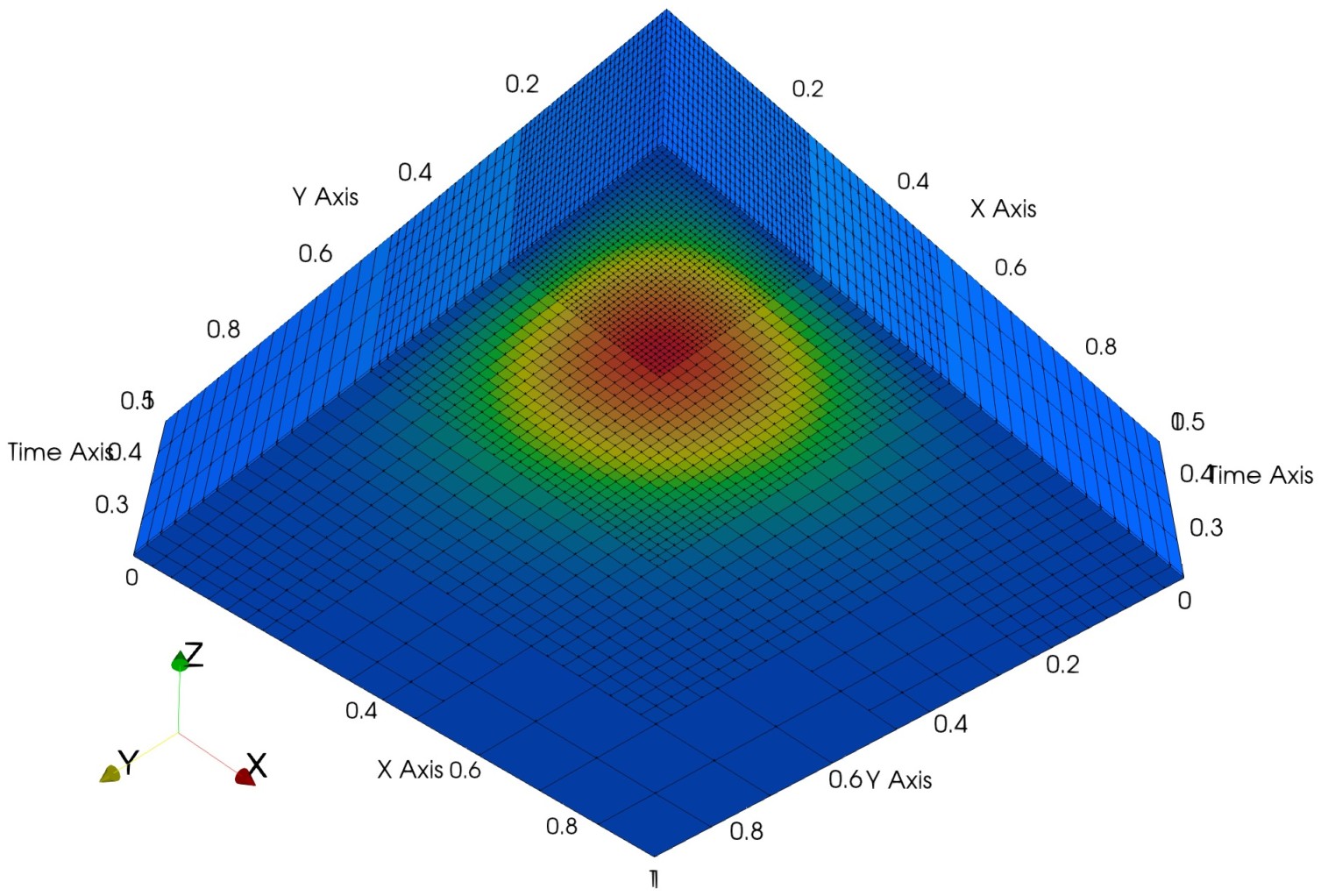}
		
		\includegraphics[width=1\columnwidth]{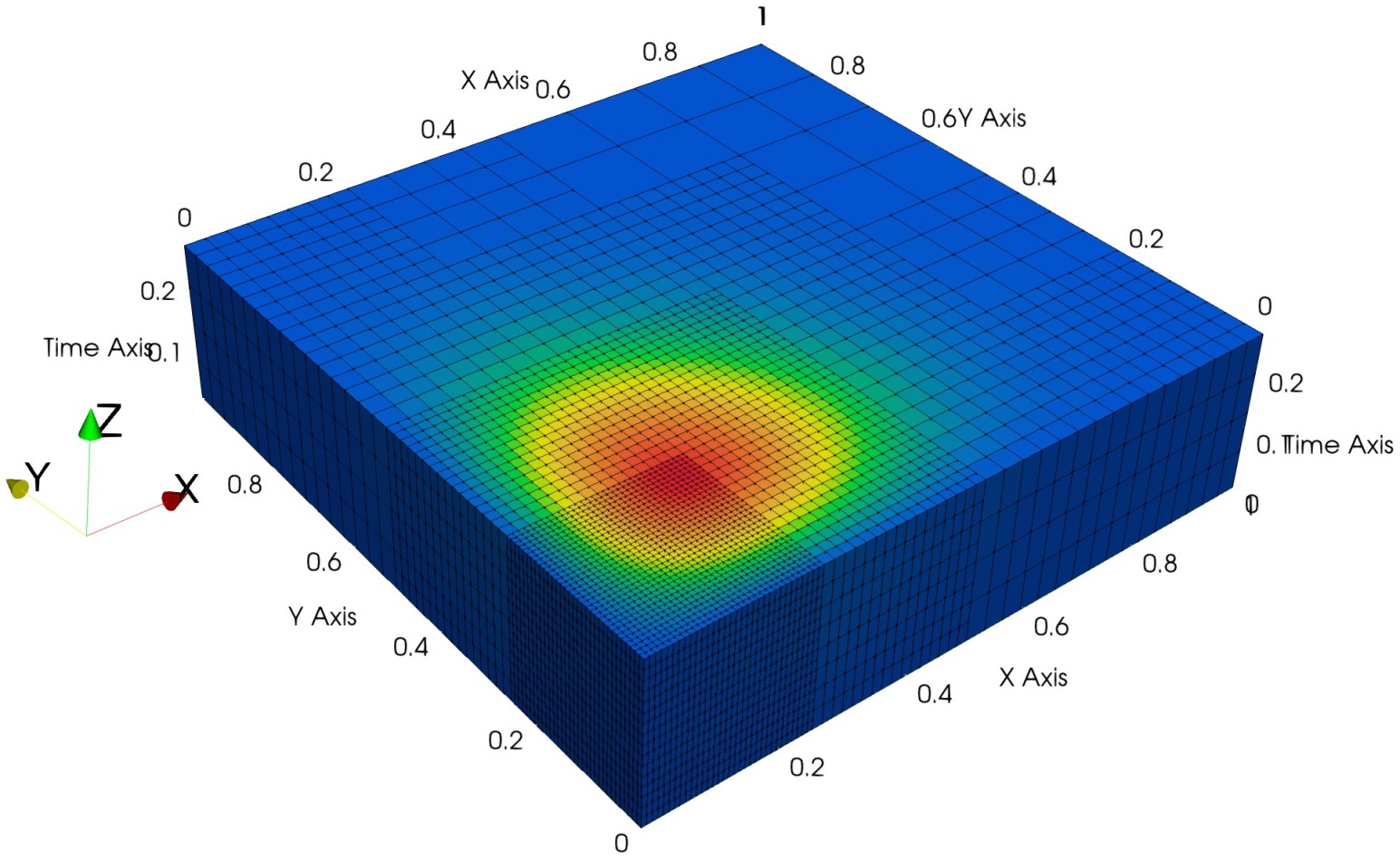}%
\end{minipage}%
\begin{minipage}[c]{0.5\columnwidth}%
\includegraphics[width=1\columnwidth]{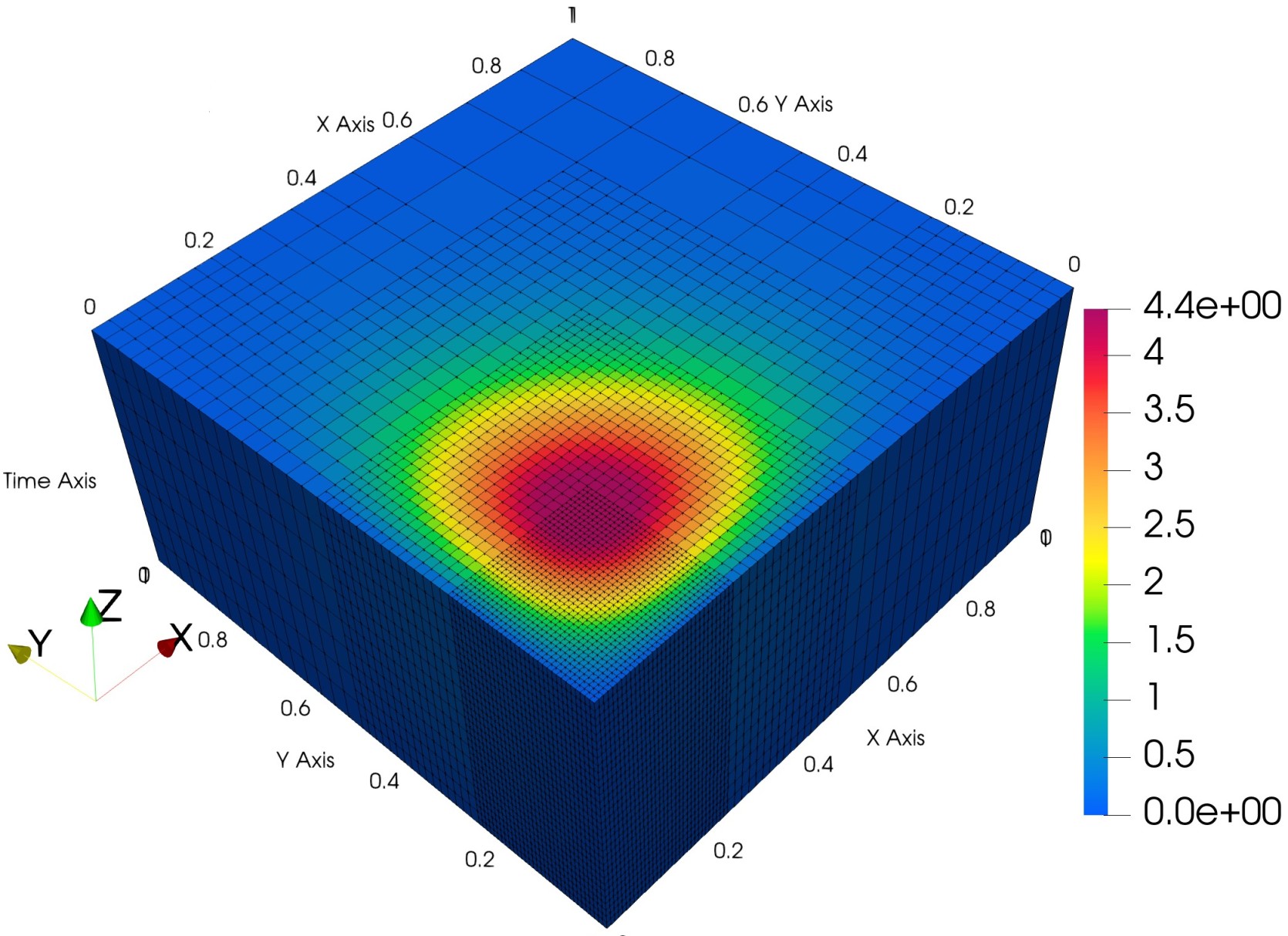}%

\includegraphics[width=1\columnwidth]{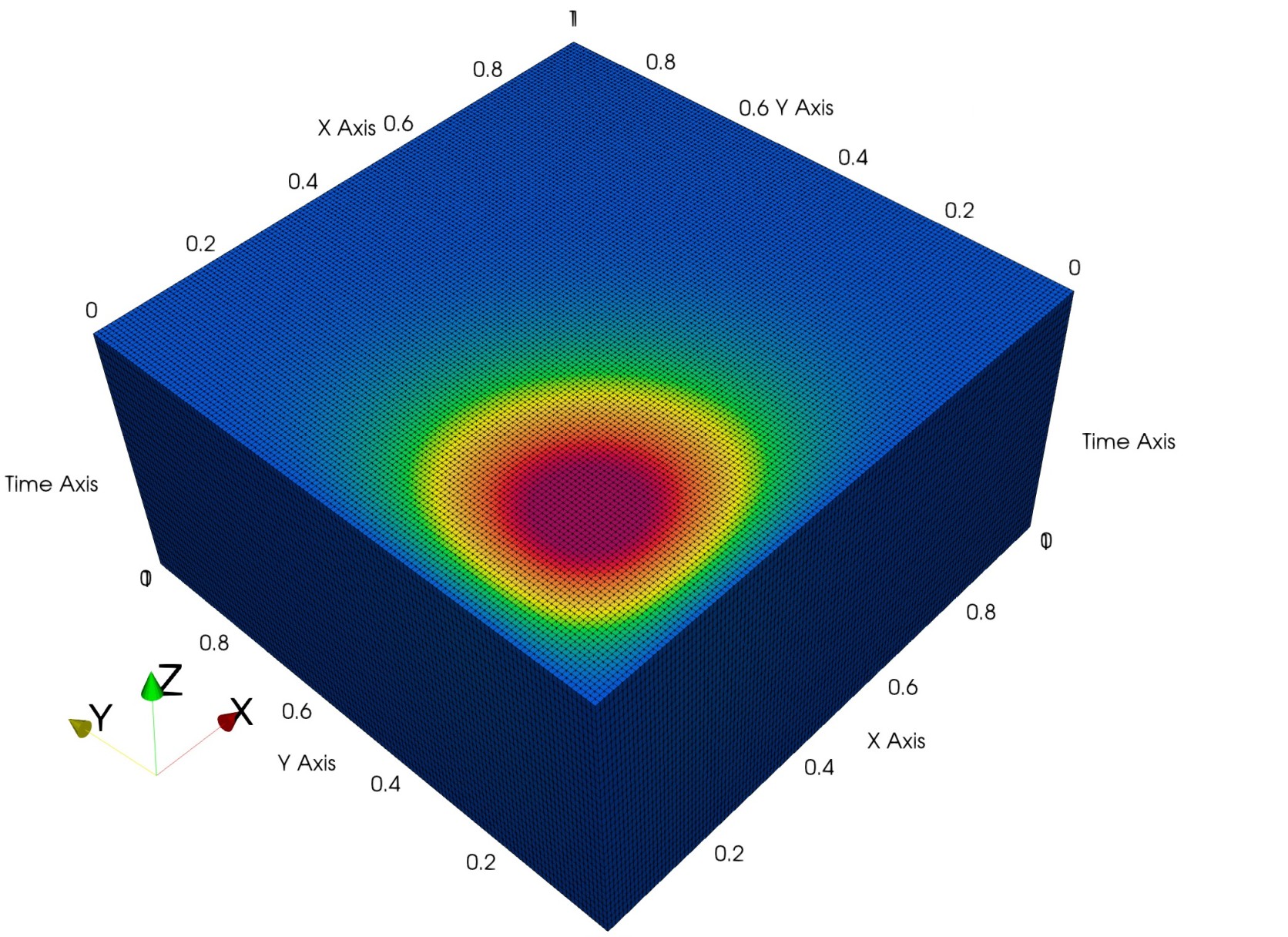}

\end{minipage}
	
\caption{Example 2, left: pressure from the multiscale method, cut along the plane $t=0.35$; right: pressure from the multiscale (top) and fine-scale (bottom) methods on the whole domain.}
\label{fig:ST:Example-2:pressure_top_bottom}
\end{figure}

\begin{figure}[h]
	\begin{raggedright}
		\includegraphics[width=0.42\columnwidth]{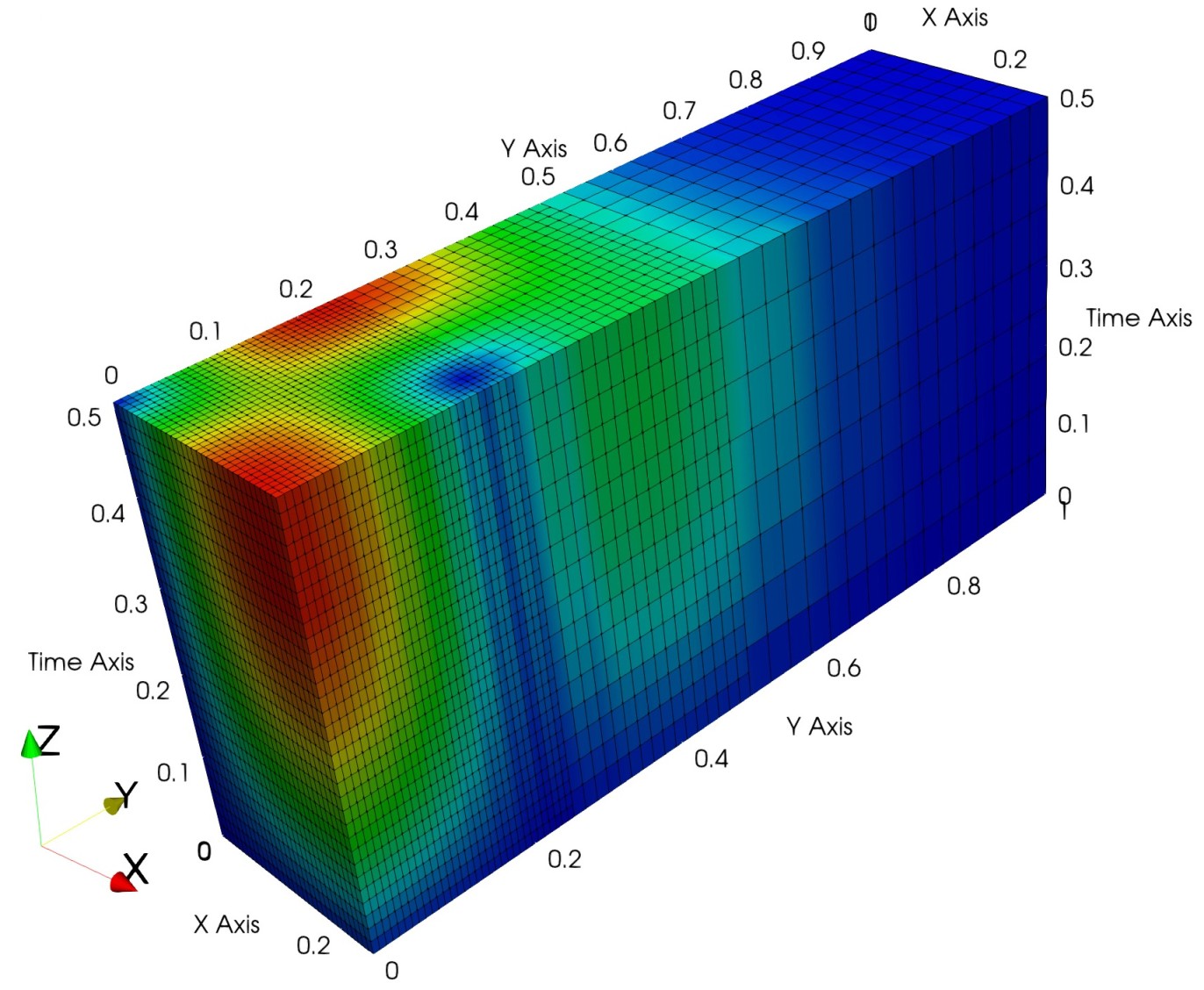}\includegraphics[width=0.55\columnwidth]{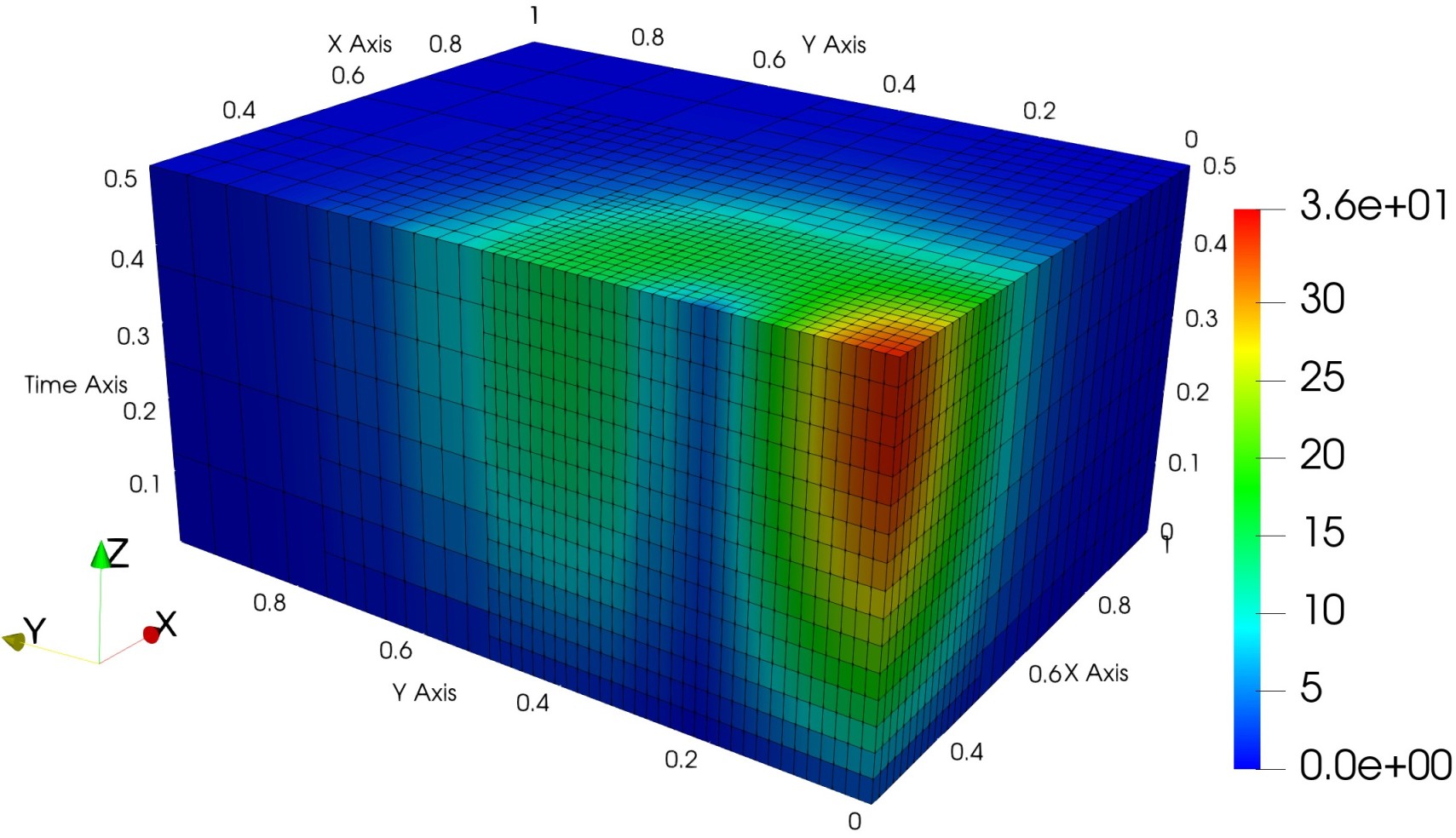}
\par\end{raggedright}
\caption{Example 2, velocity magnitude from the multiscale method, cut along the
  plane $x=0.25$.}
\label{fig:ST:Example-2:velocity_left_right}
\end{figure}

\begin{figure}[h]
	\begin{minipage}[c]{0.5\columnwidth}%
		\includegraphics[width=1\columnwidth]{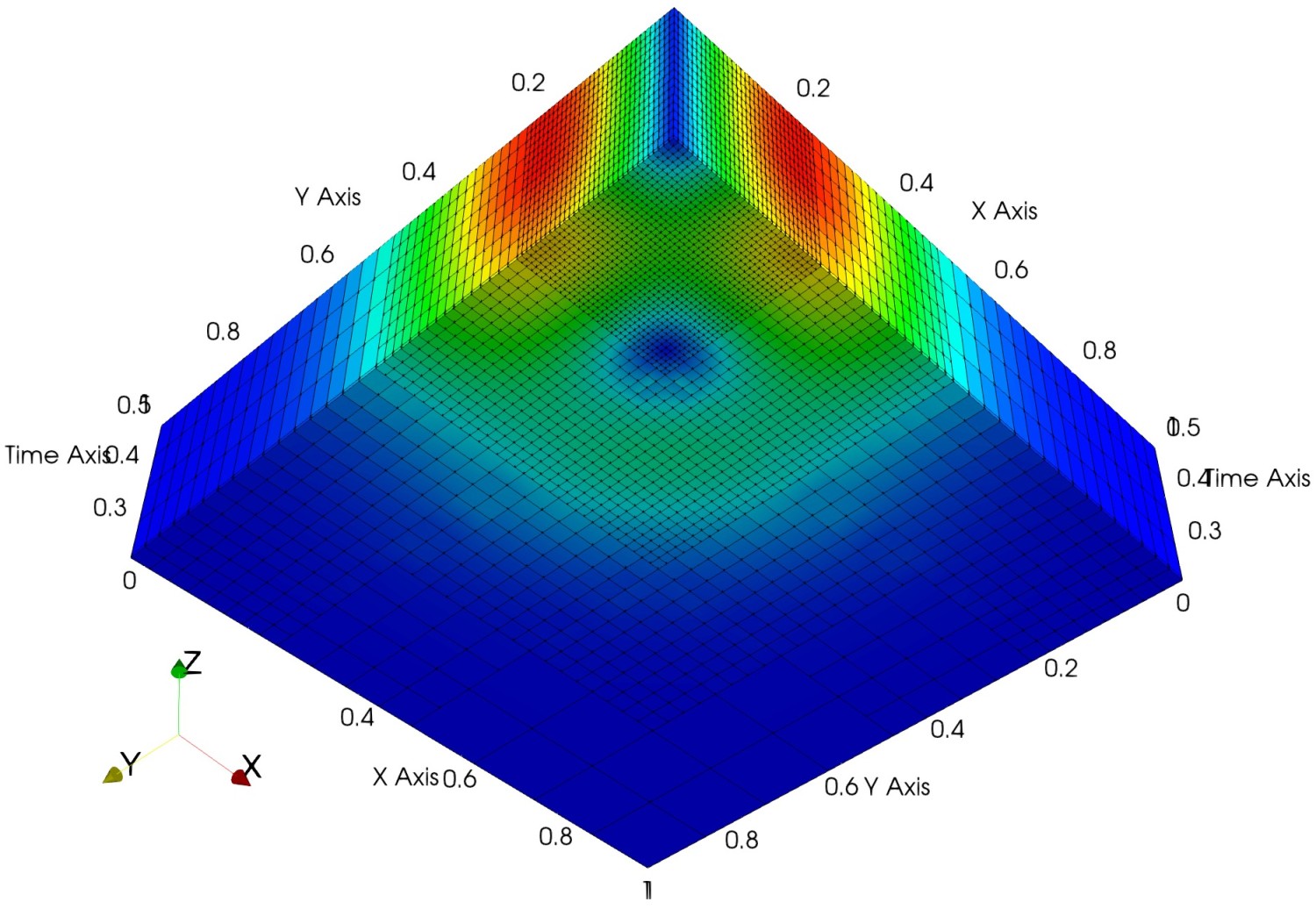}
		
		\includegraphics[width=1\columnwidth]{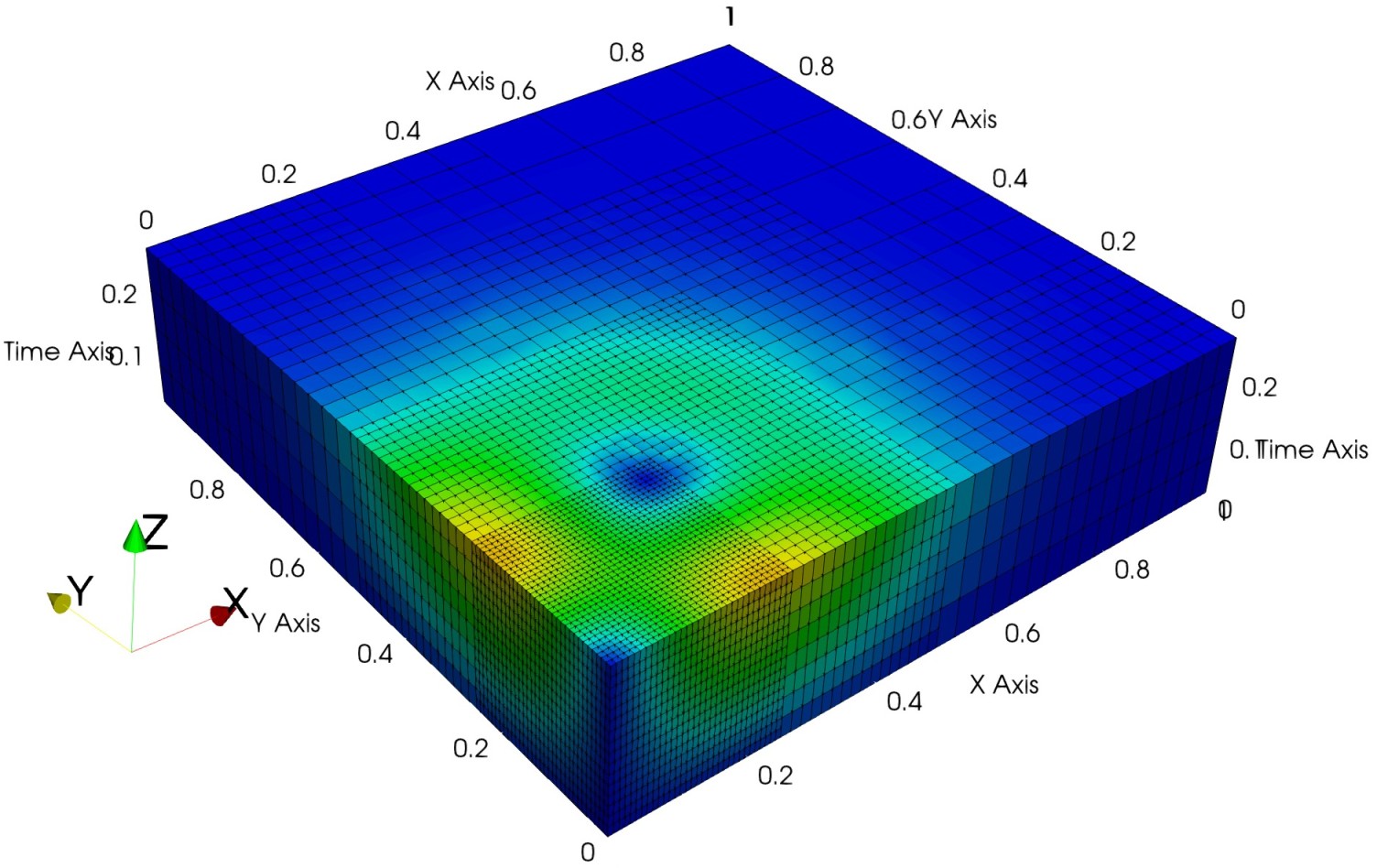}%
\end{minipage}
\begin{minipage}[c]{0.5\columnwidth}%
		\includegraphics[width=1\columnwidth]{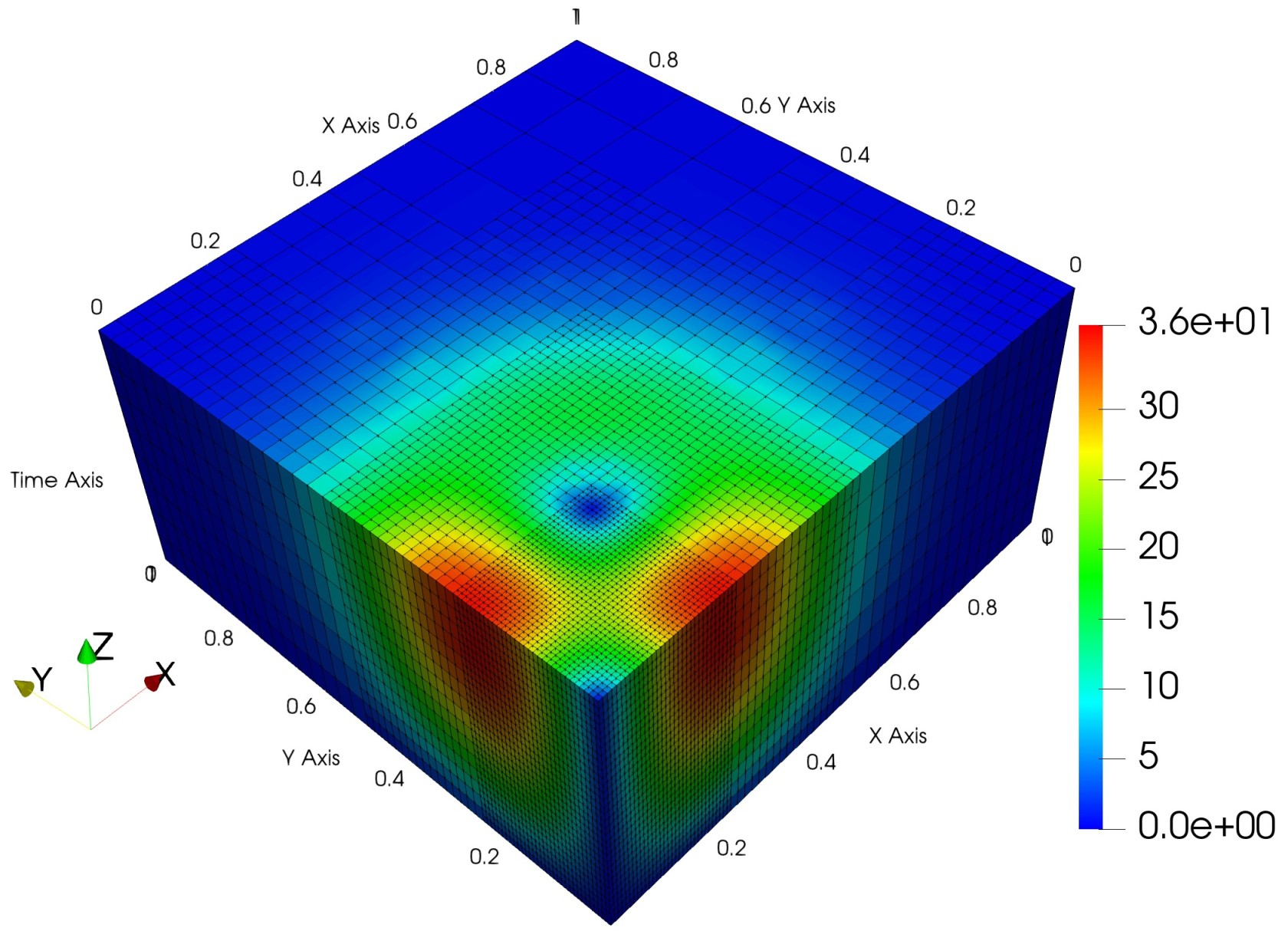}%

\includegraphics[width=1\columnwidth]{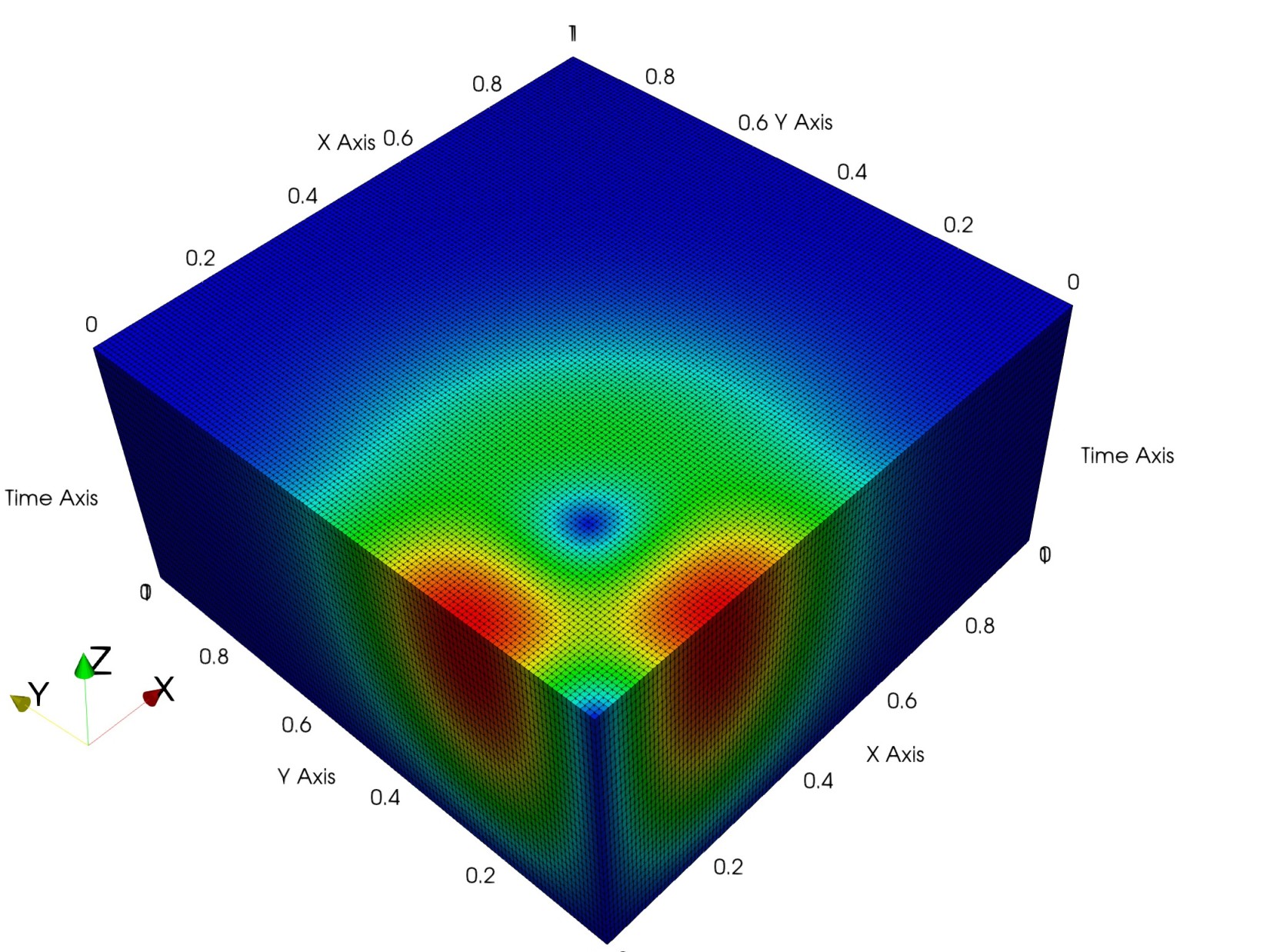}
\end{minipage}
	
\caption{Example 2, left: velocity magnitude from the multiscale method, cut along the plane $t=0.35$; right: velocity magnitude from the multiscale (top) and fine-scale (bottom) methods on the whole domain.}
\label{fig:ST:Example-2:velocity_top_bottom}
\end{figure}

The computed multiscale solution is presented in
Figures~\ref{fig:ST:Example-2:pressure_left_right}--\ref{fig:ST:Example-2:velocity_top_bottom}. The
plots show that the multiscale method provides good resolution where
it matters -- in the regions with high solution variation. Moreover,
we observe very good enforcement of continuity of both pressure and
velocity across various space and time interfaces. Side-to-side
comparisons of the multiscale and fine-scale solutions are given on
the right sides in Figure~\ref{fig:ST:Example-2:pressure_top_bottom}
and Figure~\ref{fig:ST:Example-2:velocity_top_bottom}. They show
excellent agreement between the two solutions and once again confirm
that the less expensive multiscale method provides comparable accuracy
to the more expensive fine-scale method.



\section{Conclusions}\label{sec_concl}

We presented a space-time domain decomposition mixed finite element method for parabolic problems that allows for non-matching spatial grids and local time stepping via space-time mortar finite elements. Well-posedness and a priori error estimates were established. A parallel non-overlapping domain decomposition algorithm was developed, which reduces the algebraic problem to a coarse-scale interface problem for the mortar variable. The theoretical results and the flexibility of the method were illustrated in a series of numerical experiments. Future
work may include the development and analysis of a Neumann--Neumann preconditioner for the interface problem in Section~\ref{sec_red_int} as in~\cite{hoang2013space}, using techniques from \cite{PenYot}, as well as deriving a posteriori error estimates, possibly building upon the ideas from~\cite{Ern_Sme_Voh_heat_HO_Y_17, AliH_Japh_Kern_Voh_DD_MFE_18,Hassan-aposteriori}.

\bibliographystyle{abbrv}
\bibliography{space-time}

\begin{thebibliography}{10}

\bibitem{AliH_Japh_Kern_Voh_DD_MFE_18}
S.~Ali~Hassan, C.~Japhet, M.~Kern, and M.~Vohral{\'{\i}}k.
\newblock A posteriori stopping criteria for optimized {S}chwarz domain
  decomposition algorithms in mixed formulations.
\newblock {\em Comput. Methods Appl. Math.}, 18(3):495--519, 2018.

\bibitem{Hassan-aposteriori}
S.~Ali~Hassan, C.~Japhet, and M.~Vohral\'{\i}k.
\newblock A posteriori stopping criteria for space-time domain decomposition
  for the heat equation in mixed formulations.
\newblock {\em Electron. Trans. Numer. Anal.}, 49:151--181, 2018.

\bibitem{ACWY}
T.~Arbogast, L.~C. Cowsar, M.~F. Wheeler, and I.~Yotov.
\newblock Mixed finite element methods on nonmatching multiblock grids.
\newblock {\em SIAM J. Numer. Anal.}, 37(4):1295--1315, 2000.

\bibitem{Arbogast-aposteriori}
T.~Arbogast, D.~Estep, B.~Sheehan, and S.~Tavener.
\newblock A posteriori error estimates for mixed finite element and finite
  volume methods for parabolic problems coupled through a boundary.
\newblock {\em SIAM/ASA J. Uncertain. Quantif.}, 3(1):169--198, 2015.

\bibitem{APWY}
T.~Arbogast, G.~Pencheva, M.~F. Wheeler, and I.~Yotov.
\newblock A multiscale mortar mixed finite element method.
\newblock {\em Multiscale Model. Simul.}, 6(1):319--346, 2007.

\bibitem{ArshadParkShin}
M.~Arshad, E.-J. Park, and D.~Shin.
\newblock Multiscale mortar mixed domain decomposition approximations of
  nonlinear parabolic equations.
\newblock {\em Comput. Math. Appl.}, 97:375--385, 2021.

\bibitem{BangerthHartmannKanschat2007}
W.~Bangerth, R.~Hartmann, and G.~Kanschat.
\newblock deal.{II}---a general-purpose object-oriented finite element library.
\newblock {\em ACM Trans. Math. Software}, 33(4):Art. 24, 27, 2007.

\bibitem{bause2017error}
M.~Bause, F.~A. Radu, and U.~K{\"o}cher.
\newblock Error analysis for discretizations of parabolic problems using
  continuous finite elements in time and mixed finite elements in space.
\newblock {\em Numer. Math.}, 137(4):773--818, 2017.

\bibitem{beck:2005}
B.~Beckermann, S.~A. Goreinov, and E.~E. Tyrtyshnikov.
\newblock Some remarks on the {E}lman estimate for {GMRES}.
\newblock {\em SIAM J. Matrix Anal. Appl.}, 27(3):772--778, 2005.

\bibitem{Benes}
M.~Bene\v{s}, A.~Nekvinda, and M.~K. Yadav.
\newblock Multi-time-step domain decomposition method with non-matching grids
  for parabolic problems.
\newblock {\em Appl. Math. Comput.}, 267:571--582, 2015.

\bibitem{BF}
F.~Brezzi and M.~Fortin.
\newblock {\em Mixed and hybrid finite element methods}.
\newblock Springer-Verlag, New York, 1991.

\bibitem{Cascon}
J.~M. Casc\'{o}n, L.~Ferragut, and M.~I. Asensio.
\newblock Space-time adaptive algorithm for the mixed parabolic problem.
\newblock {\em Numer. Math.}, 103(3):367--392, 2006.

\bibitem{ciarlet}
P.~G. Ciarlet.
\newblock {\em The finite element method for elliptic problems}.
\newblock North-Holland Publishing Co., Amsterdam-New York-Oxford, 1978.
\newblock Studies in Mathematics and its Applications, Vol. 4.

\bibitem{crouzeix1987stability}
M.~Crouzeix and V.~Thom{\'e}e.
\newblock The stability in {$L_p$ and $W^1_p$ of the $L_2$}-projection onto
  finite element function spaces.
\newblock {\em Math. Comp.}, 48(178):521--532, 1987.

\bibitem{DawDuDup}
C.~N. Dawson, Q.~Du, and T.~F. Dupont.
\newblock A finite difference domain decomposition algorithm for numerical
  solution of the heat equation.
\newblock {\em Math. Comp.}, 57(195):63--71, 1991.

\bibitem{ADM-LTS}
L.~Delpopolo~Carciopolo, M.~Cusini, L.~Formaggia, and H.~Hajibeygi.
\newblock Adaptive multilevel space-time-stepping scheme for transport in
  heterogeneous porous media ({ADM}-{LTS}).
\newblock {\em J. Comput. Phys. X}, 6:100052, 21, 2020.

\bibitem{EisElmSch:83}
S.~C. Eisenstat, H.~C. Elman, and M.~H. Schultz.
\newblock Variational iterative methods for nonsymmetric systems of linear
  equations.
\newblock {\em SIAM J. Numer. Anal.}, 20(2):345--357, 1983.

\bibitem{Ern_Sme_Voh_heat_HO_Y_17}
A.~Ern, I.~Smears, and M.~Vohral{\'{\i}}k.
\newblock Guaranteed, locally space-time efficient, and polynomial-degree
  robust a posteriori error estimates for high-order discretizations of
  parabolic problems.
\newblock {\em SIAM J. Numer. Anal.}, 55(6):2811--2834, 2017.

\bibitem{EwiLazVas90}
R.~E. Ewing, R.~D. Lazarov, and P.~S. Vassilevski.
\newblock Finite difference schemes on grids with local refinement in time and
  space for parabolic problems. {I}. {D}erivation, stability, and error
  analysis.
\newblock {\em Computing}, 45(3):193--215, 1990.

\bibitem{Falgout}
R.~D. Falgout, S.~Friedhoff, T.~V. Kolev, S.~P. MacLachlan, and J.~B. Schroder.
\newblock Parallel time integration with multigrid.
\newblock {\em SIAM J. Sci. Comput.}, 36(6):C635--C661, 2014.

\bibitem{FAUCHER}
V.~Faucher and A.~Combescure.
\newblock A time and space mortar method for coupling linear modal subdomains
  and non-linear subdomains in explicit structural dynamics.
\newblock {\em Comput. Methods Appl. Mech. Engrg.}, 192(5):509--533, 2003.

\bibitem{GaiGloMas}
S.~Gaiffe, R.~Glowinski, and R.~Masson.
\newblock Domain decomposition and splitting methods for mortar mixed finite
  element approximations to parabolic equations.
\newblock {\em Numer. Math.}, 93(1):53--75, 2002.

\bibitem{Gander50year}
M.~J. {Gander}.
\newblock {50 years of time parallel time integration}.
\newblock In {\em {Multiple shooting and time domain decomposition methods.
  MuS-TDD, Heidelberg, Germany, May 6--8, 2013}}, pages 69--113. Cham:
  Springer, 2015.

\bibitem{GanHal}
M.~J. Gander and L.~Halpern.
\newblock Optimized {S}chwarz waveform relaxation methods for advection
  reaction diffusion problems.
\newblock {\em SIAM J. Numer. Anal.}, 45(2):666--697, 2007.

\bibitem{GanderKwokMandal}
M.~J. {Gander}, F.~{Kwok}, and B.~C. {Mandal}.
\newblock {Dirichlet-Neumann and Neumann-Neumann waveform relaxation algorithms
  for parabolic problems}.
\newblock {\em {Electron. Trans. Numer. Anal.}}, 45:424--456, 2016.

\bibitem{GanNeum}
M.~J. {Gander} and M.~{Neum\"uller}.
\newblock {Analysis of a new space-time parallel multigrid algorithm for
  parabolic problems}.
\newblock {\em {SIAM J. Sci. Comput.}}, 38(4):A2173--A2208, 2016.

\bibitem{GanVan}
M.~J. {Gander} and S.~{Vandewalle}.
\newblock {Analysis of the parareal time-parallel time-integration method}.
\newblock {\em {SIAM J. Sci. Comput.}}, 29(2):556--578, 2007.

\bibitem{GANIS20093989}
B.~Ganis and I.~Yotov.
\newblock Implementation of a mortar mixed finite element method using a
  multiscale flux basis.
\newblock {\em Comput. Methods Appl. Mech. Engrg.}, 198(49):3989--3998, 2009.

\bibitem{GreenbBook:97}
A.~Greenbaum.
\newblock {\em Iterative methods for solving linear systems}, volume~17 of {\em
  Frontiers in Applied Mathematics}.
\newblock Society for Industrial and Applied Mathematics (SIAM), Philadelphia,
  PA, 1997.

\bibitem{grisvard2011elliptic}
P.~Grisvard.
\newblock {\em Elliptic problems in nonsmooth domains}, volume~69 of {\em
  Classics in Applied Mathematics}.
\newblock Society for Industrial and Applied Mathematics (SIAM), Philadelphia,
  PA, 2011.

\bibitem{Hager}
C.~Hager, P.~Hauret, P.~Le~Tallec, and B.~I. Wohlmuth.
\newblock Solving dynamic contact problems with local refinement in space and
  time.
\newblock {\em Comput. Methods Appl. Mech. Engrg.}, 201/204:25--41, 2012.

\bibitem{HalJapSze}
L.~Halpern, C.~Japhet, and J.~Szeftel.
\newblock Optimized {S}chwarz waveform relaxation and discontinuous {G}alerkin
  time stepping for heterogeneous problems.
\newblock {\em SIAM J. Numer. Anal.}, 50(5):2588--2611, 2012.

\bibitem{hoang2013space}
T.-T.-P. Hoang, J.~Jaffr\'e, C.~Japhet, M.~Kern, and J.~E. Roberts.
\newblock Space-time domain decomposition methods for diffusion problems in
  mixed formulations.
\newblock {\em SIAM J. Numer. Anal.}, 51(6):3532--3559, 2013.

\bibitem{hoang-space-time-fracture}
T.-T.-P. Hoang, C.~Japhet, M.~Kern, and J.~E. Roberts.
\newblock Space-time domain decomposition for reduced fracture models in mixed
  formulation.
\newblock {\em SIAM J. Numer. Anal.}, 54(1):288--316, 2016.

\bibitem{hoang2021global}
T.-T.-P. Hoang and H.~Lee.
\newblock A global-in-time domain decomposition method for the coupled
  nonlinear {S}tokes and {D}arcy flows.
\newblock {\em J. Sci. Comp.}, 87(1):1--22, 2021.

\bibitem{HornJohns:94}
R.~A. Horn and C.~R. Johnson.
\newblock {\em Topics in matrix analysis}.
\newblock Cambridge University Press, Cambridge, 1994.
\newblock Corrected reprint of the 1991 original.

\bibitem{IpsenGMRES}
I.~C.~F. Ipsen.
\newblock Expressions and bounds for the {GMRES} residual.
\newblock {\em BIT}, 40(3):524--535, 2000.

\bibitem{kelley1995iterative}
C.~T. Kelley.
\newblock {\em Iterative methods for linear and nonlinear equations}, volume~16
  of {\em Frontiers in Applied Mathematics}.
\newblock Society for Industrial and Applied Mathematics, Philadelphia, 1995.

\bibitem{Kheriji-near-well}
W.~Kheriji, R.~Masson, and A.~Moncorg\'{e}.
\newblock Nearwell local space and time refinement in reservoir simulation.
\newblock {\em Math. Comput. Simulation}, 118:273--292, 2015.

\bibitem{KimParkSeo}
D.~Kim, E.-J. Park, and B.~Seo.
\newblock Space-time adaptive methods for the mixed formulation of a linear
  parabolic {P}roblem.
\newblock {\em J. Sci. Comput.}, 74(3):1725--1756, 2018.

\bibitem{Krause}
D.~Krause and R.~Krause.
\newblock Enabling local time stepping in the parallel implicit solution of
  reaction-diffusion equations via space-time finite elements on shallow tree
  meshes.
\newblock {\em Appl. Math. Comput.}, 277:164--179, 2016.

\bibitem{parareal}
J.-L. {Lions}, Y.~{Maday}, and G.~{Turinici}.
\newblock {R\'esolution d'EDP par un sch\'ema en temps ``parar\'eel''}.
\newblock {\em {C. R. Acad. Sci., Paris, S\'er. I, Math.}}, 332(7):661--668,
  2001.

\bibitem{lions2011non}
J.-L. Lions and E.~Magenes.
\newblock {\em Non-homogeneous boundary value problems and applications. {V}ol.
  {I}}.
\newblock Springer-Verlag, New York-Heidelberg, 1972.

\bibitem{Makr_Noch_a_post_par_06}
C.~Makridakis and R.~H. Nochetto.
\newblock A posteriori error analysis for higher order dissipative methods for
  evolution problems.
\newblock {\em Numer. Math.}, 104:489--514, 2006.

\bibitem{NakNakTor}
P.~B. Nakshatrala, K.~B. Nakshatrala, and D.~A. Tortorelli.
\newblock A time-staggered partitioned coupling algorithm for transient heat
  conduction.
\newblock {\em Internat. J. Numer. Methods Engrg.}, 78(12):1387--1406, 2009.

\bibitem{PenYot}
G.~Pencheva and I.~Yotov.
\newblock Balancing domain decomposition for mortar mixed finite element
  methods.
\newblock {\em Numer. Linear Algebra Appl.}, 10(1-2):159--180, 2003.

\bibitem{RybMag}
I.~Rybak and J.~Magiera.
\newblock A multiple-time-step technique for coupled free flow and porous
  medium systems.
\newblock {\em J. Comput. Phys.}, 272:327--342, 2014.

\bibitem{Scott-Zhang}
L.~R. Scott and S.~Zhang.
\newblock Finite element interpolation of nonsmooth functions satisfying
  boundary conditions.
\newblock {\em Math. Comput.}, 54(190):483--493, 1990.

\bibitem{StarkeFOV:97}
G.~Starke.
\newblock Field-of-values analysis of preconditioned iterative methods for
  nonsymmetric elliptic problems.
\newblock {\em Numer. Math.}, 78(1):103--117, 1997.

\bibitem{Thomee}
V.~Thom\'ee.
\newblock {\em Galerkin finite element methods for parabolic problems},
  volume~25 of {\em Springer Series in Computational Mathematics}.
\newblock Springer-Verlag, Berlin, second edition, 2006.

\bibitem{Yu-space-time}
H.~Yu.
\newblock A local space-time adaptive scheme in solving two-dimensional
  parabolic problems based on domain decomposition methods.
\newblock {\em SIAM J. Sci. Comput.}, 23(1):304--322, 2001.

\end{thebibliography}
\end{document}